\documentclass[11pt]{amsart}
\usepackage[a4paper, margin=3cm]{geometry} 
\usepackage{amsmath,amssymb,amscd,amsthm,amsxtra}
\usepackage{dsfont}
\usepackage{graphicx, mathrsfs}
\usepackage{amssymb,amsmath,amsthm}
\usepackage{mathrsfs,dsfont}
\usepackage{fancyhdr}
\usepackage{enumerate}
\usepackage{hyperref}
\usepackage[toc,page]{appendix}
\usepackage{color}
\usepackage{caption}
\usepackage{subcaption}

\usepackage{mathtools}
\usepackage{float}
\usepackage{tikz}
\usetikzlibrary{patterns}
\usepackage{quiver}
\newtheorem{thm}{Theorem}[section]
\newtheorem{prop}[thm]{Proposition}
\newtheorem{lem}[thm]{Lemma}

\theoremstyle{definition}
\newtheorem{definition}[thm]{Definition}

\newtheorem{assumption}[thm]{Assumption}
\theoremstyle{remark}

\newcommand{\bfK}{{\bf K}}

\newcommand{\bfy}{{\bf y}}

\newtheorem{remark}[thm]{Remark}
\numberwithin{equation}{section}

\begin{document}
	
	
	\title{
		Control and optimization   for Neural Partial Differential Equations in Supervised Learning}
	
	\author[A. Bensoussan]{Alain Bensoussan }
	\address{Naveen Jindal School of Management, The University of Texas at Dallas, Richardson, TX 75080, USA}
	\email{alain.bensoussan@utdallas.edu} 
	\thanks{A. B is  funded in part by  the  NSF Grant DMS-2204795.}

	\author[M.-B. Tran]{Minh-Binh Tran}
	\address{Department of Mathematics, Texas A\&M University, College Station, TX 77843, USA}
	\email{minhbinh@tamu.edu} 
	\thanks{B. W and M.-B. T are  funded in part by  the  NSF Grants DMS-2204795, DMS-2305523,    Humboldt Fellowship,   NSF CAREER  DMS-2303146, DMS-2306379.}
	
	\author[B. Wang]{Bangjie Wang}
	\address{Department of Mathematics, Texas A\&M University, College Station, TX 77843, USA}
	\email{bangjiewang@tamu.edu}

	\begin{abstract}
		
		Although there is a substantial body of literature on control and optimization problems for parabolic and hyperbolic systems, the specific problem of controlling and optimizing the \emph{coefficients} of the associated operators within such systems has not yet been thoroughly explored. In this work, we aim to initiate a line of research in control theory focused on optimizing and controlling the coefficients of these operators, a problem that naturally arises in the context of neural networks and supervised learning.
		
		In supervised learning, the primary objective is to transport initial data toward target data through the layers of a neural network. We propose a novel perspective: neural networks can be interpreted as partial differential equations (PDEs). From this viewpoint, the control problem traditionally studied in the context of ordinary differential equations (ODEs) is reformulated as a control problem for PDEs, specifically targeting the optimization and control of coefficients in parabolic and hyperbolic operators. To the best of our knowledge, this specific problem has not yet been systematically addressed in the control theory of PDEs.
		
		To this end, we propose a dual system formulation for the control and optimization problem associated with parabolic PDEs, laying the groundwork for the development of efficient numerical schemes in future research. We also provide a theoretical proof showing that the control and optimization problem for parabolic PDEs admits minimizers. Finally, we investigate the control problem associated with hyperbolic PDEs and prove the existence of solutions for a corresponding approximated control problem.
		
	\end{abstract}
	
	\maketitle
	\tableofcontents
	\section{Introduction}
	Neural networks are a powerful tool in machine learning, designed to mimic the structure of the human brain to process patterns and data. These networks consist of layers of interconnected neurons that collaborate to handle complex information \cite{chizat2020implicit,goodfellow2016deep,hochreiter1997long,krizhevsky2012imagenet,lecun2015deep,rumelhart1986learning,vaswani2017attention}. With their ability to manage intricate data, neural networks are widely used in numerous fields, such as finance, robotics, healthcare, and intelligent decision-making systems.

	Deep neural networks process data through multiple layers, each applying nonlinear functions and affine transformations. In convolutional neural networks (CNNs), these transformations are implemented using convolution operators with compact filters~\cite{lecun1995convolutional}. CNNs excel in large-scale tasks involving discrete signals, such as images, audio, and video~\cite{krizhevsky2012imagenet, lecun1995convolutional, lecun2010convolutional}. In supervised learning, filters and model parameters (weights) are learned from training data, leveraging local spatial relationships to enhance computational efficiency.
	
	However, CNNs face challenges in architecture design, involving decisions about network depth, width, and layer connections. Deeper networks are often preferred for better generalization, but designing them requires careful selection of transformations and nonlinearities. Treating these as hyperparameters and optimizing them jointly with weights is effective but computationally costly.
	
	To address these challenges, Ruthotto and Haber~\cite{ruthotto2020deep} propose interpreting CNN architectures as discretizations of partial differential equations (PDEs), where each layer corresponds to a time step in a PDE solver. This perspective offers a principled and mathematically grounded framework for network design.
	
	In addition to architectural complexity, CNNs face issues of interpretability and robustness, especially in safety-critical applications such as autonomous driving. The PDE-based framework facilitates a better understanding of network behavior and supports the design of stable architectures that are resilient to small perturbations in the input.
	
	This approach has shown significant potential in improving CNN design and addressing the aforementioned challenges. For instance, recent studies~\cite{pei2017deepxplore} have demonstrated that deep network predictions can be highly sensitive to minor changes in input images. These results underscore the importance of designing stable models—i.e., networks whose outputs remain robust under small input variations. This stability aligns with fundamental principles of PDE analysis, reinforcing the value of the PDE-inspired method  in CNN architecture design. The approach has inspired parabolic network designs that smooth and denoise images, as well as networks based on hyperbolic equations and semi-implicit architectures \cite{haber2019imexnet, lensink2022fully}.

	The task of supervised learning is to map initial data to target data, which can be viewed as a control problem. In a lecture at Collège de France~\cite{lions}, P.L. Lions proposed extending supervised learning to continuous time using partial differential equations (PDEs).
	
	In this work, we bring P.-L.~Lions' idea to fruition by generalizing the approach of Ruthotto and Haber~\cite{ruthotto2020deep}, which interprets convolutional neural networks (CNNs) as partial differential equations (PDEs), and by placing this interpretation on a more rigorous mathematical foundation. We establish a novel connection between supervised learning---framed as mapping initial data to target data---and the control and optimization of coefficients in associated parabolic and hyperbolic PDEs, a class of problems largely unexplored in the existing literature~\cite{lions1971optimal}.
	
	Standard neural networks benefit from several well-established universal approximation theorems, making supervised learning feasible from a control-theoretic perspective~\cite{hernandez2024deep}. However, for CNNs designed through the lens of PDEs, such theoretical foundations are currently lacking despite their many practical advantages. The goal of this work is to address this gap and contribute to the development of a corresponding theoretical framework. In particular, second-order partial differential operators emerge naturally from this formulation, rendering the associated control problem  tractable.

	We show that our optimization and control problem for parabolic PDEs admits minimizers under suitable conditions (see Section~\ref{Section:Full}).	In the context of mathematical optimization, optimal control, and game theory, dual systems offer alternative formulations that can provide deeper insights and simplify computation. We introduce a dual system corresponding to our control and optimization problem for parabolic PDEs, which appears to be novel (see Section~\ref{Section:Dual}).
	Finally, in Section~\ref{Section:Hyper}, we demonstrate the existence of control parameters for an approximate control problem involving hyperbolic equations.

	{\it The current work aims to establish the first basic theory for the control and optimization PDE problem arising from the theory of CNNs. }
	
	We finally remark that Neural Ordinary Differential Equations (Neural ODEs) represent a class of deep learning models that use differential equations to describe the dynamics of neural networks. By employing an ordinary differential equation parameterized by a neural network, Neural ODEs can capture the continuous evolution of hidden states within the network~\cite{chen2018neural1, dupont2019augmented, grathwohl2019ffjord, kidger2020neural, massaroli2020dissecting, rackauckas2020universal, zhang2020anodev2}.  The problem of controlling Neural ODEs has also been explored recently~\cite{doi:10.1137/21M1411433}.

	{\bf Acknowledgment:} The authors would like to thank Prof. P.-L.Lions, Prof. E. Pozzoli, Prof. E. Trelat, Prof. E. Zuazua and Mr. 
	A. A. Lopez for fruitful discussions on the work.
	\section{Neural Partial Differential Equations}\label{Subsec:gAIControl}
	For any given dataset \((X_0, X_T)\), where \(X_0 \in \mathbb{R}^{N}\) represents the input and \(X_T \in \mathbb{R}^{N}\) the output, \emph{Supervised Learning (SL)} aims to learn a mapping from \(X_0\) to \(X_T\).
	
	We consider an extended version of the CNN layer used in Ruthotto and Haber~\cite{ruthotto2020deep}:
	
	\begin{equation}\label{RHLayer}
		\mathbf{F}(\Theta, X) \ = \ \sum_{i=1}^P \mathbf{K}_{i,2}(\Theta_{i,3})\, \sigma_{i}\left( \mathcal{N}(\mathbf{K}_{i,1}(\Theta_{i,1})X, \Theta_{i,2}) \right),
	\end{equation}	
	where we generalize the setting of \cite[Equation (1)]{ruthotto2020deep} from \(P = 1\) to \(P = 1\) and \(P = 2\). Since this work aims to lay the foundation for the theory of PDE-based CNNs, we consider such general setting. The special case \( P = 1 \), as studied in~\cite{ruthotto2020deep}, is also covered by our theorems. The parameters are defined as follows:
	
	\begin{itemize}
		\item We denote $\mathbf{K}=(\mathbf{K}_1,\cdots,\mathbf{K}_P)$. The matrices \(\mathbf{K}_i\) are partitioned into two components:  \(\mathbf{K}_i=( \mathbf{K}_{i,1}\), \(\mathbf{K}_{i,2})\), where the matrices \(\mathbf{K}_{i,1} \in \mathbb{R}^{N \times N}\) and \(\mathbf{K}_{i,2} \in \mathbb{R}^{N \times N}\) represent convolution operators. The parameter \(N\) denotes the width of the layer, i.e., the number of input, intermediate, and output features.
		
		\item The activation functions \(\sigma_i\) are vector-valued functions \(\sigma_i : \mathbb{R}^N \to \mathbb{R}^N\), defined as follows. Let \(\tilde{\sigma}_i : \mathbb{R} \to \mathbb{R}\) be a real-valued function. Then, we define
		\[
		\sigma_i(x) = \begin{pmatrix}
			\tilde{\sigma}_i(x_1) \\
			\tilde{\sigma}_i(x_2) \\
			\vdots \\
			\tilde{\sigma}_i(x_N)
		\end{pmatrix}.
		\]

		\item We denote $\Theta=(\Theta_1,\cdots,\Theta_P)$. The parameter vectors \(\Theta_i\) are partitioned into three components:  \(\Theta_i=( \Theta_{i,1}\), \(\Theta_{i,2},\Theta_{i,3})\). The choices of \(\Theta_i\) will be explained clearer below in Subsection \ref{Subsect:Para}.
		\item  \(\mathcal{N}\) denotes the normalization layer. Following \cite{ruthotto2020deep}, we choose \(\mathcal{N}\) to be the identity operator in this paper, i.e.,
		\[
		\mathcal{N}(\mathbf{K}_{i,1}(\Theta_{i,1})X, \Theta_{i,2}) = \mathbf{K}_{i,1}(\Theta_{i,1})X.
		\]
		Therefore, the parameters \(\Theta_{i,2}\) plays no role and can be omitted.
		.
		
	\end{itemize}
	
	\begin{remark}
		\leavevmode
		\begin{itemize}

			\item The inputs \(X_0\) and outputs \(X_T\) can also be matrices in \(\mathbb{R}^{N \times S}\), where \(S\) denotes the number of samples. In this case, $\mathbf{F}(\Theta,\cdot)$ acts on $X$ componentwise. In this work, we set \(S = 1\) for simplicity.
			
			\item One could consider different dimensions \(N_1\), \(N_2\), and \(N_3\) for the input, intermediate, and output features, respectively. However, for clarity, we assume \(N_1 = N_2 = N_3 = N\).
			
		\end{itemize}
	\end{remark}
	
	%
	%

	Given an input feature \(X_0\), a ResNet unit~\cite{hinton2012deep, lecun2015deep} with \(L\) layers produces a transformed output \(X_L = X_T\) according to the following update rule:
	\begin{equation}\label{RHLayer2}
		\frac{X_{j+1} - X_j}{\Delta t} = \mathbf{F}(\Theta, X_j), \quad j = 0, \dots, L - 1,
	\end{equation}
	where \(\Delta t > 0\) is a constant step size.
	
	Equation~\eqref{RHLayer2} has the following continuous counterpart:
	\begin{equation}\label{NeuralDE}
		\dot{X} = \mathbf{F}(\Theta, X).
	\end{equation}
	
	The supervised learning (SL) task can now be reformulated as a control problem: given \(\{{\sigma}_i\}_{i=1}^P\), \(X_0, X_T \in \mathbb{R}^N\), and a fixed terminal time \(T > 0\), find parameters \(\mathbf{K}\) and \(\Theta\) such that
	\[
	X(0) = X_0, \quad X(T) = X_T,
	\]
	where \(X\) solves Equation~\eqref{NeuralDE}.
	
	\subsection{From Neural Networks to Parabolic Partial Differential Equations}\label{Subsect:Para}
	
	In \cite{ruthotto2020deep}, Ruthotto and Haber proposed interpreting Neural Networks as discrete parabolic operators, leading to a way of using PDEs to model Neural Networks. However, the coefficients of these parabolic operators are constant. Below, we generalize this idea by allowing the coefficients to depend on both time \( t \) and space \( x \) (distributed control), and we place it within a more rigorous mathematical framework, thereby bringing Lions' idea~\cite{lions} to fruition.

	For any vector $\bfy \in \mathbb{R}^N$, we can associate $\bfy$ with a function $y: [-L,L]\to\mathbb{R}$ for some constant $L>0$, such that  
	\begin{equation}\label{Function}
		\bfy = [y(z_1),\ldots,y(z_N)]^{\top}, \quad \text{where} \quad z_i = i h, \quad h = \frac{L}{N'}, \quad i=-N',\ldots,N'.
	\end{equation} 
	where we assume $N=2N'+1$ with $N'\in\mathbb{N}$.

	Next, given parameters $\beta,\alpha,\theta\in\mathbb{R}^{N\times N}$, we define:

	\[
	\bfK_\Delta = \frac{\beta}{h^2} 
	\begin{bmatrix} 
		-2 & 1 & 0 & \cdots & 0 \\ 
		1 & -2 & 1 & \cdots & 0 \\ 
		0 & 1 & -2 & \cdots & 0 \\ 
		\vdots & \vdots & \vdots & \ddots & 1 \\ 
		0 & 0 & 0 & 1 & -2 
	\end{bmatrix} \ - \ \frac{\alpha}{2h} 
	\begin{bmatrix} 
		0 & 1 & 0 & \cdots & 0 \\ 
		-1 & 0 & 1 & \cdots & 0 \\ 
		0 & -1 & 0 & \cdots & 0 \\ 
		\vdots & \vdots & \vdots & \ddots & 1 \\ 
		0 & 0 & 0 & -1 & 0 
	\end{bmatrix} \ - \ \theta\mathrm{Id}_{\mathbb{R}^N\times\mathbb{R}^N},
	\]
	where $\mathrm{Id}_{\mathbb{R}^{N\times N}}$ is the identity matrix in $\mathbb{R}^{N\times N}$, and
	\[
	\beta= \begin{bmatrix}
		\beta(z_1,t) & \cdots & 0 \\
		\vdots &\ddots & \vdots \\
		0 &\cdots &\beta(z_N,t)
	\end{bmatrix},\alpha= \begin{bmatrix}
		\alpha(z_1,t) & \cdots & 0 \\
		\vdots &\ddots & \vdots \\
		0 &\cdots &\alpha(z_N,t)
	\end{bmatrix}, \theta= \begin{bmatrix}
		\theta(z_1,t) & \cdots & 0 \\
		\vdots &\ddots & \vdots \\
		0 &\cdots &\theta(z_N,t).
	\end{bmatrix}
	\]	
	For sufficiently small $h$, we obtain for $i=1,2,\ldots,N$:
	\begin{equation}
		\begin{aligned}\label{discrete1}
			\bfK_\Delta\bfy_i &= \frac{\beta(z_i,t)(y(z_{i-1}) - 2y(z_i) + y(z_{i+1}))}{h^2} - \frac{\alpha(z_i,t)(y(z_{i+1}) - y(z_{i-1}))}{2h} -\theta(z_i,t) y(z_i)   \
			,
	\end{aligned}\end{equation}
	where $y(z_{-1})=y(z_{N+1})=0$ and $\bfK_\Delta\bfy_i $ is the $i$-th component of the vector $\bfK_\Delta\bfy$, yielding:
	\begin{equation}
		\begin{aligned}\label{discrete2}
			\bfK_\Delta\bfy_i    \approx \beta(z_i,t)  \partial_{x}^2 y(z_i) - \alpha(z_i,t)\partial_x y(z_i) - \theta(z_i,t) y(z_i)  \text{ as } h\to 0.
	\end{aligned}\end{equation}

	We use \(\mathbf{K}_\Delta\) in our CNN by choosing \(P = 2\). We set \(\Theta_{1,3} = (\alpha, \beta, \theta)\). In Equation~\eqref{NeuralDE}, we define \(\mathbf{K}_{1,2}(\Theta_{1,3})\) as the matrix \(\mathbf{K}_\Delta\) and choose \({\sigma}_1\) to be the identity activation function. We also set \(\mathbf{K}_{1,1}(\Theta_{1,1}) = \mathrm{Id}_{\mathbb{R}^{N \times N}}\) and omit the role of \(\Theta_{1,1}\).
	
	We assume that both \(\mathbf{K}_{2,1}(\Theta_{2,1})\) and \(\mathbf{K}_{2,2}(\Theta_{2,3})\) are diagonal matrices. 
	Notably, we can even set \({\sigma}_2 = 0\), as done in \cite{ruthotto2020deep} and hence $P=1$. Substituting into Equation~\eqref{NeuralDE}, we obtain the following:
	
	\begin{equation}\label{NeuralDE2}
		\dot{X}(t) = \mathbf{K}_\Delta X(t) + \mathbf{K}_{2,2}(\Theta_{2,3})\, {\sigma}_2(\mathbf{K}_{2,1}(\Theta_{2,1}) X(t)).
	\end{equation}

	Similar to \eqref{Function}, we associate \(X\) with a function \(m\), which depends on time \(t \in \mathbb{R}_+\) and space \(x \in \mathbb{R}\). Then, from \eqref{discrete2}, we approximate:
	
	\[
	\mathbf{K}_\Delta X_i(t) \approx \beta(z_i,t)\partial_x^2 m(z_i,t) -\alpha(z_i,t)\partial_x m(z_i,t)-\theta(z_i,t) m(z_i,t) .
	\]
	
	The Neural Differential Equation \eqref{NeuralDE2} can now be approximated by the following PDE, as $h\to 0$
	
	\begin{equation}\label{NeuralPDE}
		\partial_t m(x,t) = \beta(x,t)\, \partial_x^2 m(x,t) - \alpha(x,t)\, \partial_x m(x,t) -\theta(x,t) m(x,t)  + \sigma(\gamma(x,t)\, m(x,t) + \delta(x,t)),
	\end{equation}
	where the term \(\sigma(\gamma m + \delta)\) represents the nonlinearity arising from \(\mathbf{K}_{2,2}(\Theta_{2,3})\, {\sigma}_2(\mathbf{K}_{2,1}(\Theta_{2,1}) X)\). 
	
	Letting \(L \to \infty\), so that the domain \(\Omega = [-L, L]\) extends to \(\mathbb{R}\), we deduce that equation \eqref{NeuralPDE} takes the form of a PDE defined on the whole real line.

	Next, we extend this setting for $x\in\mathbb{R}^2$. To this end, for any vector $\bfy \in \mathbb{R}^N$, we associate $\bfy$ with a function $y: [-L,L]^2\to\mathbb{R}$, for some constant $L>0$, such that:
	\begin{equation*}\begin{aligned}
			& \bfy = [y(z_1,z_1'),\ldots,y(z_N,z_N')]^{\top}\\
			& \quad \text{ with } r_i= (i-1) \mod 2N'+1,\\ &\quad z_i = r_ih-L, \quad z_i' = L- \left\lfloor \frac{i-1}{2N'+1} \right\rfloor h,\quad h = \frac{L}{N'},\quad i=1,2,\ldots,N,
	\end{aligned}\end{equation*} 
	where we assume $ N= (2N'+1)^2$, with $N'\in\mathbb{N}$.  To better represent the two-dimensional structure, we can rewrite the vector $\mathbf{y}$ in matrix form as
	\[
	\mathbf{y} = \big(y(\tilde{z}_i,\tilde{z}_j)\big)_{i,j=1,2,\ldots,2N'+1},
	\]
	which reflects the natural spatial ordering of $\mathbf{y}$ on the domain $[-L,L]^2$.
	
	In what follows, for improved geometric interpretability, we will use a double index $ij$ to label the entries of a vector in $\mathbb{R}^N$, as we did for the vector $\mathbf{y}$.
	
	Let $\theta(\tilde{z}_i,\tilde{z}_j,t) \in \mathbb{R}$, 
	\[
	\big(\alpha_k(\tilde{z}_i,\tilde{z}_j,t)\big)_{k=1}^{2} \in \mathbb{R}^2, \quad 
	\text{and} \quad 
	\big(\beta_{k,l}(\tilde{z}_i,\tilde{z}_j,t)\big)_{k,l=1}^{2} \in \mathbb{R}^{2 \times 2}
	\]
	be a symmetric matrix. Then, there exists a matrix $\mathbf{K}_\Delta \in \mathbb{R}^{N \times N}$ such that the two-dimensional version of \eqref{discrete1} holds.
	
	\begin{equation}
		\begin{aligned}\label{discrete3}
			\bfK_\Delta\bfy_{ij}&\ =   2\frac{\beta_{1,2}(\tilde{z}_i,\tilde{z}_{j},t)(y(\tilde{z}_{i+1},\tilde{z}_{j+1})-y(\tilde{z}_{i+1},\tilde{z}_{j-1})-y(\tilde{z}_{i-1},\tilde{z}_{j+1})+y(\tilde{z}_{i-1},\tilde{z}_{j-1}))}{4h^2} \\
			& \ +\  \frac{\beta_{1,1}(\tilde{z}_i,\tilde{z}_{j},t)(y(\tilde{z}_{i+1},\tilde{z}_{j})-2y(\tilde{z}_{i},\tilde{z}_{j})+y(\tilde{z}_{i-1},\tilde{z}_{j}))}{h^2}  \\
			& \ +\  \frac{\beta_{2,2}(\tilde{z}_i,\tilde{z}_{j},t)(y(\tilde{z}_{i},\tilde{z}_{j+1})-2y(\tilde{z}_{i},\tilde{z}_{j})+y(\tilde{z}_{i},\tilde{z}_{j-1}))}{h^2}  \\ 
			& \ -\ \frac{\alpha_1(\tilde{z}_i,\tilde{z}_{j},t)(y(\tilde{z}_{i+1},\tilde{z}_{j})-y(\tilde{z}_{i-1},\tilde{z}_{j}))}{2h} \\
			& \ -\  \frac{\alpha_2(\tilde{z}_i,\tilde{z}_{j},t)(y(\tilde{z}_{i},\tilde{z}_{j+1})-y(\tilde{z}_{i},\tilde{z}_{j-1}))}{2h}\ - \	 \theta(\tilde{z}_i,\tilde{z}_{j},t) y(\tilde{z}_i,\tilde{z}_{j})  \\
			\approx \ & 2\beta_{1,2}(\tilde{z}_i,\tilde{z}_{j},t)\partial_{\tilde{z}_i}\partial_{\tilde{z}_j}y(\tilde{z}_i,\tilde{z}_{j})  + \beta_{1,1}(\tilde{z}_i,\tilde{z}_{j},t)\partial_{\tilde{z}_i}\partial_{\tilde{z}_i}y(\tilde{z}_i,\tilde{z}_{j}) + \beta_{2,2}(\tilde{z}_i,\tilde{z}_{j},t)\partial_{\tilde{z}_j}\partial_{\tilde{z}_j}y(\tilde{z}_i,\tilde{z}_{j})
			\\
			&  - \alpha_1(\tilde{z}_i,\tilde{z}_{j},t)\partial_{\tilde{z}_i}y(\tilde{z}_i,\tilde{z}_{j}) -\alpha_2(\tilde{z}_i,\tilde{z}_{j},t)\partial_{\tilde{z}_j}y(\tilde{z}_i,\tilde{z}_{j}) -\theta(\tilde{z}_i,\tilde{z}_{j},t) y(\tilde{z}_i,\tilde{z}_{j}),\end{aligned}
	\end{equation}
	in which $y(z_{-1},\cdot)=y(z_{N'+1},\cdot)=y(\cdot,z_{-1})=y(\cdot,z_{N'+1})=0$. This provides a concrete construction of the matrix $\mathbf{K}_\Delta$ in the two-dimensional setting.

	
	The following two-dimensional version of \eqref{NeuralPDE} can be obtained for \(x = (x_1, x_2)\):
	\begin{equation}\label{NeuralPDE1}
		\begin{aligned}
			\partial_t m(x,t) = \ & \sum_{k,l=1}^2 \beta_{k,l}(x,t) \partial_{x_k} \partial_{x_l} m(x,t) - \sum_{k=1}^2 \alpha_k(x,t) \partial_{x_k} m(x,t) \\
			&  - \theta(x,t) m(x,t) + \sigma\big(\gamma(x,t) m(x,t) + \delta(x,t)\big),
		\end{aligned}
	\end{equation}
	which, with a slight abuse of notation, can be rewritten as
	\begin{equation}\label{NeuralPDE2}
		\begin{aligned}
			\partial_t m(x,t) = \ & \sum_{k,l=1}^2 \partial_{x_k} \left(\beta_{k,l}(x,t) \partial_{x_l} m(x,t)\right) - \sum_{k=1}^2 \alpha_k(x,t) \partial_{x_k} m(x,t) \\
			& - \theta(x,t) m(x,t)  + \sigma\big(\gamma(x,t) m(x,t) + \delta(x,t)\big).
		\end{aligned}
	\end{equation}

	This process can be easily generalized to the case where \(x\) is a vector in \(\mathbb{R}^d\), resulting in an \(d\)-dimensional parabolic system:
	\begin{equation}\label{Para1}
		\begin{aligned}
			& \frac{\partial m}{\partial t}(x,t) 
			- \sum_{k,l=1}^d \frac{\partial}{\partial x_k} \left( \beta_{kl}(x,t) \frac{\partial}{\partial x_l} m(x,t) \right)
			+ \sum_{k=1}^d \alpha_k(x,t) \frac{\partial}{\partial x_k} m(x,t)+ \theta(x,t) m(x,t) \\
			& = \sigma\big(\gamma(x,t) m(x,t) + \delta(x,t)\big),
			\\
			& \text{with initial condition } m(x,0) = m_0(x).
		\end{aligned}
	\end{equation}
	
	In this system, \(\sigma\) represents the nonlinearity, and \(\gamma\) and \(\delta\) are functions from \(\mathbb{R}^d \times \mathbb{R}_+\) to \(\mathbb{R}\).
	
	The task of Supervised Learning (SL) now becomes the problem of to transport the initial data $m_0$ associated to $X_0$ towards the target data $m_T$ associated to $X_T$, which can be rewritten as an   optimization and control problem
	
	\begin{equation}\label{Para2}
		\min_{\beta,\alpha,\gamma,\theta}\left\|m(\cdot,T)-m_T(\cdot)\right\|^2,
	\end{equation}
	for some suitable norm $\|\cdot\|$. 
	The optimization is performed over \(\beta, \alpha, \gamma,\theta\), which are selected from suitable function spaces.

	We aim to minimize over $\alpha,\beta,\gamma,\theta$, while keeping $\delta$ fixed. 
	\begin{remark}
		If \( S > 1 \), meaning that we have multiple samples, we consider \( S \) instances of the PDE in \eqref{Para1}, each associated with different initial data \( \{m_0^i\}_{i=1}^S \) and corresponding target data \( \{m_T^i\}_{i=1}^S \). Accordingly, the optimization and control problem becomes
		\[
		\min_{\beta,\alpha,\gamma,\theta} \sum_{i=1}^S \left\| m^i(\cdot, T) - m_T^i(\cdot) \right\|^2,
		\]
		where each \( m^i \) solves \eqref{Para1} with its own initial condition \( m_0^i \).
		
	\end{remark}

	\textit{The system \eqref{Para1}-\eqref{Para2} represents a control and optimization problem over a parabolic system. While there has been extensive research on control and optimization problems for parabolic equations and systems (see, for instance, \cite{lions1971optimal,zuazua2007controllability}), to the best of our knowledge, this is the first time that the minimization problem over the coefficients of the parabolic operators has been considered. We believe that minimizing over the coefficients of the parabolic operators is a rich and promising problem.}
	
	A natural question is whether or not the optimization problem \eqref{Para1}-\eqref{Para2} has a solution. In the next section, we will provide an affirmative answer to this question, which is equivalent to showing that the problem \eqref{Para1}-\eqref{Para2} possesses a solution. A dual system of the optimization problem will also be given.

	\subsection{From Neural Networks to Hyperbolic Partial Differential Equations}\label{Subsect:Hyper}
	One advantageous property of hyperbolic equations is their reversibility, which eliminates the need to store intermediate network states and leads to improved memory efficiency. This feature of hyperbolic equations is especially crucial for very deep networks, where memory constraints can impede training. In~\cite{ruthotto2020deep}, the authors also propose using neural networks inspired by wave-type equations. 
	
	Below, we establish a more solid mathematical foundation for this framework.
	One approach to designing neural networks is to construct the leapfrog network \cite{chang2018reversible, gomez2017reversible,ruiz2022interpolation}:
	
	\begin{equation}
		X_{l+1} - 2X_{l} + X_{l-1} = (\Delta t)^2 \sum_{i=1}^P \mathbf{K}_{i,2}(\Theta_{i,3})\, \sigma_i\left( \mathcal{N}(\mathbf{K}_{i,1}(\Theta_{i,1})X_l, \Theta_{i,2}) \right),
	\end{equation}
	which can be rewritten in continuous form as a second-order neural differential equation:
	\begin{equation}\label{NeuralDEleapfrog}
		\ddot{X}(t) = \sum_{i=1}^P \mathbf{K}_{i,2}(\Theta_{i,3})\, \sigma_i\left( \mathcal{N}(\mathbf{K}_{i,1}(\Theta_{i,1})X(t), \Theta_{i,2}) \right),
	\end{equation}	
	with the initial conditions
	\[X(0)=X_0,\dot{X}(0)\approx\frac{X_1-X_0}{\Delta t}.\]
	
	Using the same argument as in Subsection~\ref{Subsect:Para}, we obtain
	\begin{equation}\label{Hyper1} 
		\begin{aligned}
			\frac{\partial^2 m}{\partial t^2}(x,t) 
			&- \sum_{k,l=1}^d \frac{\partial}{\partial x_k} \left( \beta_{kl}(x,t) \frac{\partial}{\partial x_l}m(x,t) \right) 
			+ \sum_{k=1}^d \alpha_k(x,t) \frac{\partial}{\partial x_k} m(x,t)  
			+ \theta(x,t) m(x,t) \\
			&= \sigma\left( \gamma(x,t) m(x,t) + \delta(x,t) \right),
		\end{aligned}
	\end{equation}
	with the initial conditions
	\[
	m(x,0) = m_0(x), \qquad \partial_t m(x,0) = \tilde{m}_0(x).
	\]

	In this system, $\sigma$ is the nonlinearity, and $\gamma$ and $\delta$ are functions from $\mathbb{R}^d\times\mathbb{R}_+$ to $\mathbb{R}$. Note that we can even set $\sigma = 0$, as was done in \cite{ruthotto2020deep}. 
	
	%

The task of Supervised Learning (SL) is now to transport the initial data $(m_0,\tilde{m}_0)$ towards the target data $(m_T,\tilde{m}_T)$. This task is essentially an optimization and control problem, which reads as follows:

\begin{equation}\label{Hyper3}
\min_{\beta,\alpha,\gamma,\theta}\left\|(m(\cdot,T), \partial_t{m}(\cdot,T))-(m_T(\cdot),\tilde{m}_T(\cdot))\right\|^2.
\end{equation}

We aim to minimize over $\alpha,\beta,\gamma,\theta$, while keeping $\delta$ fixed.  The optimization is performed over \(\beta, \alpha, \gamma,\theta\), which are selected from suitable function spaces.

\textit{The system \eqref{Hyper1}-\eqref{Hyper3} represents a control and optimization problem over a hyperbolic system.}

A natural question is whether the optimization problem \eqref{Hyper1}--\eqref{Hyper3} admits a solution. In Section \ref{Section:Hyper}, we will provide a partial answer to this question. More precisely, we will prove that, given \(T\), initial states \((m_0(\cdot), \tilde{m}_0(\cdot))\), final states \((m_T(\cdot), \tilde{m}_T(\cdot))\), and any small constant \(\varepsilon > 0\), it is possible to construct parameters for the equation such that
\begin{equation}\label{Hyper2}
\left\| (m(\cdot, T),  \partial_t{m}(\cdot, T)) - (m_T(\cdot), \tilde{m}_T(\cdot)) \right\|^2 < \varepsilon,
\end{equation}
which establishes an \textit{approximate controllability} result.

\section{Control and Optimization of Parabolic Equations}\label{Section:Parabolic}

\subsection{Derivation of the dual system}
\label{Section:Dual}
As discussed above, \eqref{Para1}-\eqref{Para2} is a control and optimization problem over a parabolic system. In this subsection, we will show that the optimization problem \eqref{Para1}-\eqref{Para2} can be written as a dual system, which is easier to handle in designing numerical algorithm. As discussed above,  even though there have been quite a lot of works on the control and optimization problem for parabolic equations and systems,  the minimization problem  over the coefficients of the parabolic operators is considered and this is the first dual system written for such problem.

We consider four controls:
\[
\beta(x,t) = \big(\beta_{k,l}(x,t)\big) \in \mathbb{R}^{d \times d}, \quad
\alpha(x,t) = \big(\alpha_k(x,t)\big) \in \mathbb{R}^d, \quad
\gamma(x,t) \in \mathbb{R}, \quad
\theta(x,t) \in \mathbb{R}.
\]

In this subsection, we impose the following uniform boundedness and ellipticity conditions:
\begin{equation}\label{cons:dual_system}
\begin{cases}
	|\alpha(x,t)| \leq A, \\
	b I \leq \beta(x,t) \leq B I, \\
	|\gamma(x,t)| \leq C, \\
	|\theta(x,t)| \leq D,
\end{cases}
\end{equation}
where \(A, b, B, C, D > 0\) are given constants, \(I\) is the \(d \times d\) identity matrix, and the inequality on \(\beta\) is understood in the sense of symmetric matrices (i.e., \(b I \leq \beta(x,t) \leq B I\) means \(\beta\) is uniformly elliptic and bounded).

Additionally, we assume
\[
\sigma \in C^1(\mathbb{R}).
\]

In this section, we set \(\| \cdot\|\) for \eqref{Para2} as follows
\[ \| \cdot\| = \| \cdot\|_{L^2(\mathbb{R}^d)}.\]

If \((\hat{\alpha}, \hat{\beta}, \hat{\gamma}, \hat{\theta})\) is an optimal control, then the corresponding optimal state \(m(x,t)\) satisfies the PDE
\begin{equation}
\begin{cases}
	\displaystyle
	\frac{\partial m}{\partial t}(x,t) 
	- \sum_{k,l=1}^d \frac{\partial}{\partial x_k} \left( \hat{\beta}_{kl}(x,t) \frac{\partial }{\partial x_l}m(x,t) \right)
	+ \sum_{k=1}^d \hat{\alpha}_k(x,t) \frac{\partial}{\partial x_k} m(x,t)
	+ \hat{\theta}(x,t) m(x,t) \\
	\qquad = \sigma \big( \hat{\gamma}(x,t) m(x,t) + \delta(x,t) \big), \\
	m(x,0) = m_0(x).
\end{cases}
\end{equation}

Next, we consider another admissible control defined by
\begin{equation}
\begin{cases}
	\hat{\beta} + \varepsilon (\beta - \hat{\beta}), \\
	\hat{\alpha} + \varepsilon (\alpha - \hat{\alpha}), \\
	\hat{\gamma} + \varepsilon (\gamma - \hat{\gamma}), \\
	\hat{\theta} + \varepsilon (\theta - \hat{\theta}),
\end{cases}
\quad \text{with } 0 < \varepsilon < 1.
\end{equation}

Denote by \(m_\varepsilon(x,t)\) the corresponding state, which solves
\begin{equation}
\begin{cases}
	\displaystyle
	\frac{\partial m_{\varepsilon}}{\partial t}(x,t)
	- \sum_{k,l=1}^d \frac{\partial}{\partial x_k} \left( \left[ \hat{\beta} + \varepsilon (\beta - \hat{\beta}) \right]_{kl} \frac{\partial}{\partial x_l}m_{\varepsilon}(x,t) \right)
	+ \sum_{k=1}^d \left[ \hat{\alpha} + \varepsilon (\alpha - \hat{\alpha}) \right]_k \frac{\partial}{\partial x_k} m_{\varepsilon}(x,t) \\
	\qquad + \left[ \hat{\theta} + \varepsilon (\theta - \hat{\theta}) \right] m_{\varepsilon}(x,t)
	= \sigma \left( \left[ \hat{\gamma} + \varepsilon (\gamma - \hat{\gamma}) \right] m_{\varepsilon}(x,t) + \delta(x,t) \right), \\
	m_\varepsilon(x,0) = m_0(x).
\end{cases}
\end{equation}

We now introduce the quantity
$$\tilde{m}_{\varepsilon}\left( x,t \right) =\frac{m_{\varepsilon}\left( x,t \right) -m\left( x,t \right)}{\varepsilon}.$$

It then follows from the Taylor expansion that
\begin{align*}
\sigma \left( \left[ \hat{\gamma}+\varepsilon \left( \gamma -\hat{\gamma} \right) \right] m_{\varepsilon}+\delta \right)
&= \sigma \left( \hat{\gamma}m+\delta + \left[ \hat{\gamma}+\varepsilon (\gamma - \hat{\gamma}) \right] m_{\varepsilon} - \hat{\gamma} m \right) \\
&= \sigma \left( \hat{\gamma}m + \delta + \hat{\gamma}(m_{\varepsilon} - m) + \varepsilon (\gamma - \hat{\gamma}) m_{\varepsilon} \right) \\
&= \sigma \left( \hat{\gamma}m + \delta \right)
+ \sigma'(\hat{\gamma} m+\delta) \left( \hat{\gamma}(m_{\varepsilon} - m) + \varepsilon (\gamma - \hat{\gamma}) m_{\varepsilon} \right) + o\left(|h(\varepsilon)|\right),
\end{align*}
where we define
\[
h(\varepsilon) := \hat{\gamma}(m_{\varepsilon} - m) + \varepsilon (\gamma - \hat{\gamma}) m_{\varepsilon},
\]
and the term \( o(|h|) \) is defined as follows:

If \( r(h) = o(|h|) \), then
\[
\lim_{h \to 0} \frac{r(h)}{|h|} = 0.
\]

Note that
\[
h(\varepsilon) = \varepsilon \hat{\gamma} \tilde{m}_\varepsilon + \varepsilon (\gamma - \hat{\gamma}) m_\varepsilon = O(\varepsilon),
\]
where \( O(\varepsilon) \) is defined as:

If \( r(\varepsilon) = O(\varepsilon) \), then there exist constants \(\mathcal C, \delta > 0 \) such that for all \( |\varepsilon| < \delta \),
\[
|r(\varepsilon)| \le \mathcal C|\varepsilon|.
\]

As a result, we find
\begin{align*}
&\lim_{\varepsilon \to 0} \frac{\sigma \left( \left[ \hat{\gamma}+\varepsilon \left( \gamma -\hat{\gamma} \right) \right] m_{\varepsilon}+\delta \right) -\sigma \left( \hat{\gamma}m+\delta \right)}{\varepsilon} \\
&= \lim_{\varepsilon \to 0} \frac{\sigma'(\hat{\gamma}m+\delta)  \left( \hat{\gamma}\left( m_{\varepsilon}-m \right) +\varepsilon \left( \gamma -\hat{\gamma} \right) m_{\varepsilon} \right) + o( |h(\varepsilon)|)}{\varepsilon} \\
&= \lim_{\varepsilon \to 0} \sigma'(\hat{\gamma}m+\delta)  \left( \hat{\gamma}\tilde{m}_\varepsilon + \left( \gamma - \hat{\gamma} \right) m_{\varepsilon} \right) + \frac{o(| h(\varepsilon)|)}{\varepsilon} \\
&= \sigma'(\hat{\gamma}m+\delta)  \left( \hat{\gamma} \tilde{m} + \left( \gamma - \hat{\gamma} \right) m \right),
\end{align*}
where \(\tilde{m} = \lim_{\varepsilon \to 0} \tilde{m}_\varepsilon\).

We have the corresponding system for $\tilde{m}$:
\begin{equation}\label{tildem}
\begin{dcases}
	\displaystyle
	\frac{\partial \tilde{m}}{\partial t}(x,t)
	- \sum_{k,l=1}^d \frac{\partial}{\partial x_k} \left( \hat{\beta}_{kl}(x,t) \frac{\partial}{\partial x_l}\tilde{m}(x,t) \right)
	+ \sum_{k=1}^d \hat{\alpha}_k(x,t) \frac{\partial}{\partial x_k} \tilde{m}(x,t)
	+ \hat{\theta}(x,t) \tilde{m}(x,t) \\
	\qquad = \sum_{k,l=1}^d \frac{\partial}{\partial x_k} \left( \left( \beta_{kl} - \hat{\beta}_{kl} \right)(x,t) \frac{\partial}{\partial x_l} m(x,t) \right)
	- \sum_{k=1}^d \left( \alpha_k - \hat{\alpha}_k \right)(x,t) \frac{\partial}{\partial x_k} m(x,t) \\
	\qquad - (\theta - \hat{\theta})(x,t) m(x,t)
	+ \sigma'(\hat{\gamma}m+\delta) \left( \left( \gamma - \hat{\gamma} \right) m + \hat{\gamma} \tilde{m} \right), \\
	\tilde{m}(x,0)= 0.
\end{dcases}
\end{equation}

On the other hand, we can write
\[
J\left( \hat{\alpha}+\varepsilon (\alpha -\hat{\alpha}),\ \hat{\beta}+\varepsilon (\beta -\hat{\beta}),\ \hat{\gamma}+\varepsilon (\gamma -\hat{\gamma}),\ \hat{\theta}+\varepsilon (\theta -\hat{\theta}) \right) \ge J\left( \hat{\alpha}, \hat{\beta}, \hat{\gamma}, \hat{\theta} \right),
\]
and deduce that
\begin{align*}
&J\left( \hat{\alpha}+\varepsilon (\alpha -\hat{\alpha}),\ \hat{\beta}+\varepsilon (\beta -\hat{\beta}),\ \hat{\gamma}+\varepsilon (\gamma -\hat{\gamma}),\ \hat{\theta}+\varepsilon (\theta -\hat{\theta}) \right) \\
&= \left\| m_{\varepsilon}(\cdot, T) - m_T(\cdot) \right\|^2 \\
&= \left\| m(\cdot, T) + \varepsilon \tilde{m}_{\varepsilon}(\cdot, T) - m_T(\cdot) \right\|^2.
\end{align*}

Then the optimality condition is equivalent to
\[
\left\| m(\cdot ,T)+\varepsilon \tilde{m}_{\varepsilon}(\cdot ,T) -m_{T}(\cdot ) \right\|^2 \ge \left\| m(\cdot ,T)-m_{T}(\cdot ) \right\|^2,
\]
which, in the limit as $\varepsilon \to 0$, yields
\begin{equation}\label{tildeminequ}
\int_{\mathbb{R}^d} \left( m(x,T)-m_{T}(x) \right) \tilde{m}(x,T) \, dx \ge 0.
\end{equation}

As a result, the system composed of \eqref{tildem} and \eqref{tildeminequ} constitutes a necessary condition of optimality for the control problem \eqref{Para1}--\eqref{Para2}.

We then introduce the adjoint function \( u(x,t) \), defined by the equation
\[\begin{cases}
\displaystyle
-\frac{\partial u}{\partial t}(x,t) - \sum_{k,l=1}^d \frac{\partial}{\partial x_l}\left(\hat{\beta}_{kl}(x,t) \frac{\partial}{\partial x_k} u\right) - \sum_{k=1}^d \frac{\partial }{\partial x_k}   \left(  \hat{\alpha}_k(x,t) u \right)    + \hat{\theta}(x,t) u \\
\qquad = \sigma' \left( \hat{\gamma} m + \delta \right) \hat{\gamma} u,
\\
u(x,T) = m(x,T) - m_T(x).
\end{cases}\]

Thus, we have
\[
\begin{aligned}
& \int_{\mathbb{R}^d} \left( m(x,T) - m_T(x) \right) \tilde{m}(x,T) \, dx \\
=& \int_{\mathbb{R}^d} u(x,T) \tilde{m}(x,T) \, dx \\
=& \int_0^T \int_{\mathbb{R}^d} \frac{\partial u}{\partial t}(x,t) \tilde{m}(x,t) \, dx \, dt + \int_0^T \int_{\mathbb{R}^d} u(x,t) \frac{\partial \tilde{m}}{\partial t}(x,t) \, dx \, dt,
\end{aligned}
\]
where
\[
\frac{\partial u}{\partial t}(x,t) \tilde{m}(x,t) = \left[ -\sum_{k,l=1}^d \frac{\partial}{\partial x_l} \left( \hat{\beta}_{kl} \frac{\partial}{\partial x_k} u \right) - \sum_{k=1}^d \frac{\partial }{\partial x_k} \left(  \hat{\alpha}_k u\right)    + \hat{\theta} u - \sigma' \left( \hat{\gamma} m + \delta \right) \hat{\gamma} u \right] \tilde{m}(x,t),
\]
and
\[
\begin{aligned}
u(x,t) \frac{\partial \tilde{m}}{\partial t}(x,t) &= u(x,t) \left[ \sum_{k,l=1}^d \frac{\partial}{\partial x_k} \left( \hat{\beta}_{kl} \frac{\partial}{\partial x_l} \tilde{m}\right) - \sum_{k=1}^d \hat{\alpha}_k \frac{\partial}{\partial x_k} \tilde{m} - \hat{\theta} \tilde{m} \right] \\
& \quad + u(x,t) \sum_{k,l=1}^d \frac{\partial}{\partial x_k} \left( \left( \beta - \hat{\beta} \right)_{kl} \frac{\partial }{\partial x_l}m \right) \\
& \quad - u(x,t) \sum_{k=1}^d \left( \alpha - \hat{\alpha} \right)_k \frac{\partial}{\partial x_k} m - u(x,t) \left( \theta - \hat{\theta} \right) m \\
& \quad + u(x,t) \sigma' \left( \hat{\gamma} m + \delta \right) \left( \left( \gamma - \hat{\gamma} \right) m + \hat{\gamma} \tilde{m} \right).
\end{aligned}
\]

By applying integration by parts, the terms involving \(\tilde{m}\) in the second expression are canceled by those in the first expression.

Next, we have
\begin{equation}
\begin{aligned}
	& \int_{\mathbb{R}^d} \left( m(x,T) - m_T(x) \right) \tilde{m}(x,T) \, dx \\
	=& \int_0^T \int_{\mathbb{R}^d} u(x,t) \Bigg[ \sum_{k,l=1}^d \frac{\partial}{\partial x_k} \left( \left( \beta - \hat{\beta} \right)_{kl} \frac{\partial }{\partial x_l} m \right) - \sum_{k=1}^d \left( \alpha - \hat{\alpha} \right)_k \frac{\partial}{\partial x_k} m  \\
	& \quad - \left( \theta - \hat{\theta} \right) m + \sigma' \left( \hat{\gamma} m + \delta \right) \left( \left( \gamma - \hat{\gamma} \right) m \right) \Bigg] \, dx \, dt \\
	=& \int_0^T \int_{\mathbb{R}^d} \left[ - \nabla u(x,t)^T \left( \beta - \hat{\beta} \right)(x,t) \nabla m(x,t) - u(x,t) \left( \alpha - \hat{\alpha} \right)(x,t)\cdot \nabla m(x,t) \right. \\
	& \quad \left. - u \left( \theta - \hat{\theta} \right) m + u \sigma' \left( \hat{\gamma} m + \delta \right) \left( \left( \gamma - \hat{\gamma} \right) m \right) \right] \, dx \, dt.
\end{aligned}
\end{equation}

Summing all the above equations, the optimality system now reads
\begin{equation}\label{eqn:optimalitysystem}
\begin{dcases}
	\displaystyle
	\frac{\partial m}{\partial t}(x,t) - \sum_{k,l=1}^d \frac{\partial}{\partial x_k}\left(\hat{\beta}_{kl}(x,t) \frac{\partial }{\partial x_l}m\right) + \sum_{k=1}^d \hat{\alpha}_k(x,t) \frac{\partial }{\partial x_k}m + \hat{\theta}(x,t) m = \sigma\left(\hat{\gamma} m + \delta\right), \\
	m(x,0) = m_0(x), \\
	-\frac{\partial u}{\partial t}(x,t) - \sum_{k,l=1}^d \frac{\partial}{\partial x_l}\left(\hat{\beta}_{kl}(x,t) \frac{\partial }{\partial x_k}u\right) - \sum_{k=1}^d \frac{\partial}{\partial x_k}   \left( \hat{\alpha}_k(x,t) u \right) + \hat{\theta}(x,t) u =\sigma'\left(\hat{\gamma} m + \delta\right) \hat{\gamma} u, \\
	u(x,T) = m(x,T) - m_T(x),
\end{dcases}
\end{equation}
and the optimality conditions are
\[
\hat{\beta} \text{ minimizes } -\nabla u(x,t)^T \beta(x,t) \nabla m(x,t),
\]
\[
\hat{\alpha} \text{ minimizes } - u(x,t) \alpha(x,t) \cdot \nabla m(x,t),
\]
\[
\hat{\theta} \text{ minimizes } -u(x,t) \theta(x,t) m(x,t),
\]
\[
u(x,t) \sigma' \left(\hat{\gamma}(x,t) m(x,t) + \delta(x,t)\right) \cdot \left(\left(\gamma(x,t) - \hat{\gamma}(x,t)\right) m(x,t)\right) \geq 0,
\]
under the assumption in \eqref{cons:dual_system}.
\subsection{Existence Theorem}
\label{Section:Full}
In this section, we establish the existence of solutions to the system defined by \eqref{Para1}–\eqref{Para2}. The system is given by
\begin{equation}\label{PDE:unbounded}
\begin{cases}
	\displaystyle
	\frac{\partial m}{\partial t}(x,t) 
	- \sum_{k,l=1}^d \frac{\partial}{\partial x_k} \left( \beta_{kl}(x,t) \frac{\partial }{\partial x_l}m(x,t) \right)
	+ \sum_{k=1}^d \alpha_k(x,t)  \frac{\partial}{\partial x_k} m(x,t)
	+ \theta(x,t) m(x,t) \\
	\qquad = \sigma \left( \gamma(x,t) m(x,t) + \delta(x,t) \right), \quad \text{on }  \mathbb{R}^d\times [0,T], \\
	m(x,0) = m_0(x), \quad \text{on } \mathbb{R}^d,
\end{cases}
\end{equation}
subject to the generalized optimization problem
\begin{equation}\label{control:unbounded}
\inf_{\alpha,\beta,\gamma,\theta} \left\| m(\cdot, T) - m_T(\cdot) \right\|^2,
\end{equation}
where we set $\|\cdot\|$ as follows
\[ \| \cdot\| = \| \cdot\|_{L^2(\mathbb{R}^d)}.\]

The  set of  $\alpha,\beta,\gamma,\theta$ will be specified later in Assumption 
\ref{asp:unbounded_para}.

We first recall the definitions of weighted $L^p$ spaces and weighted Sobolev spaces.

\begin{definition}\label{def:weightedSobolev}
Let $\omega>0$ a.e.~be a measurable weight function on a domain $\Omega\subset\mathbb{R}^d$.
We define the weighted $L^p$ space $L^p_\omega(\Omega)$ for $1\le p<\infty$ as the space of all measurable functions $u:\Omega\rightarrow \mathbb{R}$ such that
$$\left\| u \right\|_{L^p_\omega(\Omega)}:= \left( \int_\Omega |u|^p\omega d\mu \right)^{\frac{1}{p}}<\infty .$$

We define the weighted Sobolev space \( W^{k,p}_\omega(\Omega) \) for \( 1 \leq p < \infty \) as follows:
\[
W^{k,p}_\omega(\Omega) := \left\{ u : \forall\, |\kappa| \leq k,\ D^\kappa u \text{ exists and } D^\kappa u \in L^p_\omega(\Omega) \right\},
\]
with the norm
\[
\left\| u \right\|_{W^{k,p}_\omega(\Omega)} := \left( \sum_{|\kappa| \leq k} \left\| D^\kappa u \right\|_{L^p_\omega(\Omega)}^p \right)^{\frac{1}{p}}.
\]

For \( p = 2 \), we denote \( W^{k,2}_\omega(\Omega) \) by \( H^k_\omega(\Omega) \).

\end{definition}

In this section, we take our weight \(\omega\) to be $$\omega(x) = (1 + |x|^2)^\lambda$$ for \(\lambda > 1\). Then, it follows from the compact embedding theorem by Adams~\cite{ADAMS1971325} that, for any \(d\ge 1\),
\begin{equation}\label{eqn:compact_embedding}
H^1_\omega(\mathbb{R}^d) \hookrightarrow\hookrightarrow L^2(\mathbb{R}^d).
\end{equation}

We impose the following additional assumptions for this subsection.
\begin{assumption}\label{asp:unbounded_para}
For our four control parameters, 
\[
\beta(x,t) = \big(\beta_{k,l}(x,t)\big) \in \mathbb{R}^{d \times d}, \quad
\alpha(x,t) = \big(\alpha_k(x,t)\big) \in \mathbb{R}^d, \quad 
\gamma(x,t),\, \theta(x,t) \in \mathbb{R},
\]
we impose the following conditions:
\begin{equation}\label{cons:unbounded}
	\begin{aligned}
		|\alpha(x,t)| &\le A, \\
		bI \le \beta(x,t) &\le BI, \\
		|\gamma(x,t)| &\le C, \\
		|\theta(x,t)| &\le D,
	\end{aligned}
\end{equation}
where \( A, b, B, C, D > 0 \) are given constants. 

Without loss of generality, we may assume that \( \beta \) is symmetric, i.e., \( \beta_{kl} = \beta_{lk} \).

Moreover, we assume that 
\[
\alpha_k,\, \beta_{kl},\, \gamma,\, \theta \in C^{1,1}_{\mathrm{loc}}(\mathbb{R}^d \times \mathbb{R}), 
\quad \text{for all } k, l \in \{1, \ldots, d\},
\]
and that these functions are uniformly bounded in the weighted Sobolev space \( H^1_\omega(\mathbb{R}^d \times \mathbb{R}) \), and $\partial_t \beta_{kl}$ are uniformly bounded in $L^\infty(\mathbb{R}^d)$. We restricted all parameters on $\mathbb{R}^d\times\mathbb{R}_+$ in the parabolic equation.
\end{assumption}

\begin{remark}
The smoothness assumptions on the parameters can be weakened. However, since we are dealing with a control problem, the current regularity is sufficient for our purposes.

\end{remark}

\begin{assumption}\label{asp:unbounded_sigma}
We assume that the function $\sigma$ is Lipschitz continuous with Lipschitz constant $\mathrm{Lip}(\sigma)$, and that it satisfies $\sigma(\delta) = 0$.

\end{assumption}

\begin{thm}\label{thm:exioptcontrol:unbounded}
Under Assumptions~\ref{asp:unbounded_para} and~\ref{asp:unbounded_sigma}, given $m_0 \in L^2_\omega(\mathbb{R}^d)$, there exist optimal controls $\alpha^*, \beta^*, \gamma^*, \theta^* \in L^2(\mathbb{R}^d \times [0,T])$ such that
\[
\left\|m_{\alpha^*, \beta^*, \gamma^*, \theta^*}(\cdot,T) - m_T(\cdot)\right\|^2 = \inf_{\substack{\alpha_k, \beta_{kl}, \gamma, \theta \in \mathcal{U} \\ k, l \in \{1, \ldots, d\}}} \left\|m(\cdot,T) - m_T(\cdot)\right\|^2.
\]
where 
\[
\mathcal{U}=\{ (\alpha_k,\beta_{kl},\gamma,\theta) : \alpha_k,\beta_{kl},\gamma,\theta \text{ satisfying Assumption }\ref{asp:unbounded_para}\}
\]
\end{thm}

\begin{remark}
Theorem~\ref{thm:exioptcontrol:unbounded} shows that the optimization problem~\eqref{PDE:unbounded}--\eqref{control:unbounded}, taken over parameters $\alpha_k,\, \beta_{kl},\, \gamma,\, \theta \in \mathcal{U}$ for all $k, l \in \{1, \ldots, d\}$, admits an infimum. Moreover, this infimum is attained for parameters $\alpha^*, \beta^*, \gamma^*, \theta^* \in L^2(\mathbb{R}^d \times [0,T])$.

\end{remark}
Before proceeding with the proof of Theorem \ref{thm:exioptcontrol:unbounded}, we first recall a lemma that guarantees the existence and uniqueness of solutions to the semilinear parabolic equations \eqref{PDE:unbounded} restricted on a bounded domain with homogeneous Dirichlet boundary condition. The lemma is Theorem 6, Theorem 9, and Remark (b) in Section 4, Chapter 7 of \cite{friedman2013partial}.

\begin{lem}\label{lem:wellposednessbounded}
Suppose that $\Omega$ is bounded with sufficiently smooth boundary $\partial\Omega$, and that the coefficients $\alpha_k, \beta_{kl}, \gamma, \theta \in C^{1,1}(\overline{\Omega} \times [0,T])$ for all $k, l \in \{1, \ldots, d\}$. Under Assumption~\ref{asp:unbounded_sigma}, if $m_0 \in C^{2+\delta}(\overline{\Omega})$ for some $0 < \delta < 1$, then there exists a unique solution $m$ to the equation~\eqref{PDE:unbounded} restricted to $\Omega$, subject to the homogeneous Dirichlet boundary condition. Moreover, this solution satisfies $m \in C^{2+\eta,1+\eta/2}_{x,t}(\Omega \times [0,T])$ for some $0 < \eta < 1$.

\end{lem}

In this section, we consider the notion of a weak solution, defined as follows.

We say that a function \( m \in L^2([0,T]; H^1(\mathbb{R}^d)) \) with \( \partial_t m \in L^2([0,T]; H^{-1}(\mathbb{R}^d)) \) is a \emph{weak solution} of \eqref{PDE:unbounded} if:
\begin{itemize}
\item for any test function \( \varphi(x,t) \in C_c^\infty(\mathbb{R}^d\times [0,T] ) \),
\begin{align*}
	&\int_{\mathbb{R}^d\times [0,T]} \bigg[ \frac{\partial m}{\partial t}(x,t) 
	- \sum_{k,l=1}^d \frac{\partial}{\partial x_k} \left( \beta_{kl}(x,t) \frac{\partial}{\partial x_l}m(x,t) \right)
	+ \sum_{k=1}^d \alpha_k(x,t) \frac{\partial}{\partial x_k}m(x,t) \\
	&\quad + \theta(x,t) m(x,t) \bigg] \varphi(x,t) \, dx\, dt
	= \int_{\mathbb{R}^d\times [0,T]} \sigma \left( \gamma(x,t) m(x,t) + \delta(x,t) \right) \varphi(x,t) \, dx\, dt,
\end{align*}

\item \( m(x,0) = m_0(x) \).
\end{itemize}

In particular, given initial data \( m_0 \in H^1_\omega(\mathbb{R}^d) \), the following theorem shows that the solution \( m\) of \eqref{PDE:unbounded} lies in  
\[
L^\infty\big([0,T]; H^1_\omega(\mathbb{R}^d)\big), \quad \text{with } \partial_t m \in L^2\big([0,T]; L^2_\omega(\mathbb{R}^d)\big),
\]
which implies that the solution is continuous in time in the sense that  
\[
m \in C\big([0,T]; L^2_\omega(\mathbb{R}^d)\big).
\]

We then prove the following theorem, which guarantees the existence and uniqueness of weak solutions to \eqref{PDE:unbounded}.

\begin{thm}
Under Assumptions \ref{asp:unbounded_para} and \ref{asp:unbounded_sigma}, given initial data \(m_0 \in H^1_\omega(\mathbb{R}^d)\), the problem \eqref{PDE:unbounded} admits a unique weak solution \(m\) satisfying
\[
m \in L^\infty\big([0,T]; H^1_\omega(\mathbb{R}^d)\big), \quad \frac{\partial m}{\partial t} \in L^2\big([0,T]; L^2_\omega(\mathbb{R}^d)\big),
\]
with the following estimate:
\begin{equation}\label{est:unbounded_energy_est}
	\left\| m \right\|_{L^\infty([0,T];H^1_\omega(\mathbb{R}^d))}^2 + \left\| \frac{\partial m}{\partial t} \right\|_{L^2([0,T];L^2_\omega(\mathbb{R}^d))}^2 \leq \mathcal{C} \left\| m_0 \right\|_{H^1_\omega(\mathbb{R}^d)}^2,
\end{equation}
where the constant \(\mathcal{C}\) is independent of the parameters \(\alpha, \beta, \gamma, \theta\).

\end{thm}
\begin{proof}

Define
\[
\mathcal{B}(0,N) = \{ x \in \mathbb{R}^d \mid |x| < N \},
\]
and consider cutoff functions \(\chi_N \in C_c^\infty(\mathbb{R}^d)\) satisfying:
\begin{itemize}
	\item \(\chi_N = 1\) for \(|x| \leq N-1\);
	\item \(\chi_N = 0\) for \(|x| \geq N\);
	\item \(0 \leq \chi_N \leq 1\), with \(|\nabla \chi_N| \leq \mathcal{C}_1\),
\end{itemize}
for some constant $\mathcal{C}_1>0$ independent of $N$.

We consider the following sequence of equations on \(\mathcal{B}(0,N)\):
\begin{equation}\label{PDE:bounded_onball}
	\begin{dcases}
		\displaystyle
		\frac{\partial m_N}{\partial t}(x,t) 
		- \sum_{k,l=1}^d \frac{\partial}{\partial x_k} \left( \beta_{kl}(x,t) \frac{\partial}{\partial x_l} m_N(x,t) \right) 
		+ \sum_{k=1}^d \alpha_k(x,t) \frac{\partial}{\partial x_k} m_N(x,t) 
		+ \theta(x,t) m_N(x,t) \\
		\qquad = \sigma \left( \gamma(x,t) m_N(x,t) + \delta(x,t) \right) \quad \text{on } \mathcal{B}(0,N)\times [0,T], \\
		m_N(x,0) = \chi_N m_0(x) \quad \text{on } \mathcal{B}(0,N), \\
		m_N(x,t) = 0 \quad \text{on } \partial \mathcal{B}(0,N) \times [0,T].
	\end{dcases}
\end{equation}

Note that if we are only interested in weak solutions, the assumptions in Lemma \ref{lem:wellposednessbounded} can be relaxed. In particular, following the general strategy used in the proof of Theorem 5 in \cite[Section~7.1]{EvansPDE}, the equations \eqref{PDE:bounded_onball} are well-posed with  
\[
m_N \in L^2([0,T]; H^2(\mathcal{B}(0,N))) \cap L^\infty([0,T]; H^1_0(\mathcal{B}(0,N))), \quad \partial_t m_N \in L^2([0,T]; L^2(\mathcal{B}(0,N)))
\]
for initial data \( \chi_N m_0 \in H^1_0(\mathcal{B}(0,N)) \), which is guaranteed by the assumption \( m_0 \in H^1_\omega(\mathbb{R}^d) \).

Thanks to the available regularity, and since $\omega$ is bounded on the bounded domain, we can multiply both sides of \eqref{PDE:bounded_onball} by \( m_N \omega \) and integrating over \( \mathcal{B}(0,N) \), as follows
\begin{align*}
	\int_{\mathcal{B}(0,N)} \frac{\partial m_N}{\partial t} m_N \, \omega \, dx 
	&= \int_{\mathcal{B}(0,N)} \bigg[ \nabla \cdot (\beta \nabla m_N) m_N - \alpha \cdot \nabla m_N m_N - \theta m_N^2 + \sigma(\gamma m_N + \delta) m_N \bigg] \omega \, dx \\
	&= \int_{\mathcal{B}(0,N)} \bigg[ -(\nabla m_N)^T \beta \nabla m_N \, \omega - (\nabla \omega)^T \beta \nabla m_N \, m_N 
	- \alpha \cdot \nabla m_N m_N \, \omega \\
	&\quad - \theta m_N^2 \omega + \sigma(\gamma m_N + \delta) m_N \omega \bigg] dx.
\end{align*}

Thanks to the uniform ellipticity condition \( bI \leq \beta(x,t) \leq BI \), we obtain the following estimates:
\begin{equation}\label{eqn:est_for_nabla_omega}
	\nabla \omega \cdot (\beta \nabla m_N) \leq |\nabla \omega| \, |\beta \nabla m_N| \leq\mathcal{C}_2 |\nabla m_N| \, \omega,
\end{equation}
where \(\mathcal {C}_2 > 0\) is a generic constant depending on the bounds of \(\beta\) and the weight \(\omega\).

Hence,
\begin{align*}
	\frac{d}{dt} \left\| m_N(t) \right\|^2_{L^2_\omega(\mathcal{B}(0,N))} 
	&= 2 \int_{\mathcal{B}(0,N)} m_N \frac{\partial m_N}{\partial t} \, \omega \, dx \\
	&\le 2 \int_{\mathcal{B}(0,N)} \left[ 
	- (\nabla m_N)^T \beta \nabla m_N \, \omega 
	+ \mathcal{C}_3 |\nabla m_N||m_N| \omega 
	+ \mathcal{C}_4 |m_N|^2 \omega 
	\right] dx \\
	&\le -\mathcal{C}_5 \left\| \nabla m_N \right\|^2_{L^2_\omega(\mathcal{B}(0,N))} 
	+ \mathcal{C}_6 \left\| m_N \right\|^2_{L^2_\omega(\mathcal{B}(0,N))},
\end{align*}
where $\mathcal{C}_5,\mathcal{C}_6$ are generic constants. 

Therefore,
\[
\frac{d}{dt}\left(e^{-\mathcal{C}_6 t} \left\| m_N(t) \right\|^2_{L^2_\omega(\mathcal{B}(0,N))} \right) 
\le -\mathcal{C}_5 e^{-\mathcal{C}_6 t} \left\| \nabla m_N(t) \right\|^2_{L^2_\omega(\mathcal{B}(0,N))},
\]
which implies
\[
\left\| m_N(t) \right\|^2_{L^2_\omega(\mathcal{B}(0,N))} 
\le e^{\mathcal{C}_6 t} \left\| m_N(0) \right\|^2_{L^2_\omega(\mathcal{B}(0,N))} 
- \mathcal{C}_5 e^{\mathcal{C}_6 t} \int_0^t e^{-\mathcal{C}_6 s} \left\| \nabla m_N(s) \right\|^2_{L^2_\omega(\mathcal{B}(0,N))} \, ds,
\]
for all \( 0 \le t \le T \).

Since the time interval \([0,T]\) is finite, we obtain the following uniform energy estimate:
\begin{equation}\label{est:bounded_ball_energy_est}
	\left\| m_N \right\|_{L^{\infty}([0,T]; L^2_{\omega}(\mathcal{B}(0,N)))}^2 +
	\left\| \nabla m_N \right\|_{L^2([0,T]; L^2_{\omega}(\mathcal{B}(0,N)))}^2 
	\le \mathcal{C}_7 \left\| m_0 \right\|_{L^2_\omega(\mathcal{B}(0,N))}^2,
\end{equation}
where the constant \( \mathcal{C}_7 \) depends only on \( T \), the bounds on the parameters,  but is independent of \( N \), \( \alpha \), \( \beta \), \( \gamma \), and \( \theta \).

With the given regularity, we can also multiply both sides of \eqref{PDE:bounded_onball} by \( \frac{\partial m_N}{\partial t} \, \omega \) and integrate over \( \mathcal{B}(0,N) \) as follows:
\begin{align*}
	\int_{\mathcal{B}(0,N)} \frac{\partial m_N}{\partial t} \frac{\partial m_N}{\partial t} \, \omega \, dx 
	&= \int_{\mathcal{B}(0,N)} \bigg[ \nabla \cdot (\beta \nabla m_N) \frac{\partial m_N}{\partial t} - \alpha \cdot \nabla m_N \frac{\partial m_N}{\partial t} - \theta m_N \frac{\partial m_N}{\partial t} \\
	& \quad + \sigma(\gamma m_N + \delta) \frac{\partial m_N}{\partial t} \bigg] \omega \, dx \\
	&= \int_{\mathcal{B}(0,N)} -(\nabla \frac{\partial m_N}{\partial t})^T \beta \nabla m_N \, \omega \, dx + \int_{\mathcal{B}(0,N)} \bigg[ - (\nabla \omega)^T \beta \nabla m_N \frac{\partial m_N}{\partial t} \\
	& \quad - \alpha \cdot \nabla m_N \frac{\partial m_N}{\partial t} \, \omega - \theta m_N \frac{\partial m_N}{\partial t} \omega + \sigma(\gamma m_N + \delta) \frac{\partial m_N}{\partial t} \omega \bigg] dx \\
	&=: \mathcal{I}_1 + \mathcal{I}_2.
\end{align*}

For \(\mathcal{I}_1\), note that
\[
\frac{\partial}{\partial t} \int_{\mathcal{B}(0,N)} (\nabla m_N)^T \beta \nabla m_N \, \omega \, dx 
= \int_{\mathcal{B}(0,N)} (\nabla m_N)^T \frac{\partial \beta}{\partial t} \nabla m_N \, \omega \, dx 
+ 2 \int_{\mathcal{B}(0,N)} \left(\nabla \frac{\partial m_N}{\partial t}\right)^T \beta \nabla m_N \, \omega \, dx.
\]

Hence,
\[
\mathcal{I}_1
= \frac{1}{2} \int_{\mathcal{B}(0,N)} (\nabla m_N)^T \frac{\partial \beta}{\partial t} \nabla m_N \, \omega \, dx 
- \frac{1}{2} \frac{\partial}{\partial t} \int_{\mathcal{B}(0,N)} (\nabla m_N)^T \beta \nabla m_N \, \omega \, dx,
\]
where
\[
\int_{\mathcal{B}(0,N)} (\nabla m_N)^T \frac{\partial \beta}{\partial t} \nabla m_N \, \omega \, dx \leq \mathcal{C}_8 \left\| \nabla m_N \right\|_{L^2_\omega(\mathcal{B}(0,N))}^2.
\]

In the above inequality $\mathcal{C}_8$ is a generic constant. 

For \(\mathcal{I}_2\), using \eqref{eqn:est_for_nabla_omega}, we obtain the following estimate for \(\mathcal{I}_2\):

\[
\mathcal{I}_2 \leq \frac{1}{2} \left\| \frac{\partial m_N}{\partial t} \right\|_{L^2_\omega(\mathcal{B}(0,N))}^2 + \mathcal{C}_9 \left\| m_N \right\|_{H^1_\omega(\mathcal{B}(0,N))}^2,
\]
where \(\mathcal{C}_9\) is a generic constant.

After combining the estimates for \(\mathcal{I}_1\) and \(\mathcal{I}_2\), we obtain
\[
\frac{1}{2} \left\| \frac{\partial m_N}{\partial t} \right\|_{L^2_\omega(\mathcal{B}(0,N))}^2 + \frac{1}{2} \frac{\partial}{\partial t} \int_{\mathcal{B}(0,N)} (\nabla m_N)^T \beta \nabla m_N \, \omega \, dx \leq \mathcal{C}_{10} \left\| m_N \right\|_{H^1_\omega(\mathcal{B}(0,N))}^2.
\]

Thus,
\begin{align*}
	&\min\{1,b\} \left( \left\| \frac{\partial m_N}{\partial t} \right\|_{L^2([0,T];L^2_\omega(\mathcal{B}(0,N)))}^2 + \left\| m_N \right\|_{L^\infty([0,T];H^1_\omega(\mathcal{B}(0,N)))}^2 \right) \\
	\leq\ & \int_0^T \left\| \frac{\partial m_N}{\partial t} \right\|_{L^2_\omega(\mathcal{B}(0,N))}^2 \, dt + \sup_{0 \leq t \leq T} \int_{\mathcal{B}(0,N)} (\nabla m_N)^T \beta \nabla m_N \, \omega \, dx \\
	\leq\ & \int_{\mathcal{B}(0,N)} (\nabla m_N(x,0))^T \beta \nabla m_N(x,0) \, \omega \, dx + 2 \mathcal{C}_{10} \int_0^T \left\| m_N \right\|_{H^1_\omega(\mathcal{B}(0,N))}^2 \, dt \\
	\leq\ & B \left\| \nabla \left( \chi_N m_0(x) \right) \right\|_{L^2_\omega(\mathcal{B}(0,N))}^2 + 2 \mathcal{C}_{10} \left\| m_N \right\|_{L^2([0,T];H^1_\omega(\mathcal{B}(0,N)))}^2,
\end{align*}
where \eqref{est:bounded_ball_energy_est} implies
\[
\left\| \frac{\partial m_N}{\partial t} \right\|_{L^2([0,T];L^2_\omega(\mathcal{B}(0,N)))}^2 + \left\| m_N \right\|_{L^\infty([0,T];H^1_\omega(\mathcal{B}(0,N)))}^2 \leq \mathcal{C}_{11} \left\| m_0 \right\|_{H^1_\omega(\mathcal{B}(0,N))}^2.
\]
The constants \(\mathcal{C}_{10}\) and \(\mathcal{C}_{11}\) depend only on \(T\) and the bounds on the parameters, but are independent of \(N\), \(\alpha\), \(\beta\), \(\gamma\), and \(\theta\).

After extending $m_N$ to the whole domain $\mathbb{R}^d$, we have that
\begin{equation}\label{est:bounded_ball_energy_est_improved}
	\left\|  \frac{\partial m_N}{\partial t}\right\|_{L^2([0,T];L^2_\omega(\mathbb{R}^d))}^2+\left\|  m_N\right\|_{L^\infty([0,T];H^1_\omega(\mathbb{R}^d))}^2 \le \mathcal{C}_{11} \left\| m_0 \right\|_{H^1_\omega(\mathbb{R}^d)}^2.
\end{equation}

Note that \eqref{est:bounded_ball_energy_est_improved} implies the boundness of $\partial_t m_N$ in $L^2([0,T];L^2(\mathbb{R}^d))$ and the boundness of $m_N$ in $L^\infty([0,T];H^1_\omega(\mathbb{R}^d))$. Thus by the Aubin–Lions lemma \cite{boyer2012mathematical}, the compact embedding by \eqref{eqn:compact_embedding} 
\[
H^1_\omega(\mathbb{R}^d) \hookrightarrow\hookrightarrow L^2(\mathbb{R}^d),
\]
implies that the sequence \( m_N \) is relatively compact in $C([0,T]; L^2(\mathbb{R}^d))$. Then we can extract a subsequence \( m_{N_j} \) (still denoted by \( m_N \) for simplicity), such that
\[
\left\{\begin{aligned}
	m_N &\to m_* &&\quad \text{in } C([0,T]; L^2(\mathbb{R}^d)), \\
	m_N &\overset{*}{\rightharpoonup} m_* &&\quad \text{in } L^\infty([0,T]; H^1_\omega(\mathbb{R}^d)), \\
	m_N &\rightharpoonup m_* &&\quad \text{in } L^2([0,T]; H^1_\omega(\mathbb{R}^d)), \\
	\frac{\partial m_N}{\partial t} &\rightharpoonup \frac{\partial m_*}{\partial t} &&\quad \text{in } L^2([0,T]; L^2_\omega(\mathbb{R}^d)).
\end{aligned}
\right.
\]


For any test function \( \varphi \in C^\infty_c(\mathbb{R}^d\times [0,T]) \), there exists a sufficiently large ball \( B(0,N') \) such that $B(0,N')\times [0,T]  $ contains the support of \( \varphi \). For all $N\ge N'$, since
\begin{align*}
	&\left| \int_{\mathbb{R}^d\times [0,T]} \left[ \sigma \left( \gamma m_N + \delta \right) - \sigma \left( \gamma m_* + \delta \right) \right] \varphi \, dx\,dt \right|\\
	\le & \, \text{Lip}(\sigma) C \left\| m_N - m_* \right\|_{L^\infty([0,T];L^2(\mathbb{R}^d))} \left\| \varphi \right\|_{L^1([0,T];L^2(\mathbb{R}^d))} \rightarrow 0,
\end{align*}
we obtain
\begin{equation}\label{eqn:weakform_m_n_ball}
	\int_{\mathbb{R}^d\times [0,T]} \frac{\partial m_N}{\partial t} \varphi 
	+ \left[ -\nabla \left( \beta \nabla m_N \right) + \alpha\cdot \nabla m_N + \theta m_N \right] \varphi \, dx \, dt 
	= \int_{\mathbb{R}^d\times [0,T]} \sigma \left( \gamma m_N + \delta \right) \varphi \, dx \, dt.
\end{equation}

Letting \( N \to \infty \), we obtain
\begin{equation}\label{eqn:weakform_m_*_ball}
	\int_{\mathbb{R}^d\times [0,T]} \frac{\partial m_*}{\partial t} \varphi
	+ \left[-\nabla \left( \beta \nabla m_* \right) + \alpha\cdot \nabla m_* + \theta m_* \right] \varphi \, dx \, dt
	= \int_{\mathbb{R}^d\times [0,T]} \sigma \left( \gamma m_* + \delta \right) \varphi \, dx \, dt.
\end{equation}

Suppose \(\varphi(x,T) = 0\). By integrating by parts in \eqref{eqn:weakform_m_n_ball} and \eqref{eqn:weakform_m_*_ball}, we obtain
\[
\int_{\mathbb{R}^d} m_N(x,0) \varphi(x,0) \, dx = \int_0^T \int_{\mathbb{R}^d} \left[ -m_N \frac{\partial \varphi}{\partial t} + \left(-\nabla \left(\beta \nabla m_N \right) + \alpha\cdot \nabla  m_N + \theta m_N - \sigma(\gamma m_N + \delta) \right) \varphi \right] dx \, dt,
\]
and similarly,
\[
\int_{\mathbb{R}^d} m_*(x,0) \varphi(x,0) \, dx = \int_0^T \int_{\mathbb{R}^d} \left[ -m_* \frac{\partial \varphi}{\partial t} + \left(-\nabla \left(\beta \nabla m_* \right) + \alpha\cdot \nabla m_* + \theta m_* - \sigma(\gamma m_* + \delta) \right) \varphi \right] dx \, dt.
\]

Let \(N \to \infty\), then
\[
\int_{\mathbb{R}^d} m_N(x,0) \varphi(x,0) \, dx \to \int_{\mathbb{R}^d} m_*(x,0) \varphi(x,0) \, dx.
\]

Note that since $\text{supp}(\varphi)\subset B(0,N')$, we shall have
\[
\int_{\mathbb{R}^d} m_N(x,0) \varphi(x,0) \, dx \to \int_{\mathbb{R}^d} \chi_{N'+1} m_0(x) \varphi(x,0) \, dx= \int_{\mathbb{R}^d} m_0(x) \varphi(x,0) \, dx.
\]

By the arbitrariness of \(\varphi(x,0)\), we conclude
\[
m_*(x,0) = m_0(x).
\]

As a consequence, \(m_*\) is a weak solution of \eqref{PDE:unbounded} with
\[
m_* \in L^\infty\big([0,T]; H^1_\omega(\mathbb{R}^d)\big), \frac{\partial}{\partial t} m_* \in L^2\big([0,T]; L^2_\omega(\mathbb{R}^d)\big).
\]

Thus, the existence of the solution is established.

Estimate \eqref{est:unbounded_energy_est} follows from \eqref{est:bounded_ball_energy_est_improved},
which also implies the uniqueness of the solution.
\end{proof}

Now we are ready to prove Theorem \ref{thm:exioptcontrol:unbounded}.

\begin{proof}[Proof of Theorem \ref{thm:exioptcontrol:unbounded}]
Let $\{\alpha^n,\beta^n,\gamma^n,\theta^n\}$ be a minimizing sequence such that
\[
\left\|m_{\alpha^n,\beta^n,\gamma^n,\theta^n}(\cdot,T)-m_T(\cdot)\right\|^2 \to \inf_{\alpha,\beta,\gamma,\theta \in \mathcal{U}} \left\|m(\cdot,T)-m_T(\cdot)\right\|^2.
\]


By \eqref{eqn:compact_embedding} and Assumption \ref{asp:unbounded_para}, these parameter sequences admit subsequences that converge in \(L^2(\mathbb{R}^d \times [0,T])\). By iteratively extracting subsequences, we obtain a subsequence \(\{\alpha^{n_j}, \beta^{n_j}, \gamma^{n_j}, \theta^{n_j}\}\) that converges, and we denote its limit by \(\{\alpha^*, \beta^*, \gamma^*, \theta^*\}\).

Thanks to \eqref{est:unbounded_energy_est}, the corresponding sequence of solutions
\[
m_{\alpha^{n_j}, \beta^{n_j}, \gamma^{n_j}, \theta^{n_j}} \in L^\infty([0,T]; H^1_\omega(\mathbb{R}^d)), \quad \frac{d}{dt} m_{\alpha^{n_j}, \beta^{n_j}, \gamma^{n_j}, \theta^{n_j}} \in L^2([0,T]; L^2_\omega(\mathbb{R}^d)),
\]
satisfies the uniform bounds
\[
\left\| m_{\alpha^{n_j}, \beta^{n_j}, \gamma^{n_j}, \theta^{n_j}} \right\|_{L^\infty([0,T]; H^1_\omega(\mathbb{R}^d))}^2 \leq \mathcal{C}_1, \quad
\left\| \frac{d}{dt} m_{\alpha^{n_j}, \beta^{n_j}, \gamma^{n_j}, \theta^{n_j}} \right\|_{L^2([0,T]; L^2_\omega(\mathbb{R}^d))}^2 \leq \mathcal{C}_2,
\]
for some universal constants \(\mathcal{C}_1, \mathcal{C}_2 > 0\).

Using the compact embedding \(H^1_\omega (\mathbb{R}^d)\hookrightarrow\hookrightarrow L^2(\mathbb{R}^d)\), we apply the Aubin-Lions lemma \cite{boyer2012mathematical}. By iteratively extracting subsequences, we obtain a further subsequence such that

\[
\left\{\begin{aligned}
	m_{\alpha^{n_k}, \beta^{n_k}, \gamma^{n_k}, \theta^{n_k}}  &\to m_* &&\quad \text{in } C([0,T]; L^2(\mathbb{R}^d)), \\
	m_{\alpha^{n_k}, \beta^{n_k}, \gamma^{n_k}, \theta^{n_k}}  &\overset{*}{\rightharpoonup} m_* &&\quad \text{in } L^\infty([0,T]; H^1_\omega(\mathbb{R}^d)), \\
	m_{\alpha^{n_k}, \beta^{n_k}, \gamma^{n_k}, \theta^{n_k}} &\rightharpoonup m_* &&\quad \text{in } L^2([0,T]; H^1_\omega(\mathbb{R}^d)), \\
	\frac{\partial }{\partial t} m_{\alpha^{n_k}, \beta^{n_k}, \gamma^{n_k}, \theta^{n_k}}  &\rightharpoonup \frac{\partial m_*}{\partial t} &&\quad \text{in } L^2([0,T]; L^2_\omega(\mathbb{R}^d)).
\end{aligned}
\right.
\]

For simplicity, we still denote the subsequence \(\{\alpha^{n_k}, \beta^{n_k}, \gamma^{n_k}, \theta^{n_k}\}\) as \(\{\alpha^n, \beta^n, \gamma^n, \theta^n\}\), and set 
\(m_{\alpha^{n_k}, \beta^{n_k}, \gamma^{n_k}, \theta^{n_k}} = m_n\). Then \(m_n\) satisfies the following equation:
\[
\left\{
\begin{aligned}
	\frac{\partial m_n}{\partial t}(x,t) 
	&- \nabla \cdot \left( \beta^n \nabla m_n \right) 
	+ \alpha^n \cdot \nabla m_n + \theta^n m_n = \sigma \left( \gamma^n m_n + \delta \right) 
	&& \text{on } \mathbb{R}^d\times [0,T], \\
	m_n(x,0) &= m_0(x) && \text{on } \mathbb{R}^d,
\end{aligned}
\right.
\]
with the following convergence properties, which are weaker than those established earlier but nevertheless sufficient for the analysis below

\[
\left\{
\begin{aligned}
	m_n &\to m_* && \text{in } C([0,T]; L^2(\mathbb{R}^d)), \\
	m_n &\rightharpoonup m_* && \text{in } L^2([0,T]; H^1(\mathbb{R}^d)), \\
	\frac{\partial m_n}{\partial t} &\rightharpoonup \frac{\partial m_*}{\partial t} && \text{in } L^2([0,T]; L^2(\mathbb{R}^d)).
\end{aligned}
\right.
\]

For any \(\varphi \in C_c^\infty(\mathbb{R}^d\times [0,T])\), as \(n \to \infty\), we have
\begin{align*}
	&\left| \int_0^T \int_{\mathbb{R}^d} \big(\nabla \cdot (\beta^n \nabla m_n) - \nabla \cdot (\beta^* \nabla m_*) \big) \varphi \, dx dt \right| \\
	&= \left| \int_0^T \int_{\mathbb{R}^d} (\beta^n \nabla m_n - \beta^* \nabla m_*) \cdot \nabla \varphi \, dx dt \right| \\
	&\leq \|\beta^n - \beta^*\|_{L^2_{x,t}} \|\nabla m_n\|_{L^2_{x,t}} \|\nabla \varphi\|_{L^\infty_{x,t}} + \left| \int_0^T \int_{\mathbb{R}^d} (\nabla m_n - \nabla m_*) \cdot (\beta^* \nabla \varphi) \, dx dt \right| \to 0,
\end{align*}
since \(\beta^n \to \beta^*\) strongly in \(L^2_{x,t}\) and \(\nabla m_n \rightharpoonup \nabla m_*\) weakly in \(L^2_{x,t}\).

Similarly,
\begin{align*}
	&\left| \int_0^T \int_{\mathbb{R}^d} \big(\alpha^n \cdot \nabla m_n - \alpha^*\cdot\nabla m_* \big) \varphi \, dx dt \right| \\
	&\leq  \|\alpha^n - \alpha^*\|_{L^2_{x,t}} \|\nabla m_n\|_{L^2_{x,t}}\|\varphi\|_{L^\infty_{x,t}} + \left| \int_0^T \int_{\mathbb{R}^d} (\nabla m_n - \nabla m_*) \cdot (\alpha^* \varphi) \, dx dt \right| \to 0,
\end{align*}
and
\[
\left| \int_0^T \int_{\mathbb{R}^d} (\theta^n m_n - \theta^* m_*) \varphi \, dx dt \right| \leq \left( \|\theta^n - \theta^*\|_{L^2_{x,t}} \|m_n\|_{L^2_{x,t}} + \|\theta^*\|_{L^2_{x,t}} \|m_n - m_*\|_{L^2_{x,t}} \right) \|\varphi\|_{L^\infty_{x,t}} \to 0.
\]

Finally, using the Lipschitz continuity of \(\sigma\),
\begin{align*}
	&\left| \int_0^T \int_{\mathbb{R}^d} \big(\sigma(\gamma^n m_n + \delta) - \sigma(\gamma^* m_* + \delta)\big) \varphi \, dx dt \right| \\
	&\leq \mathrm{Lip}(\sigma) \int_0^T \int_{\mathbb{R}^d} | \gamma^n m_n - \gamma^* m_* | |\varphi| \, dx dt \\
	&\leq \mathrm{Lip}(\sigma) \left( \|\gamma^n - \gamma^*\|_{L^2_{x,t}} \|m_n\|_{L^2_{x,t}} + \|\gamma^*\|_{L^2_{x,t}} \|m_n - m_*\|_{L^2_{x,t}} \right) \|\varphi\|_{L^\infty_{x,t}} \to 0.
\end{align*}

Therefore, we conclude the following weak convergences:
\[
\left\{
\begin{aligned}
	\nabla \cdot (\beta^n \nabla m_n) &\rightharpoonup \nabla \cdot (\beta^* \nabla m_*) && \text{in } L^2([0,T]; H^{-1}(\mathbb{R}^d)), \\
	\alpha^n \cdot \nabla m_n &\rightharpoonup \alpha^*\cdot \nabla m_* && \text{in } L^2([0,T]; L^{2}(\mathbb{R}^d)), \\
	\theta^n m_n &\rightharpoonup \theta^* m_* && \text{in } L^2([0,T]; L^2(\mathbb{R}^d)), \\
	\sigma(\gamma^n m_n + \delta) &\rightharpoonup \sigma(\gamma^* m_* + \delta) && \text{in } L^2([0,T]; L^2(\mathbb{R}^d)).
\end{aligned}
\right.
\]

For any test function \(\varphi \in C_c^\infty(\mathbb{R}^d\times [0,T])\), the weak form of the PDE reads
\begin{equation}\label{eqn:weakform_m_n_unbounded}
	\int_0^T \int_{\mathbb{R}^d} \frac{\partial m_n}{\partial t} \varphi \, dx dt
	+ \int_0^T \int_{\mathbb{R}^d} \biggl[ 
	-\nabla \cdot \left( \beta^n \nabla m_n \right) 
	+ \alpha^n \cdot \nabla m_n
	+ \theta^n m_n
	\biggr] \varphi \, dx dt
	= \int_0^T \int_{\mathbb{R}^d} \sigma \left( \gamma^n m_n + \delta \right) \varphi \, dx dt.
\end{equation}

Letting \( n \to \infty \), we find
\begin{equation}\label{eqn:weakform_m_*_unbounded}
	\int_0^T \int_{\mathbb{R}^d} \frac{\partial m_*}{\partial t} \varphi \, dx dt
	+ \int_0^T \int_{\mathbb{R}^d} \biggl[
	-\nabla \cdot \left( \beta^* \nabla m_* \right) 
	+ \alpha^* \cdot \nabla m_*  
	+ \theta^* m_*
	\biggr] \varphi \, dx dt
	= \int_0^T \int_{\mathbb{R}^d} \sigma \left( \gamma^* m_* + \delta \right) \varphi \, dx dt.
\end{equation}

Supposing \(\varphi(x,T) = 0\) and integrating by parts in time for \eqref{eqn:weakform_m_n_unbounded} and \eqref{eqn:weakform_m_*_unbounded}, we have
\[
\int_{\mathbb{R}^d} m_n(x,0) \varphi(x,0) \, dx\] \[ = \int_0^T \int_{\mathbb{R}^d} \left[
-m_n \frac{\partial \varphi}{\partial t} + \bigl(-\nabla \cdot (\beta^n \nabla m_n) + \alpha^n \cdot\nabla m_n) + \theta^n m_n - \sigma(\gamma^n m_n + \delta)\bigr) \varphi
\right] dx dt,
\]
yielding
\[
\int_{\mathbb{R}^d} m_*(x,0) \varphi(x,0) \, dx\] \[ = \int_0^T \int_{\mathbb{R}^d} \left[
-m_* \frac{\partial \varphi}{\partial t} + \bigl(-\nabla \cdot (\beta^* \nabla m_*) + \alpha^* \cdot \nabla m_*) + \theta^* m_* - \sigma(\gamma^* m_* + \delta)\bigr) \varphi
\right] dx dt.
\]

Letting \(n \to \infty\), we have
\[
\int_{\mathbb{R}^d} m_n(x,0) \varphi(x,0) \, dx \to \int_{\mathbb{R}^d} m_*(x,0) \varphi(x,0) \, dx.
\]

By the arbitrariness of \(\varphi(x,0)\), it follows that
\[
m_*(x,0) = m_n(x,0) = m_0(x).
\]

Therefore, \(m_*\) is the weak solution of the PDE with parameters \(\alpha^*, \beta^*, \gamma^*, \theta^*\).

Recalling the strong convergence
\[
m_n \to m_* \quad \text{in} \quad C([0,T]; L^2(\mathbb{R}^d)),
\]
we conclude that
\[
\| m_{\alpha^*, \beta^*, \gamma^*, \theta^*}(\cdot, T) - m_T(\cdot) \|^2 = \inf_{\alpha, \beta, \gamma, \theta\in\mathcal{U}} \| m(\cdot, T) - m_T(\cdot) \|^2,
\]
which leads to the conclusion of Theorem \ref{thm:exioptcontrol:unbounded}.
\end{proof}

\section{Bilinear Control for Hyperbolic Equations}\label{Section:Hyper}

We consider a simplified version of \eqref{Hyper1} for control purposes:
\begin{equation}\label{eqn:ControlPDE_nonlinear}
\begin{cases}
	\displaystyle
	\frac{\partial^2}{\partial t^2}m(x,t) - \Delta m(x,t) + 2x \cdot \nabla m(x,t) + \theta(x,t)m(x,t) = \sigma(m(x,t)) \quad \text{on } \mathbb{R}^d \times \mathbb{R},\\
	m(x,0)=m_0(x),\\
	\partial_t m(x,0)=\tilde{m}_0(x),
\end{cases}
\end{equation}
where the functions $\alpha$, $\beta$, $\gamma$, and $\delta$ in \eqref{Hyper1} are specified as:
\begin{equation}
\begin{cases}
	\alpha(x,t) = 2x, \\
	\beta(x,t) = I, \\
	\gamma(x,t) = 1, \\
	\delta(x,t) = 0.
\end{cases}
\end{equation}

Control is applied solely through the function $\theta(x,t)$. The function $\sigma$ is a Lipschitz continuous nonlinearity, satisfying $\sigma(0) = 0$, consistent with Assumption~\ref{asp:unbounded_sigma}.

The problem now reduces to a bilinear control problem for hyperbolic equations. Controllability results for such problems are known in certain cases; see, for example, \cite{beauchard2018minimal,cannarsa2023bilinear,duca2024bilinear,POZZOLI2024421}.



Note that the Ornstein--Uhlenbeck operator \((\Delta - 2x \cdot \nabla)\) can be rewritten as
\[
\Delta_\omega = \omega^{-1} \nabla \cdot (\omega \nabla),
\]
where \(\omega = e^{-|x|^2}\). Recalling the definitions of the weighted \(L^p\) and Sobolev spaces introduced in Definition~\ref{def:weightedSobolev}, the natural functional setting for \eqref{eqn:ControlPDE_nonlinear} is 
\[
\bigl(m(\cdot, t), \tfrac{\partial}{\partial t}m(\cdot, t)\bigr) \in H^1_\omega(\mathbb{R}^d) \times L^2_\omega(\mathbb{R}^d),
\]
where the weight is \(\omega = e^{-|x|^2}\).

Our aim now is to establish an approximate controllability result for \eqref{eqn:ControlPDE_nonlinear}--\eqref{Hyper2}. To this end, we define the norm $\|\cdot\|$ in this section as follows
\[ \|(u,v)\| =\| (u,v)  \|_{H^1_\omega(\mathbb{R}^d)\times L^2_\omega(\mathbb{R}^d)} =  \left( \|u\|_{H^1_\omega(\mathbb{R}^d)}^2+\|v\|_{L^2_\omega(\mathbb{R}^d)}^2 \right)^{\frac{1}{2}}.\]

Our strategy is to reduce the problem to controlling the following linear PDE:
\begin{equation}\label{eqn:ControlPDE_linear}
\begin{cases}
	\displaystyle
	\frac{\partial^2}{\partial t^2}m(x,t) - \Delta m(x,t) + 2x \cdot \nabla m(x,t) + \theta(x,t)m(x,t) = 0 \quad \text{on } \mathbb{R}^d \times \mathbb{R},\\
	m(x,0)=m_0(x),\\
	\partial_t m(x,0)=\tilde{m}_0(x).
\end{cases}
\end{equation}

Similar to \cite{POZZOLI2024421}, we control the system under the assumption that the control function $\theta$ is given by
\begin{equation}\label{eqn:formula_theta}
\theta(x,t) = \sum_{j=0}^{2d} \mathfrak{p}_j(t)\theta_j(x),
\end{equation}
where $\mathfrak{p}_j(t)$ are piecewise constant control functions taking values in $\mathbb{R}$. We denote $$\mathfrak{p}:=(\mathfrak{p}_0,\cdots,\mathfrak{p}_{2d}).$$ 

The spatial basis functions $\theta_j(x)$ are defined as
\begin{equation}\label{eqn:set_theta}
(\theta_0(x),\dots,\theta_{2d}(x)) := \big(1, \cos(\phi(x_1)), \sin(\phi(x_1)), \dots, \cos(\phi(x_d)), \sin(\phi(x_d)) \big),
\end{equation}
where $\phi$ is a mapping from $\mathbb{R}$ to $(-\pi,\pi)$ defined by
\begin{equation}\label{eqn:formula_for_phi}
\phi(x) = 2\sqrt{\pi} \int_{0}^{x} e^{-s^2}\,ds.
\end{equation}

In this section, the solution we consider is the \emph{mild solution}, defined as follows.

By setting \( M(t) = (m(t), \partial_t m(t)) \), we can reformulate \eqref{eqn:ControlPDE_nonlinear} and \eqref{eqn:ControlPDE_linear} as the following first-order abstract Cauchy problem:
\begin{equation}\label{eqn:ControlPDE_linear_abstract}
\frac{\partial}{\partial t} M(x,t) = (\mathcal{X} - \theta(x,t) \mathcal{Y}) M(x,t) + \mathcal{Z}(M(x,t)), \quad M_0 = (m_0, \tilde{m}_0),
\end{equation}
where \( \theta \) satisfies \eqref{eqn:formula_theta} and \eqref{eqn:set_theta}, and
\begin{equation}
\mathcal{X} = \begin{pmatrix} 0 & I \\ \Delta - 2x \cdot \nabla & 0 \end{pmatrix}, \quad
\mathcal{Y} = \begin{pmatrix} 0 & 0 \\ I & 0 \end{pmatrix}, \quad
\mathcal{Z}(M) = \begin{pmatrix} 0 \\ \sigma(m) \end{pmatrix}.
\end{equation}

For \eqref{eqn:ControlPDE_nonlinear}, \( \sigma \) is a Lipschitz nonlinearity satisfying \( \sigma(0) = 0 \); for \eqref{eqn:ControlPDE_linear}, we set \( \sigma = 0 \).

We say that a function \( (m(x,t), \partial_t m(x,t)) \in C\left([0,T]; H^1_\omega(\mathbb{R}^d) \times L^2_\omega(\mathbb{R}^d)\right) \) is a \emph{mild solution} of \eqref{eqn:ControlPDE_nonlinear} or \eqref{eqn:ControlPDE_linear} if it satisfies the following integral equation for all $t\in [0,T]$:
\[
M(t) = e^{t\mathcal{X}} M_0 + \int_0^t e^{(t-s)\mathcal{X}} \left( -\theta(x,s)\mathcal{Y} + \mathcal{Z} \right) M(s) \, ds.
\]

The operator $- \left( \Delta - 2x \cdot \nabla \right)$ on $\mathbb{R}^d$ has the following eigenfunctions and eigenvalues:
\begin{equation}\label{eqn:eigenfunction}
\psi_{\mathbf{n}}(x) = \prod_{i=1}^d \frac{H_{n_i}(x_i)}{\sqrt{2^{n_i} n_i! \sqrt{\pi}}}, \quad \rho_{\mathbf{n}} = 2|\mathbf{n}|, \quad \langle \psi_{\mathbf{n}}, \psi_{\mathbf{m}} \rangle = \delta_{\mathbf{n}\mathbf{m}},
\end{equation}
where $\mathbf{n} = (n_1, \dots, n_d) \in \mathbb{N}^d$, with $\mathbb{N}:=\{0,1,2,...\}$, $|\mathbf{n}|=\sum_{i=1}^d|n_i|$, and $\langle \cdot, \cdot \rangle$ denotes the inner product in $L^2_\omega(\mathbb{R}^d)$, i.e.,

\[
\langle f, g \rangle := \int_{\mathbb{R}^d} f(x)\,g(x)\,e^{-|x|^2} \, dx,
\]
and $H_n$ are the Hermite polynomials defined by
\begin{equation}
H_n(x) = (-1)^n e^{x^2} \frac{d^n}{dx^n} e^{-x^2}, \quad x \in \mathbb{R}.
\end{equation}

In the following, we will consider two cases for the initial data:
\begin{assumption}\label{asp:finitecombination}
For a given state $(m_0,\tilde{m}_0) \in H^1_\omega(\mathbb{R}^d) \times L^2_\omega(\mathbb{R}^d)$, we assume that $m_0 \neq 0$ and that $m_0$ is a finite linear combination of eigenfunctions $\psi_{\mathbf{n}}$.
\end{assumption}

\begin{assumption}\label{asp:finitecombination2}
For a given state $(m_0,\tilde{m}_0) \in H^1_\omega(\mathbb{R}^d) \times L^2_\omega(\mathbb{R}^d)$, we assume that $m_0 = 0$, $\tilde{m}_0 \neq 0$, and that $\tilde{m}_0$ is a finite linear combination of eigenfunctions $\psi_{\mathbf{n}}$.
\end{assumption}
Our results will apply whenever either one of these assumptions is satisfied.

\begin{thm}\label{thm:control_for_nonlinear_wave}
For any initial state $(m_0,\tilde{m}_0) \in H^1_\omega(\mathbb{R}^d) \times \bigl( L^2_\omega (\mathbb{R}^d)\cap L^p_{\mathrm{loc}}(\mathbb{R}^d)\bigr)$ for some $p > 2$ satisfying Assumption~\ref{asp:finitecombination} or Assumption~\ref{asp:finitecombination2}, and for any final state $(m_T, \tilde{m}_T) \in H^1_\omega \times L^2_\omega(\mathbb{R}^d)$, and any $\varepsilon, T > 0$, there exists a control $\theta \in L^\infty(\mathbb{R}^d \times [0,T])$ such that the mild solution \((m,\partial_t m)\in C\left([0,T]; H^1_\omega(\mathbb{R}^d) \times L^2_\omega(\mathbb{R}^d)\right)\) of \eqref{eqn:ControlPDE_nonlinear} with initial state $(m_0, \tilde{m}_0)$ satisfies
\[
\left\| (m(\cdot, T),  \partial_t{m}(\cdot, T)) - (m_T, \tilde{m}_T) \right\|_{H^1_\omega(\mathbb{R}^d) \times L^2_\omega(\mathbb{R}^d)} < \varepsilon.
\]
\end{thm}

\begin{remark}
Our assumption on the initial state, namely \((m_0, \tilde{m}_0) \in H^1_\omega \times \bigl( L^2_\omega(\mathbb{R}^d) \cap L^p_{\text{loc}}(\mathbb{R}^d)\bigr)\) for some \(p > 2\), satisfying either Assumption \ref{asp:finitecombination} or Assumption \ref{asp:finitecombination2}, includes the case where \((m_0, \tilde{m}_0) \neq (0,0)\) and both \(m_0\) and \(\tilde{m}_0\) are polynomials in variable $x_1,...,x_n$. Note that $(0,0)$ is the stationary solution of \eqref{eqn:ControlPDE_nonlinear} or \eqref{eqn:ControlPDE_linear} for any choice of the control. Hence, the system is uncontrollable from the initial state $(0,0)$.
\end{remark}

The proof of Theorem \ref{thm:control_for_nonlinear_wave} relies on the following approximate controllability result for \eqref{eqn:ControlPDE_linear}.

\begin{prop}\label{prop:control_for_linear_wave}
For any initial state \((m_0, \tilde{m}_0) \in H^1_\omega(\mathbb{R}^d) \times \bigl( L^2_\omega(\mathbb{R}^d) \cap L^p_{\mathrm{loc}}(\mathbb{R}^d)\bigr)\) for some \(p > 2\) satisfying Assumption \ref{asp:finitecombination} or Assumption \ref{asp:finitecombination2}, any final state \((m_T, \tilde{m}_T) \in H^1_\omega(\mathbb{R}^d) \times L^2_\omega(\mathbb{R}^d)\), and any \(\varepsilon, T > 0\), there exists a piecewise constant control \(\mathfrak{p} : [0,T] \to \mathbb{R}^{2d+1}\) such that the mild solution \((m,\partial_t m)\in C\left([0,T]; H^1_\omega(\mathbb{R}^d) \times L^2_\omega(\mathbb{R}^d)\right)\) of \eqref{eqn:ControlPDE_linear} under the control law given by \eqref{eqn:formula_theta} and \eqref{eqn:set_theta} with initial state \((m_0, \tilde{m}_0)\) satisfies
\[
\left\| \left( m(\cdot, T), \partial_t{m}(\cdot, T) \right) - (m_T, \tilde{m}_T) \right\|_{H^1_\omega(\mathbb{R}^d) \times L^2_\omega (\mathbb{R}^d)} < \varepsilon.
\]
\end{prop}

The proof of Proposition \ref{prop:control_for_linear_wave} mainly relies on Propositions \ref{prop:control_displacement} and \ref{prop:control_velocity}. Note that the case of a periodic domain in Proposition \ref{prop:control_for_linear_wave} has been considered in \cite{POZZOLI2024421}.  In the next subsection, we first establish the well-posedness of \eqref{eqn:ControlPDE_linear}, then prove these two propositions in the subsequent subsections.

\subsection{Well-posedness of the wave equation}


Under the spectral decomposition of the Ornstein–Uhlenbeck operator \( \left( \Delta - 2x \cdot \nabla \right) \), we can express the free evolution with the given initial state \( (m(t), \partial_t m(t)) := e^{t\mathcal{X}}(m_0, \tilde{m}_0) \in H^1_\omega(\mathbb{R}^d) \times L^2_\omega(\mathbb{R}^d) \) as follows:
\begin{equation}\label{eqn:linear_displacement_evolution}
m(t) = \left( \langle m_0, \psi_0 \rangle + \langle \tilde{m}_0, \psi_0 \rangle t \right) \psi_0 
+ \sum_{\mathbf{n} \in \mathbb{N}^d \setminus\{\mathbf{0}\}} 
\left( 
\langle m_0, \psi_{\mathbf{n}} \rangle \cos(\sqrt{\rho_{\mathbf{n}}} t) 
+ \langle \tilde{m}_0, \psi_{\mathbf{n}} \rangle \frac{\sin(\sqrt{\rho_{\mathbf{n}}} t)}{\sqrt{\rho_{\mathbf{n}}}} 
\right) \psi_{\mathbf{n}},
\end{equation}
\begin{equation}\label{eqn:linear_velocity_evolution}
\partial_t m(t) = \sum_{\mathbf{n} \in \mathbb{N}^d} 
\left( 
- \sqrt{\rho_{\mathbf{n}}} \langle m_0, \psi_{\mathbf{n}} \rangle \sin(\sqrt{\rho_{\mathbf{n}}} t) 
+ \langle \tilde{m}_0, \psi_{\mathbf{n}} \rangle \cos(\sqrt{\rho_{\mathbf{n}}} t) 
\right) \psi_{\mathbf{n}}.
\end{equation}

\begin{remark}
Based on the expressions above, the operator \( e^{t\mathcal{X}} \) generates a \( C_0 \)-group on \( H^1_\omega(\mathbb{R}^d) \times L^2_\omega(\mathbb{R}^d) \).
\end{remark}
\begin{prop}\label{Mild}
For any \( T > 0 \), any 
\( \mathfrak{p} \in L^1([0,T]; \mathbb{R}^{2d+1}) \),
and any initial state 
\( M_0 \in H^1_\omega(\mathbb{R}^d) \times L^2_\omega(\mathbb{R}^d) \),
there exists a unique mild solution
\[
\mathcal{S}(\cdot, M_0, \mathfrak{p}) \in C\left([0,T]; H^1_\omega(\mathbb{R}^d) \times L^2_\omega(\mathbb{R}^d)\right)
\]
of \eqref{eqn:ControlPDE_linear_abstract}, which satisfies, for all \( t \in [0,T] \),
\begin{equation}
	\mathcal{S}(t, M_0, \mathfrak{p}) = e^{t\mathcal{X}} M_0 
	+ \int_{0}^t e^{(t-s)\mathcal{X}} 
	\left( 
	- \sum_{j=0}^{2d} \mathfrak{p}_j(s) \theta_j(x) \mathcal{Y} + \mathcal{Z}
	\right) 
	\mathcal{S}(s, M_0, \mathfrak{p}) \, ds.
\end{equation}

Moreover, there exists a constant \( \mathcal C = \mathcal C(\mathfrak{p}, T) \) such that, for any other initial state 
\( M_1 = (m_1, \dot{m}_1) \in H^1_\omega(\mathbb{R}^d) \times L^2_\omega(\mathbb{R}^d) \),
the corresponding solutions satisfy
\begin{equation}\label{eqn:Lipschitz_continuity}
	\left\| 
	\mathcal{S}(\cdot, M_0, \mathfrak{p}) 
	- \mathcal{S}(\cdot, M_1, \mathfrak{p}) 
	\right\|_{C([0,T]; H^1_\omega(\mathbb{R}^d) \times L^2_\omega(\mathbb{R}^d))} 
	\leq 
	C \left\| M_0 - M_1 \right\|_{H^1_\omega(\mathbb{R}^d) \times L^2_\omega(\mathbb{R}^d)}.
\end{equation}

\end{prop}
\begin{proof}

We use a fixed-point argument to prove existence and uniqueness.

Let $\tau>0$ be a small constant to be fixed later. For a given initial state \( M_0 \) and control law \( \mathfrak{p} \), we define the following map on the space
\[
X : = \left\{ M(t)=(m(t),\partial_t m(t)) \in C\left([0,\tau], H^1_\omega(\mathbb{R}^d) \times L^2_\omega(\mathbb{R}^d)\right) \,\middle|\,M(0) = M_0 \right\},
\]
\begin{equation}
	\Phi(M)(t) := e^{t\mathcal{X}} M_0 
	+ \int_{0}^t e^{(t-s)\mathcal{X}} 
	\left( 
	- \sum_{j=0}^{2d} \mathfrak{p}_j(s) \theta_j(x) \mathcal{Y} + \mathcal{Z}
	\right) M(s) \, ds, \quad \forall t \in [0,\tau].
\end{equation}

Since \( e^{t\mathcal{X}} \) is a \( C_0 \)-group on \( H^1_\omega(\mathbb{R}^d) \times L^2_\omega(\mathbb{R}^d) \), there exists a constant \( \mathcal{C}(\tau) > 0 \) such that for all \( 0 \le t \le \tau \),
\[
\left\| e^{t\mathcal{X}} \right\|_{H^1_\omega(\mathbb{R}^d) \times L^2_\omega(\mathbb{R}^d) \to H^1_\omega(\mathbb{R}^d) \times L^2_\omega(\mathbb{R}^d)} \le \mathcal{C}(\tau).
\]

Then we have , for any $M(t)=(m(t),\partial_t m(t))\in X$:
\begin{align*}
	\left\| \Phi(M)(t) \right\|_{H^1_\omega \times L^2_\omega}
	&\le \left\| e^{t\mathcal{X}} M_0 \right\|_{H^1_\omega \times L^2_\omega} 
	+ \int_{0}^\tau \left\| e^{(t-s)\mathcal{X}} \left( -\sum_{j=0}^{2d} \mathfrak{p}_j(s) \theta_j(x) \mathcal{Y} + \mathcal{Z} \right) M(s) \right\|_{H^1_\omega \times L^2_\omega} ds \\
	&\le \mathcal{C}(\tau) \left\| M_0 \right\|_{H^1_\omega \times L^2_\omega} 
	+ \mathcal{C}(\tau) \left\| \mathfrak{p} \right\|_{L^1([0,\tau])} \int_0^\tau \left\| m(s) \right\|_{L^2_\omega} ds \\
	&\quad + \mathcal{C}(\tau) \int_0^\tau \left\| \sigma(m(s)) \right\|_{L^2_\omega} ds \\
	&< \infty, \quad \forall t \in [0,\tau],
\end{align*}
thus \( \Phi \) maps \( X \) into itself.

For any \( M^{(1)}(t), M^{(2)}(t) \in X \), for any $t\in [0,\tau]$, we have:
\begin{align*}
	&	\left\| \Phi(M^{(1)})(t) - \Phi(M^{(2)})(t) \right\|_{H^1_\omega \times L^2_\omega}\\
	&\le \int_{0}^\tau \left\| e^{(t-s)\mathcal{X}} \left( -\sum_{j=0}^{2d} \mathfrak{p}_j(s) \theta_j(x) \mathcal{Y} + \mathcal{Z} \right) \left( M^{(1)}(s) - M^{(2)}(s) \right) \right\|_{H^1_\omega \times L^2_\omega} \, ds \\
	&\le \mathcal{C}(\tau) \left\| \mathfrak{p} \right\|_{L^1([0,\tau])} \int_{0}^\tau \left\| m^{(1)}(s) - m^{(2)}(s) \right\|_{L^2_\omega} \, ds \\
	&\quad + \mathcal{C}(\tau) \int_0^\tau \left\| \sigma(m^{(1)}(s)) - \sigma(m^{(2)}(s)) \right\|_{L^2_\omega} \, ds \\
	&\le \mathcal{C}(\tau) \tau \left( \left\| \mathfrak{p} \right\|_{L^1([0,\tau])} + \operatorname{Lip}(\sigma) \right) 
	\left\| M^{(1)} - M^{(2)} \right\|_{C([0,\tau], H^1_\omega \times L^2_\omega)}.
\end{align*}

Since \( e^{t\mathcal{X}} \) is a \( C_0 \)-group on \( H^1_\omega(\mathbb{R}^d) \times L^2_\omega(\mathbb{R}^d) \),  we can choose a sufficiently small time \( \tau > 0 \) such that
\[
\mathcal{C}(\tau)\tau \left( \left\| \mathfrak{p} \right\|_{L^1([0,T])} + \operatorname{Lip}(\sigma) \right) < 1.
\]

Under this condition, the map \( \Phi \) is a contraction on the interval \( [0, \tau] \), and by Banach’s fixed-point theorem, it admits a unique fixed point. This guarantees the existence and uniqueness of the solution on \( [0, \tau] \).

To extend the solution to the full interval \( [0, T] \), we decompose it into subintervals
\[
0 = T_0 < \tau =T_1 < \cdots < n\tau = T_n = T.
\]

On each subinterval \( [T_i, T_{i+1}] \), we apply the same fixed-point argument, using the terminal value of the solution on \( [T_{i-1}, T_i] \) as the initial condition for the next interval. By gluing together the solutions on each subinterval, we obtain a global solution on \( [0, T] \). Uniqueness on each subinterval ensures the uniqueness of the full solution.

For any other initial state \( M_1 = (m_1, \tilde{m}_1) \in H^1_\omega(\mathbb{R}^d) \times L^2_\omega(\mathbb{R}^d) \), and for any \( t \in [0, T] \), we have
\begin{align*}
	& \|\mathcal{S}(t, M_0, \mathfrak{p}) - \mathcal{S}(t, M_1, \mathfrak{p})\|_{H^1_\omega \times L^2_\omega} \\
	\leq\; & \left\| e^{t \mathcal{X}} (M_0 - M_1) \right\|_{H^1_\omega \times L^2_\omega} + \left\| \int_0^t e^{(t-s)\mathcal{X}} \left(- \sum_{j=0}^{2d} \mathfrak{p}_j(s) \theta_j(x) \mathcal{Y} + \mathcal{Z} \right) \big(\mathcal{S}(s, M_0, \mathfrak{p}) - \mathcal{S}(s, M_1, \mathfrak{p}) \big) ds \right\|_{H^1_\omega \times L^2_\omega} \\
	\leq\; & \mathcal{C}(T) \|M_0 - M_1\|_{H^1_\omega \times L^2_\omega} + \mathcal{C}(T) \left( \|\mathfrak{p}\|_{L^1([0,T])} + \operatorname{Lip}(\sigma) \right) \int_0^t \|\mathcal{S}(s, M_0, \mathfrak{p}) - \mathcal{S}(s, M_1, \mathfrak{p})\|_{H^1_\omega \times L^2_\omega} ds.
\end{align*}

Then, by Grönwall's inequality, it follows that
\[
\|\mathcal{S}(t, M_0, \mathfrak{p}) - \mathcal{S}(t, M_1, \mathfrak{p})\|_{H^1_\omega \times L^2_\omega} \leq \mathcal{C}(T) \, e^{T \mathcal{C}(T) \left( \|\mathfrak{p}\|_{L^1([0,T])} + \operatorname{Lip}(\sigma) \right)} \|M_0 - M_1\|_{H^1_\omega \times L^2_\omega},
\]
which establishes the Lipschitz continuity of the solution map with respect to the initial data.
\end{proof}

\subsection{Controlling the displacement}

\begin{prop}\label{prop:control_displacement}
For any initial and final displacements \( m_0, m_T \in H^1_\omega(\mathbb{R}^d) \) that are finite linear combinations of eigenfunctions \(\psi_{\mathbf{n}}\), and for any \( T > 0 \), there exist a time \(\tau \in [0, T)\) and an initial velocity \( \tilde{m}_0:=v_0 \in L^2_\omega(\mathbb{R}^d) \) such that the solution \( m \) of \eqref{eqn:ControlPDE_linear} with zero control \(\mathfrak{p} = 0\) and initial state \((m_0, v_0)\) satisfies
\[
m(\tau) = m_T
\]
in the space \( H^1_\omega(\mathbb{R}^d) \).
\end{prop}

\begin{proof}
Denote 
\[
S = \{ \mathbf{n}\in\mathbb{N}^d\setminus\{\mathbf{0}\} \mid \langle m_0, \psi_{\mathbf{n}} \rangle \neq 0 \text{ or } \langle m_T, \psi_{\mathbf{n}} \rangle \neq 0 \}.
\]

It follows from our assumption that \( |S| < \infty \), thus there exists a constant \(\mathcal{C}\) such that
\[
\mathcal{C} \geq \sup \{ \sqrt{\rho_{\mathbf{n}}} \mid \mathbf{n} \in S \}.
\]

Hence, there exists \(0 < \tau < \min \{ T, \frac{\pi}{\mathcal{C}} \}\) such that for all $\mathbf{n}\in S$,
\[
0 < \sqrt{\rho_{\mathbf{n}}} \tau < \pi,
\]
and thus \(\sin(\sqrt{\rho_{\mathbf{n}}} \tau) \neq 0\) for all \(\mathbf{n} \in S\).

We can then define
\[
v_0 = \left( \frac{\langle m_T, \psi_0 \rangle - \langle m_0, \psi_0 \rangle}{\tau} \right) \psi_0 + \sum_{\mathbf{n} \in S} \left( \frac{\langle m_T, \psi_{\mathbf{n}} \rangle}{\sin(\sqrt{\rho_{\mathbf{n}}} \tau) / \sqrt{\rho_{\mathbf{n}}}} - \frac{\langle m_0, \psi_{\mathbf{n}} \rangle \cos(\sqrt{\rho_{\mathbf{n}}} \tau)}{\sin(\sqrt{\rho_{\mathbf{n}}} \tau) / \sqrt{\rho_{\mathbf{n}}}} \right) \psi_{\mathbf{n}}.
\]

Then \(v_0 \in L^2_\omega(\mathbb{R}^d)\) and, by the explicit expressions of the solution \(m\) under zero control, given in \eqref{eqn:linear_displacement_evolution} and \eqref{eqn:linear_velocity_evolution}, we obtain
\[
m(\tau) = m_T
\]
by setting \(\tilde{m}_0 = v_0\) in \eqref{eqn:linear_displacement_evolution}.
\end{proof}
\begin{remark}\label{rmk:finite_combination}
It can be observed that the chosen velocity \(v_0\) and \(\partial_t m(\tau)\) for the solution \(m\) of \eqref{eqn:ControlPDE_linear} with zero control \(\mathfrak{p} = 0\) and initial state \((m_0, v_0)\) are also finite linear combinations of the eigenfunctions \(\psi_{\mathbf{n}}\).

\end{remark}

\subsection{Controlling the velocity}

\subsubsection{Saturating space}

We recall the following mapping from \(\mathbb{R}\) to \((-\pi,\pi)\) defined in \eqref{eqn:formula_for_phi}:
\[
\phi(x) = 2\sqrt{\pi} \int_0^x e^{-s^2} \, ds.
\]

Next, we define a mapping from \(\mathbb{R}^d\) to \((-\pi,\pi)^d\):
\[
\Phi(x_1, \ldots, x_d) = \big(\phi(x_1), \ldots, \phi(x_d)\big).
\]

We then define
\[
\mathcal{V}_0^\Phi := \mathrm{span}\big\{ 1, \cos( e_1 \cdot \Phi(x)), \sin( e_1 \cdot \Phi(x)), \dots, \cos( e_d \cdot \Phi(x)), \sin( e_d \cdot \Phi(x)) \big\}\quad \text{for } x\in\mathbb{R}^d,
\]
and, for each \(n \geq 1\), we define \(\mathcal{V}_n^\Phi\) iteratively as
\[
\mathcal{V}_n^\Phi := \Bigl\{ \psi_0 - \sum_{i=1}^N \psi_i^2 \;\big|\; \psi_0, \dots, \psi_N \in \mathcal{V}_{n-1}^\Phi, \; N \in \mathbb{N} \Bigr\}.
\]

Finally, the saturating space is defined as
\[
\mathcal{V}_\infty^\Phi := \bigcup_{n=0}^\infty \mathcal{V}_n^\Phi.
\]

Similarly, we can start with
\[
\mathcal{V}_0 := \mathrm{span}\big\{ 1, \cos( e_1 \cdot x), \sin( e_1 \cdot x), \dots, \cos( e_d \cdot x), \sin( e_d \cdot x) \big\} \quad \text{for } x \in (-\pi,\pi)^d,
\]
and construct the corresponding saturating space \(\mathcal{V}_\infty = \bigcup_{n=0}^\infty \mathcal{V}_n\), where
\[
\mathcal{V}_n := \Bigl\{ \psi_0 - \sum_{i=1}^N \psi_i^2 \;\big|\; \psi_0, \dots, \psi_N \in \mathcal{V}_{n-1}, \; N \in \mathbb{N} \Bigr\}.
\]

\begin{prop}\label{prop:saturating_space}
Given any set of vectors \( V = \{v_0, \ldots, v_n\} \) with \( v_0 = 0 \), we define the corresponding basis set for \( V \) as
\[
B(V) := \{1, \cos( v_1 \cdot x), \sin( v_1 \cdot x), \ldots, \cos( v_n \cdot x), \sin( v_n \cdot x)\}.
\]

Starting from \( V_0 = \{0, e_1, \ldots, e_d\} \), we inductively define
\[
V_{n+1} = \{ v_i \pm v_j \mid v_i, v_j \in V_n \}.
\]

Then, for all \( n \geq 0 \), we have:
\begin{enumerate}
	\item \label{saturatinglem:enm1} \(\mathcal{V}_n\) is a vector space.
	\item \label{saturatinglem:enm2} \(\mathrm{span}(B(V_n)) = \mathcal{V}_n\).
	\item \label{saturatinglem:enm3} The saturating space \( \mathcal{V}_\infty := \bigcup_{n = 0}^\infty \mathcal{V}_n \) satisfies
	\[
	\mathcal{V}_\infty = \mathrm{span} \left( \{ \cos(v \cdot x), \sin(v \cdot x) \mid v \in \mathbb{Z}^d \} \right),
	\]
	and is therefore dense in \( L^p((-\pi,\pi)^d) \) for all \( 1 \leq p < \infty \).
\end{enumerate}

\end{prop}
\begin{proof}
We first prove \eqref{saturatinglem:enm2} by induction.

It is clear that
\[
\mathrm{span}(B(V_0)) = \mathcal{V}_0.
\]

Suppose that \(\mathrm{span}(B(V_k)) = \mathcal{V}_k\) for some \(k \geq 0\). Note that \(V_k\) is a finite set, which we denote by
\[
V_k = \{v_0, \ldots, v_n\} \quad \text{with} \quad v_0 = 0.
\]

Then
\[
V_{k+1} = \{ v_i \pm v_j \mid v_i, v_j \in V_k \},
\]
and
\[
B(V_{k+1}) = \{1, \cos( (v_i \pm v_j) \cdot x), \sin( (v_i \pm v_j) \cdot x) \mid v_i, v_j \in V_k \}.
\]

Since \(-1 = 1- (\sqrt{2})^2\), and
\begin{align*}
	\pm \cos( (v_i + v_j) \cdot x) &= 1 - \frac{1}{2} \bigl( \cos( v_i \cdot x) \mp \cos( v_j \cdot x) \bigr)^2 - \frac{1}{2} \bigl( \sin( v_i \cdot x) \pm \sin( v_j \cdot x) \bigr)^2, \\
	\pm \cos( (v_i - v_j) \cdot x) &= 1 - \frac{1}{2} \bigl( \cos( v_i \cdot x) \mp \cos( v_j \cdot x) \bigr)^2 - \frac{1}{2} \bigl( \sin( v_i \cdot x) \mp \sin( v_j \cdot x) \bigr)^2, \\
	\pm \sin( (v_i + v_j) \cdot x) &= 1 - \frac{1}{2} \bigl( \sin( v_i \cdot x) \mp \cos( v_j \cdot x) \bigr)^2 - \frac{1}{2} \bigl( \cos( v_i \cdot x) \mp \sin( v_j \cdot x) \bigr)^2, \\
	\pm \sin( (v_i - v_j) \cdot x) &= 1 - \frac{1}{2} \bigl( \sin( v_i \cdot x) \mp \cos( v_j \cdot x) \bigr)^2 - \frac{1}{2} \bigl( \cos( v_i \cdot x) \pm \sin( v_j \cdot x) \bigr)^2.
\end{align*}

Using the inductive hypothesis \(\mathrm{span}(B(V_k)) = \mathcal{V}_k\), it follows that
\[
\pm 1, \quad \pm \cos( (v_i \pm v_j) \cdot x), \quad \pm \sin( (v_i \pm v_j) \cdot x) \in \mathcal{V}_{k+1}.
\]

Since \(\mathcal{V}_{k+1}\) is clearly closed under addition and multiplication by non-negative scalars, we have
\begin{equation}\label{saturatinglem:1}
	\mathrm{span}(B(V_{k+1})) \subset \mathcal{V}_{k+1}.
\end{equation}

By definition, it is also clear that
\begin{equation}\label{saturatinglem:2}
	\mathcal{V}_{k+1} \subset \mathrm{span}\big\{ \psi_i, \psi_j^2 \mid \psi_i, \psi_j \in \mathcal{V}_k \big\}.
\end{equation}

Since \(\mathrm{span}(B(V_k)) = \mathcal{V}_k\), denote
\[
B(V_k) = \{ \varphi_0, \ldots, \varphi_m \}
\]
(where each \(\varphi_i\) is either the constant function \(1\) or a sine/cosine function). Then for any \(\psi \in \mathcal{V}_k\), we have
\[
\psi = \sum_{\text{finite}} a_i \varphi_i, \quad \text{and} \quad \psi^2 = \left( \sum_{\text{finite}} a_i \varphi_i \right)^2,
\]
thus
\begin{equation}\label{saturatinglem:3}
	\psi, \psi^2 \in \mathrm{span} \big\{ \varphi_j, \varphi_j \varphi_k \mid \varphi_j, \varphi_k \in B(V_k) \big\}.
\end{equation}

Hence, by \eqref{saturatinglem:2}, it follows that
\begin{equation}\label{saturatinglem:4}
	\mathcal{V}_{k+1} \subset \mathrm{span} \big\{ \varphi_j, \varphi_j \varphi_k \mid \varphi_j, \varphi_k \in B(V_k) \big\}.
\end{equation}

By definition, it is clear that for any \(\varphi_j \in B(V_k)\),
\[
\varphi_j \in B(V_k) \subset\mathrm{span}(B(V_{k+1})).
\]

For any \(\varphi_j, \varphi_k \in B(V_k)\), there are three possible cases for the product \(\varphi_j \varphi_k\) (where we represent \(1\) as \(\cos( 0 \cdot x)\)):

\begin{enumerate}
	\item \(\varphi_j \varphi_k = \cos( v_l\cdot x) \cos( v_m\cdot x) = \frac{1}{2} \big[ \cos( (v_l + v_m) \cdot x) + \cos( (v_l - v_m) \cdot x) \big]\) for some \(v_l, v_m \in V_k\).
	\item \(\varphi_j \varphi_k = \sin( v_l\cdot x) \cos( v_m\cdot x) = \frac{1}{2} \big[ \sin( (v_l + v_m) \cdot x) + \sin( (v_l - v_m)\cdot x) \big]\) for some \(v_l, v_m \in V_k\).
	\item \(\varphi_j \varphi_k = \sin( v_l\cdot x) \sin( v_m\cdot x) = \frac{1}{2} \big[ \cos( (v_l + v_m)\cdot x) - \cos( (v_l - v_m)\cdot x) \big]\) for some \(v_l, v_m \in V_k\).
\end{enumerate}

In all cases, we have \(\varphi_j \varphi_k \in \mathrm{span}(B(V_{k+1}))\).

Thus,
\begin{equation}\label{saturatinglem:5}
	\mathrm{span}\{\varphi_j, \varphi_j \varphi_k \mid \varphi_j, \varphi_k \in B(V_k)\} \subset \mathrm{span}(B(V_{k+1})).
\end{equation}

Then, by \eqref{saturatinglem:1}, \eqref{saturatinglem:4}, and \eqref{saturatinglem:5}, we have
\[
\mathrm{span}(B(V_{k+1})) \subset \mathcal{V}_{k+1} \subset \mathrm{span}\{\varphi_j, \varphi_j \varphi_k \mid \varphi_j, \varphi_k \in B(V_k)\} \subset \mathrm{span}(B(V_{k+1})).
\]

Hence,
\[
\mathrm{span}(B(V_{k+1})) = \mathcal{V}_{k+1}.
\]

Since we already know that \(\mathrm{span}(B(V_0)) = \mathcal{V}_0\), it follows by induction that
\begin{equation}\label{saturatinglem:6}
	\mathrm{span}(B(V_n)) = \mathcal{V}_n \quad \text{for all } n \geq 0.
\end{equation}

Moreover, \eqref{saturatinglem:enm1} follows directly from \eqref{saturatinglem:6}.

It follows from the construction of the sets \(V_n\) that
\[
\bigcup_{n=0}^\infty V_n = \mathbb{Z}^d.
\]

Then \eqref{saturatinglem:6} implies
\[
\mathcal{V}_\infty = \mathrm{span} \left( \{ \cos(v \cdot x), \sin(v \cdot x) \mid v \in \mathbb{Z}^d \} \right),
\]
thus \(\mathcal{V}_\infty\) is dense in \(L^p((-\pi,\pi)^d)\) for all \(1 \leq p < \infty\), which establishes \eqref{saturatinglem:enm3}.
\end{proof}
\begin{prop}\label{prop:saturating_space_Phi}
Given a set of vectors \( V = \{v_0, \ldots, v_n\} \) with \( v_0 = 0 \), we define the corresponding \(\Phi\)-basis set for \( V \) as
\[
B^\Phi(V) := \{1, \cos( v_1 \cdot \Phi(x)), \sin( v_1 \cdot \Phi(x)), \ldots, \cos( v_n \cdot \Phi(x)), \sin( v_n \cdot \Phi(x)) \}.
\]

Starting with \( V_0 = \{0, e_1, \ldots, e_d\} \), we inductively define
\[
V_{n+1} = \{ v_i \pm v_j \mid v_i, v_j \in V_n \}.
\]

Then, for all \( n \geq 0 \), the following hold:
\begin{enumerate}
	\item \label{saturatinglem_Phi:enm1} \(\mathcal{V}_n^\Phi\) is a vector space.
	\item \label{saturatinglem_Phi:enm2} \(\mathrm{span}(B^\Phi(V_n)) = \mathcal{V}_n^\Phi\).
	\item \label{saturatinglem_Phi:enm3} The saturating space \(\mathcal{V}_\infty^\Phi := \bigcup_{n=0}^\infty \mathcal{V}_n^\Phi\) satisfies
	\[
	\mathcal{V}_\infty^\Phi = \mathrm{span} \left( \{ \cos(v \cdot \Phi(x)), \sin(v \cdot \Phi(x)) \mid v \in \mathbb{Z}^d \} \right),
	\]
	and \(\mathcal{V}_\infty^\Phi\) is dense in \(L^p_\omega(\mathbb{R}^d)\) for all \( 1 \leq p < \infty \).
	\item \label{saturatinglem_Phi:enm4} \(\mathcal{V}_n^\Phi \subset W^{1,\infty}(\mathbb{R}^d)\).
\end{enumerate}

\end{prop}
\begin{proof}
Statements \eqref{saturatinglem_Phi:enm1} and \eqref{saturatinglem_Phi:enm2} follow from the same proof as \eqref{saturatinglem:enm1} and \eqref{saturatinglem:enm2} in Proposition \ref{prop:saturating_space}.

For \eqref{saturatinglem_Phi:enm3}, combining \eqref{saturatinglem_Phi:enm2} with \(\bigcup_{n=0}^\infty V_n = \mathbb{Z}^d\) implies
\[
\mathcal{V}_\infty^\Phi = \mathrm{span} \left( \{ \cos(v \cdot \Phi(x)), \sin(v \cdot \Phi(x)) \mid v \in \mathbb{Z}^d \} \right),
\]

For the density in \(L^p_\omega(\mathbb{R}^d)\), we argue as follows:

For any \(1 \le p < \infty\) and any \(f \in L^p_\omega(\mathbb{R}^d)\), define \(\tilde{f}: (-\pi,\pi)^d \to \mathbb{R}\) by
\[
\tilde{f}(y) = f(\Phi^{-1}(y)).
\]

Then
\begin{align*}
	\int_{(-\pi,\pi)^d} |\tilde{f}(y)|^p \, dy 
	&= \int_{(-\pi,\pi)^d} |f(\Phi^{-1}(y))|^p \, dy \\
	&= \int_{\mathbb{R}^d} |f(x)|^p J(\Phi)(x) \, dx \\
	&= \left(2 \sqrt{\pi}\right)^d \int_{\mathbb{R}^d} |f(x)|^p e^{-|x|^2} \, dx < \infty,
\end{align*}
where \(J(\Phi)(x) = \det \nabla \Phi(x) = \left(2 \sqrt{\pi}\right)^d e^{-|x|^2}\).

Thus, \(\tilde{f} \in L^p((-\pi,\pi)^d)\). By \eqref{saturatinglem_Phi:enm3} in Proposition \ref{prop:saturating_space}, for any \(\varepsilon > 0\), there exists \(\tilde{\psi} \in \mathcal{V}_N\) for some \(N\) such that
\[
\|\tilde{f} - \tilde{\psi}\|_{L^p((-\pi,\pi)^d)} < \varepsilon.
\]

By construction, there exists \(\psi \in \mathcal{V}_N^\Phi\) such that \(\tilde{\psi} = \psi \circ \Phi^{-1}\), and moreover,
\[
\|f - \psi\|_{L^p_\omega(\mathbb{R}^d)} < \left(\frac{1}{2 \sqrt{\pi}}\right)^d \varepsilon.
\]

This shows that \(\mathcal{V}_\infty^\Phi\) is dense in \(L^p_\omega(\mathbb{R}^d)\).

The statement \eqref{saturatinglem_Phi:enm3} implies that \(\mathcal{V}_n^\Phi\) is spanned by functions of the form \(\cos( v \cdot \Phi(x))\) and \(\sin( v \cdot \Phi(x))\) for \(v\in\mathbb{Z}^d\). Since these functions have bounded derivatives, it readily follows that
\[
\mathcal{V}_n^\Phi \subset W^{1,\infty}(\mathbb{R}^d),
\]
which establishes \eqref{saturatinglem_Phi:enm4}.
\end{proof}

\subsubsection{Elementary control steps}

In this subsection, we establish certain element control estimates using the Feynman (Dyson) expansion techniques commonly employed in mathematical physics \cite{erdHos2007quantum,erdHos2008quantum,LukkarinenSpohn:KLW:2007,LukkarinenSpohn:WNS:2011,POZZOLI2024421}.

\begin{prop}\label{prop:lemma_1}
Given any \((m_0,\tilde{m}_0)\in H^1_\omega(\mathbb{R}^d)\times L^2_\omega(\mathbb{R}^d)\) and \(\zeta \in L^\infty(\mathbb{R}^d)\), we have
\begin{equation}\label{eqn:limit_order_1}
	\left\|  e^{\tau \left(\mathcal{X}+\frac{\zeta}{\tau}\mathcal{Y}\right)}\begin{pmatrix}
		m_0\\
		\tilde{m}_0
	\end{pmatrix}
	-
	\begin{pmatrix}
		m_0\\
		\tilde{m}_0+\zeta m_0
	\end{pmatrix} \right\|_{H^1_\omega(\mathbb{R}^d)\times L^2_\omega(\mathbb{R}^d)} \rightarrow 0 \quad\text{as } \tau \rightarrow 0.
\end{equation}

\end{prop}
\begin{proof}
By the bounded perturbation theorem (see \cite[Ch.~3, Thm.~1.1]{pazy}), since \(\frac{\zeta}{\tau}\mathcal{Y}\) is clearly bounded for all nonzero \(\tau \in \mathbb{R}\), and \(\mathcal{X}\) generates a \(C_{0}\)-group \(e^{t\mathcal{X}}\), it follows that the operator
\[
\mathcal{X} + \frac{\zeta}{\tau}\mathcal{Y}
\]
also generates a \(C_0\)-group \(e^{t(\mathcal{X} + \frac{\zeta}{\tau}\mathcal{Y})}\) on 
\(H^1_\omega(\mathbb{R}^d) \times L^2_\omega(\mathbb{R}^d)\).

By setting  
\[
V_0(t) = e^{t\mathcal{X}},
\]  
and defining \( V_n(t) \) inductively for all \( n \ge 0 \) as  
\[
V_{n+1}(t) = e^{t\mathcal{X}} \int_0^t e^{-t_n\mathcal{X}} \frac{\zeta}{\tau} \mathcal{Y} V_n(t_n)\, dt_n,
\]  
we obtain the following explicit formula for $n\ge 1$:  
\[
V_n(t) = e^{t\mathcal{X}} \int_{[0,\infty)^{n+1}} \delta\left(t = \sum_{i=0}^n s_i \right) \left( \prod_{i=0}^{n-1} e^{-\sum_{j=0}^i s_j \mathcal{X}} \frac{\zeta}{\tau} \mathcal{Y} e^{\sum_{j=0}^i s_j \mathcal{X}} \right) ds_0 \dots ds_n.
\]

It follows from the strong continuity of the group generated by \(\mathcal{X}\) and the boundedness of \(\frac{\zeta}{\tau}\mathcal{Y}\) (see \cite[p.~77]{pazy}) that  
\[
e^{t\left(\mathcal{X}+\frac{\zeta}{\tau}\mathcal{Y}\right)} = \sum_{n=0}^\infty V_n(t)
\]
is a convergent series in the uniform operator topology.

We then prove the following formula for \(V_n\) by induction:
\begin{equation}\label{eqn:dyson_induction}
	V_{n}(t)=e^{t\mathcal{X}}\frac{1}{n!}\left(\int_{0}^t  e^{-s\mathcal{X}}\frac{\zeta}{\tau}\mathcal{Y}e^{s\mathcal{X}} ds\right)^n,
\end{equation}
which clearly holds when \(n=0\).

Suppose that \eqref{eqn:dyson_induction} holds for \(n = k\). Then
\[
V_{k+1}(t) = e^{t\mathcal{X}} \int_0^t e^{-t_k \mathcal{X}} \frac{\zeta}{\tau} \mathcal{Y} \left( e^{t_k \mathcal{X}} \frac{1}{k!} \left( \int_0^{t_k} e^{-s \mathcal{X}} \frac{\zeta}{\tau} \mathcal{Y} e^{s \mathcal{X}} ds \right)^k \right) dt_k.
\]

With the change of variable
\[
u = \int_0^{t_k} e^{-s\mathcal{X}} \frac{\zeta}{\tau} \mathcal{Y} e^{s\mathcal{X}} \, ds,
\]
we have
\[
V_{k+1}(t) = e^{t\mathcal{X}} \frac{1}{k!} \int_0^t u^k \, du \,   = e^{t\mathcal{X}} \frac{1}{(k+1)!} \left( \int_0^t e^{-s\mathcal{X}} \frac{\zeta}{\tau} \mathcal{Y} e^{s\mathcal{X}} \, ds \right)^{k+1}.
\]

By induction, this implies \eqref{eqn:dyson_induction}.

For a given \((m_0,\tilde{m}_0)\in H^1_\omega(\mathbb{R}^d)\times L^2_\omega(\mathbb{R}^d)\), we know that
\[
\left\| \int_{0}^\tau  e^{-s\mathcal{X}}\frac{\zeta}{\tau}\mathcal{Y}e^{s\mathcal{X}}  ds \begin{pmatrix}
	m_0\\
	\tilde{m}_0
\end{pmatrix}-\zeta \mathcal{Y} \begin{pmatrix}
	m_0\\
	\tilde{m}_0
\end{pmatrix}\right\|_{H^1_\omega\times L^2_\omega}\rightarrow 0 \quad\text{as } \tau\rightarrow 0.
\]

We can choose a sufficiently small \( T_0 > 0 \) such that there exists a constant \(\mathcal{C}\) satisfying
\[
\mathcal{C} \geq \sup_{\tau < T_0} \left\{ \left\| e^{\tau \mathcal{X}} \right\|_{H^1_\omega \times L^2_\omega \to H^1_\omega \times L^2_\omega}, \quad \left\| \int_0^\tau e^{-s \mathcal{X}} \frac{\zeta}{\tau} \mathcal{Y} e^{s \mathcal{X}} \, ds \right\|_{H^1_\omega \times L^2_\omega \to H^1_\omega \times L^2_\omega}, \quad \left\| \zeta \mathcal{Y} \right\|_{H^1_\omega \times L^2_\omega \to H^1_\omega \times L^2_\omega} \right\}.
\]

Then by \eqref{eqn:dyson_induction},
\begin{equation}\label{eqn:dyson_limit_1}
	\begin{aligned}
		&\left\| \sum_{n=0}^\infty V_n(\tau)\begin{pmatrix}
			m_0\\
			\tilde{m}_0
		\end{pmatrix} - e^{\tau\mathcal{X}}\sum_{n=0}^\infty \frac{(\zeta \mathcal{Y})^n}{n!}\begin{pmatrix}
			m_0\\
			\tilde{m}_0
		\end{pmatrix} \right\|_{H^1_\omega \times L^2_\omega} \\
		&\quad \leq \mathcal{C} \sum_{n=1}^\infty \frac{1}{n!} \left\| \left( \int_0^\tau e^{-s \mathcal{X}} \frac{\zeta}{\tau} \mathcal{Y} e^{s \mathcal{X}} ds \right)^n \begin{pmatrix}
			m_0\\
			\tilde{m}_0
		\end{pmatrix} - (\zeta \mathcal{Y})^n \begin{pmatrix}
			m_0\\
			\tilde{m}_0
		\end{pmatrix}\right\|_{H^1_\omega \times L^2_\omega} \\
		&\quad \leq \mathcal{C} \sum_{n=1}^\infty \frac{n \mathcal{C}^{n-1}}{n!} \left\| \int_0^\tau e^{-s \mathcal{X}} \frac{\zeta}{\tau} \mathcal{Y} e^{s \mathcal{X}} ds \begin{pmatrix}
			m_0\\
			\tilde{m}_0
		\end{pmatrix} - \zeta \mathcal{Y} \begin{pmatrix}
			m_0\\
			\tilde{m}_0
		\end{pmatrix} \right\|_{H^1_\omega \times L^2_\omega} \\
		&\quad \to 0 \quad \text{as } \tau \to 0.
	\end{aligned}
\end{equation}

Note that
\begin{equation}\label{eqn:dyson_formula}
	e^{t(\mathcal{X}+\frac{\zeta}{\tau}\mathcal{Y})} = \sum_{n=0}^\infty V_n(t), \quad \sum_{n=0}^\infty \frac{(\zeta \mathcal{Y})^n}{n!} = I + \zeta \mathcal{Y} = \begin{pmatrix}
		I & 0 \\
		\zeta & I
	\end{pmatrix},
\end{equation}
by $\mathcal{Y}^2 = 0$.

Then by \eqref{eqn:dyson_limit_1} and \eqref{eqn:dyson_formula}, we have
\begin{align*}
	&\left\| e^{\tau(\mathcal{X}+\frac{\zeta}{\tau}\mathcal{Y})}\begin{pmatrix}
		m_0\\
		\tilde{m}_0
	\end{pmatrix} - \begin{pmatrix}
		m_0\\
		\tilde{m}_0+\zeta m_0
	\end{pmatrix} \right\|_{H^1_\omega \times L^2_\omega}
	\le  \left\| e^{\tau(\mathcal{X}+\frac{\zeta}{\tau}\mathcal{Y})}\begin{pmatrix}
		m_0\\
		\tilde{m}_0
	\end{pmatrix} - e^{\tau\mathcal{X}}\begin{pmatrix}
		I & 0 \\
		\zeta & I
	\end{pmatrix} \begin{pmatrix}
		m_0 \\
		\tilde{m}_0
	\end{pmatrix} \right\|_{H^1_\omega \times L^2_\omega}\\
	&\quad + \left\| e^{\tau\mathcal{X}}\begin{pmatrix}
		m_0\\
		\tilde{m}_0+\zeta m_0
	\end{pmatrix} -\begin{pmatrix}
		m_0\\
		\tilde{m}_0+\zeta m_0
	\end{pmatrix}  \right\|_{H^1_\omega \times L^2_\omega} \to\, 0 \quad \text{as } \tau \to 0.
\end{align*}
\end{proof}
\begin{prop}
Given any \((m_0, \tilde{m}_0) \in H^1_\omega(\mathbb{R}^d) \times L^2_\omega(\mathbb{R}^d)\) and \(\varphi \in W^{1,\infty}(\mathbb{R}^d)\), we have
\begin{equation}\label{eqn:limit_order_2}
	\left\| e^{-\tau^{-1/2}\varphi \mathcal{Y}} \, e^{\tau \mathcal{X}} \, e^{\tau^{-1/2}\varphi \mathcal{Y}} \begin{pmatrix}
		m_0 \\
		\tilde{m}_0
	\end{pmatrix} - \begin{pmatrix}
		m_0 \\
		\tilde{m}_0 - \varphi^2 m_0
	\end{pmatrix} \right\|_{H^1_\omega(\mathbb{R}^d) \times L^2_\omega(\mathbb{R}^d)} \rightarrow 0 \quad \text{as } \tau \rightarrow 0.
\end{equation}

\end{prop}
\begin{proof}
By the similarity invariance of the matrix exponential (see, e.g., \cite{HornJohnson85} and  \cite{Higham08}), we have
\[
e^{-\tau^{-1/2}\varphi \mathcal{Y}} e^{t \mathcal{X}} e^{\tau^{-1/2}\varphi \mathcal{Y}} = \exp\left( e^{-\tau^{-1/2}\varphi \mathcal{Y}} \, t \mathcal{X} \, e^{\tau^{-1/2}\varphi \mathcal{Y}} \right).
\]

Using the Hadamard formula (see, e.g., \cite{Hall15} and \cite{Simon05}), we have
\[
e^{-\tau^{-1/2} \varphi \mathcal{Y}} \mathcal{X} e^{\tau^{-1/2} \varphi \mathcal{Y}} = \mathcal{X} + \frac{\varphi}{\tau^{1/2}} [\mathcal{X}, \mathcal{Y}] + \frac{\varphi^2}{2\tau} [[\mathcal{X}, \mathcal{Y}], \mathcal{Y}] + \frac{\varphi^3}{6\tau^{3/2}} [[[ \mathcal{X}, \mathcal{Y}], \mathcal{Y}], \mathcal{Y}] + \cdots.
\]

Here, since \(\mathcal{Y}^2 = 0\), the triple commutator \([[[\mathcal{X}, \mathcal{Y}], \mathcal{Y}], \mathcal{Y}]\) and all higher-order terms vanish. Thus, we obtain
\[
e^{-\tau^{-1/2}\varphi \mathcal{Y}} \mathcal{X} e^{\tau^{-1/2}\varphi \mathcal{Y}} = \mathcal{X} + \frac{\varphi}{\tau^{1/2}}[\mathcal{X}, \mathcal{Y}] + \frac{\varphi^2}{2\tau}[[\mathcal{X}, \mathcal{Y}], \mathcal{Y}],
\]
where it can be computed that
\[
[\mathcal{X}, \mathcal{Y}] = \begin{pmatrix}
	I & 0\\
	0 & -I
\end{pmatrix}, \quad 
[[\mathcal{X}, \mathcal{Y}], \mathcal{Y}] = \begin{pmatrix}
	0 & 0\\
	-2I & 0
\end{pmatrix}.
\]

Since \(\varphi, \nabla \varphi \in L^\infty(\mathbb{R}^d)\), the operators \(\frac{\varphi}{\tau^{1/2}} [\mathcal{X}, \mathcal{Y}]\) and \(\frac{\varphi^2}{2\tau} [[\mathcal{X}, \mathcal{Y}], \mathcal{Y}]\) are bounded for all nonzero $\tau\in\mathbb{R}$.

Proceeding as in Proposition \ref{prop:lemma_1}, we have
\[
\exp\left(e^{-\tau^{-1/2}\varphi \mathcal{Y}}\, t\mathcal{X}\, e^{\tau^{-1/2}\varphi \mathcal{Y}}\right)
= U(t) = \sum_{n=0}^\infty U_n(t),
\]
where $U_0(t)= e^{t\mathcal{X}}$, and for $n\ge1$,
\[
U_n(t)
= e^{t\mathcal{X}} \int_{[0,\infty)^{n+1}} \delta\left(t = \sum_{i=0}^n s_i\right)
\left( \prod_{i=0}^{n-1} 
e^{-\sum\limits_{j=0}^i s_j \mathcal{X}} 
\left( \frac{\varphi}{\tau^{1/2}}[\mathcal{X}, \mathcal{Y}]
+ \frac{\varphi^2}{2\tau} [[\mathcal{X}, \mathcal{Y}], \mathcal{Y}] \right)
e^{\sum\limits_{j=0}^i s_j \mathcal{X}} 
\right) ds_0 \dots ds_n.
\]

Moreover, for all \( n \), we decompose \( U_n \) as follows:
\[
U_n(t) = D_n(t) + R_n(t),
\]
where $D_0= e^{t\mathcal{X}}, R_0(t)=0$, and for $n\ge 1$,
\[
D_n(t)
= e^{t\mathcal{X}} \int_{[0,\infty)^{n+1}} \delta\left(t = \sum_{i=0}^n s_i\right)
\left( \prod_{i=0}^{n-1} 
e^{-\sum_{j=0}^i s_j \mathcal{X}} 
\frac{\varphi^2}{2\tau} [[\mathcal{X}, \mathcal{Y}], \mathcal{Y}]
e^{\sum_{j=0}^i s_j \mathcal{X}} 
\right) ds_0 \dots ds_n,
\]
while \( R_n(t) \) collects all the remaining terms involving \(\tfrac{\varphi}{\tau^{1/2}}[\mathcal{X}, \mathcal{Y}]\) in the \(n\)-th product.

Similarly to \eqref{eqn:dyson_induction}, we have for all \( n \),
\[
D_n(t)
= e^{t\mathcal{X}} \frac{1}{n!} \left( \int_0^t e^{-s\mathcal{X}} \frac{\varphi^2}{2\tau} [[\mathcal{X}, \mathcal{Y}], \mathcal{Y}] e^{s\mathcal{X}} \, ds \right)^n,
\]
and as in Proposition \ref{prop:lemma_1}, we obtain that for any given \((m_0, \tilde{m}_0) \in H^1_\omega(\mathbb{R}^d) \times L^2_\omega(\mathbb{R}^d)\), as \( \tau \to 0 \),
\[
\sum_{n=0}^\infty D_n(\tau)\begin{pmatrix}
	m_0\\\tilde{m}_0
\end{pmatrix} \rightarrow \sum_{n=0}^\infty \frac{ \left( \frac{\varphi^2}{2} [[\mathcal{X}, \mathcal{Y}], \mathcal{Y}] \right)^n }{n!}\begin{pmatrix}
	m_0\\\tilde{m}_0
\end{pmatrix} .
\]

Since
\[
[[\mathcal{X}, \mathcal{Y}], \mathcal{Y}]^2 = 
\begin{pmatrix}
	0 & 0 \\
	-2I & 0
\end{pmatrix}^2 = 0,
\]
we conclude that
\[
\sum_{n=0}^\infty D_n(\tau)\begin{pmatrix}
	m_0\\\tilde{m}_0
\end{pmatrix}  \rightarrow I + \frac{\varphi^2}{2} [[\mathcal{X}, \mathcal{Y}], \mathcal{Y}]\begin{pmatrix}
	m_0\\\tilde{m}_0
\end{pmatrix} 
= \begin{pmatrix}
	I & 0 \\
	-\varphi^2 & I
\end{pmatrix}\begin{pmatrix}
	m_0\\\tilde{m}_0
\end{pmatrix} .
\]

For \( R_n(t) \), observe that for some sufficiently small \( 1>T_0 > 0 \), if \( \sum_{j=0}^i s_j < T_0 \), there exists a constant \( \mathcal{C} > 0 \) such that
\[
\left\| e^{-\sum_{j=0}^i s_j \mathcal{X}}\, \varphi [\mathcal{X}, \mathcal{Y}]\, e^{\sum_{j=0}^i s_j \mathcal{X}} \right\|_{H^1_\omega \times L^2_\omega \rightarrow H^1_\omega \times L^2_\omega} \le \mathcal{C},
\]
\[
\left\| e^{-\sum_{j=0}^i s_j \mathcal{X}}\, \frac{\varphi^2}{2} [[\mathcal{X}, \mathcal{Y}], \mathcal{Y}]\, e^{\sum_{j=0}^i s_j \mathcal{X}} \right\|_{H^1_\omega \times L^2_\omega \rightarrow H^1_\omega \times L^2_\omega} \le \mathcal{C},
\]
and for all \( t < T_0 \),
\[
\left\| e^{t \mathcal{X}} \right\|_{H^1_\omega \times L^2_\omega \rightarrow H^1_\omega \times L^2_\omega} \le \mathcal{C}.
\]

Thus, for \( \tau < T_0 \), for $n\ge 1$,
\[
\left\| R_n(\tau) \right\|_{H^1_\omega \times L^2_\omega \rightarrow H^1_\omega \times L^2_\omega}
\le \mathcal{C} \left( \sum_{i=1}^n \binom{n}{i} \frac{\tau^n}{n!} \frac{\mathcal{C}^n}{\tau^{n - \frac{i}{2}}} \right)
= \mathcal{C} \left( \sum_{i=1}^n \frac{\mathcal{C}^n}{i!(n-i)!} \tau^{\frac{i}{2}} \right).
\]

Therefore,
\begin{align*}
	\left\| \sum_{n=0}^\infty R_n(\tau) \right\|_{H^1_\omega \times L^2_\omega \rightarrow H^1_\omega \times L^2_\omega}
	&\le \mathcal{C} \sum_{n=1}^\infty \left( \sum_{i=1}^n \frac{\mathcal{C}^n}{i!(n-i)!} \tau^{\frac{i}{2}} \right) \\
	&= \mathcal{C} \sum_{i=1}^\infty \frac{\mathcal{C}^i}{i!} \left( \sum_{n=i}^\infty \frac{\mathcal{C}^{n-i}}{(n-i)!} \right) \tau^{\frac{i}{2}} \\
	&<  \mathcal{C} e^{\mathcal{C}} \left( e^{\mathcal{C}\tau^{1/2}} -1\right),
\end{align*}
which converges to 0 as \( \tau \to 0 \).

Hence, as \( \tau \to 0 \), we have
\[
e^{-\tau^{-1/2}\varphi \mathcal{Y}} e^{\tau \mathcal{X}} e^{\tau^{-1/2}\varphi \mathcal{Y}} \begin{pmatrix}
	m_0 \\
	\tilde{m}_0
\end{pmatrix}
\rightarrow
\begin{pmatrix}
	I & 0 \\
	-\varphi^2 & I
\end{pmatrix}
\begin{pmatrix}
	m_0 \\
	\tilde{m}_0
\end{pmatrix}
=
\begin{pmatrix}
	m_0 \\
	\tilde{m}_0 - \varphi^2 m_0
\end{pmatrix}.
\]
\end{proof}

\subsubsection{A Proposition on Velocity Controllability}

To simplify the notation, we first define the \emph{concatenation} \(\mathfrak{p}' \diamond \mathfrak{p}\) of two control laws \(\mathfrak{p}: [0,T_1] \to \mathbb{R}^{2d+1}\) and \(\mathfrak{p}': [0,T_2] \to \mathbb{R}^{2d+1}\), as a function on the interval \([0, T_1 + T_2]\) given by
\begin{equation}\label{concatenation}
(\mathfrak{p}' \diamond \mathfrak{p})(t) =
\begin{cases}
	\mathfrak{p}(t), & t \in [0, T_1], \\
	\mathfrak{p}'(t - T_1), & t \in (T_1, T_1 + T_2].
\end{cases}
\end{equation}

\begin{prop}\label{prop:control_velocity}
For any initial state \((m_0, \tilde{m}_0) \in H^1_\omega(\mathbb{R}^d) \times \bigl(L^2_\omega(\mathbb{R}^d) \cap L^p_{\mathrm{loc}}(\mathbb{R}^d)\bigr)\) satisfying Assumption~\ref{asp:finitecombination}, and for any target velocity \(\tilde{m}_T := v \in L^2_\omega(\mathbb{R}^d) \cap L^p_{\mathrm{loc}}(\mathbb{R}^d)\) with some \(p > 2\), as well as any \(\varepsilon, T > 0\), there exist \(\tau \in [0, T)\) and a piecewise constant control law \(\mathfrak{p} : [0, \tau] \to \mathbb{R}^{2d+1}\) such that the solution \(m\) to \eqref{eqn:ControlPDE_linear}, under the control defined by \eqref{eqn:formula_theta} and \eqref{eqn:set_theta} and starting from \((m_0, \tilde{m}_0)\), satisfies
\[
\left\| \bigl(m(\cdot, \tau), \partial_t m(\cdot, \tau)\bigr) - (m_0, v) \right\|_{H^1_\omega(\mathbb{R}^d) \times L^2_\omega(\mathbb{R}^d)} < \varepsilon.
\]

\end{prop}
\begin{proof}

\textbf{Step 1. Substitute $v$ with $\tilde{m}_0+\psi_\eta m_0$.}

For a given \(\varepsilon > 0\), there exists a closed bounded ball \(\mathcal{B} \subset \mathbb{R}^d\) such that
\[
\left\| \tilde{m}_0 \chi_{\mathcal{B}} \right\|_{L^2_\omega(\mathbb{R}^d)} + \left\| v \chi_{\mathcal{B}} \right\|_{L^2_\omega(\mathbb{R}^d)} < \frac{\varepsilon}{4}.
\]

Define
\[
N(m_0) := \{ z \in \mathcal{B} \mid m_0(z) = 0 \},
\]
and let \(N_\eta(m_0)\) be the \(\eta\)-neighborhood of \(N(m_0)\) inside \(\mathcal{B}\):
\[
N_\eta(m_0) := \{ x \in \mathcal{B} \mid \mathrm{dist}(x, N(m_0)) < \eta \}.
\]

Then set
\begin{equation}
	\psi_\eta := \frac{v - \tilde{m}_0}{m_0} \chi_{\mathcal{B} \setminus N_\eta(m_0)},
\end{equation}
which is well-defined since the continuous, nonzero function \(m_0\) has a positive lower bound on the compact set \(\mathcal{B} \setminus N_\eta(m_0)\).

One has
\begin{align*}
	\left\| (\tilde{m}_0 + \psi_\eta m_0) - v \right\|_{L^2_\omega(\mathbb{R}^d)}
	&= \left\| \tilde{m}_0 - v + (v - \tilde{m}_0) \chi_{\mathcal{B} \setminus N_\eta(m_0)} \right\|_{L^2_\omega(\mathbb{R}^d)} \\
	&\leq \frac{\varepsilon}{4} + \left\| \tilde{m}_0 \right\|_{L^2_\omega(N_\eta(m_0))} + \left\| v \right\|_{L^2_\omega(N_\eta(m_0))}.
\end{align*}

Since \(m_0\) is a finite nonzero linear combination of eigenfunctions, it is analytic, which implies that the zero set \(N(m_0)\) satisfies \(|N(m_0)| = 0\). Moreover, continuity of \(m_0\) implies that 
\[
N(m_0) = m_0^{-1}(\{0\})
\]
is a closed set. Hence,
\[
\lim_{\eta \to 0} |N_\eta(m_0)| = \left| \bigcap_{\eta > 0} N_\eta(m_0) \right| = |\overline{N(m_0)}| = |N(m_0)| = 0.
\]

Thus,
\[
\left\| \tilde{m}_0 \right\|_{L^2_\omega(N_\eta(m_0))} + \left\| v \right\|_{L^2_\omega(N_\eta(m_0))} \to 0 \quad \text{as } \eta \to 0,
\]
and therefore there exists a sufficiently small \(\eta > 0\) such that
\[
\left\| \tilde{m}_0 \right\|_{L^2_\omega(N_\eta(m_0))} + \left\| v \right\|_{L^2_\omega(N_\eta(m_0))} < \frac{\varepsilon}{4}.
\]

Hence, for the chosen \(\eta\), we have
\begin{equation}\label{eqn:error_v_mphi}
	\left\| (\tilde{m}_0 + \psi_\eta m_0) - v \right\|_{L^2_\omega(\mathbb{R}^d)} < \frac{\varepsilon}{2}.
\end{equation}

\textbf{Step 2. Use induction to show the controllability for $\tilde{m}_0+\psi m_0$ with $\psi\in\mathcal{V}_\infty^\Phi$.}

Then, we prove by induction that the following property holds for all \(n \geq 0\), relying on \eqref{eqn:limit_order_1} and \eqref{eqn:limit_order_2}.

\textbf{Property \((\mathfrak{I}_n)\):} For any initial state \((m_0, \tilde{m}_0) \in H^1_\omega(\mathbb{R}^d) \times L^2_\omega(\mathbb{R}^d)\), any \(\psi \in \mathcal{V}_n^\Phi\), and any \(\varepsilon, T > 0\), there exist \(\tau \in [0,T)\) and a piecewise constant control \(\mathfrak{p} : [0,\tau] \to \mathbb{R}^{2d+1}\) such that the solution \(m\) of \eqref{eqn:ControlPDE_linear} under the control law specified by \eqref{eqn:formula_theta} and \eqref{eqn:set_theta}, with initial state \((m_0, \tilde{m}_0)\), satisfies
\[
\left\| \bigl(m(\cdot,\tau), \tilde{m}(\cdot,\tau)\bigr) - \bigl(m_0, \tilde{m}_0 + \psi m_0\bigr) \right\|_{H^1_\omega(\mathbb{R}^d) \times L^2_\omega(\mathbb{R}^d)} < \varepsilon.
\]

Note that, since \(\mathcal{Y}^2 = 0\), the term \(\bigl(m_0, \tilde{m}_0 + \psi m_0\bigr)\) can be interpreted as the result of applying the operator \(e^{\psi \mathcal{Y}}\) to the initial state:
\begin{equation}\label{eqn:control_for_e_phi_X}
	e^{\psi \mathcal{Y}} \begin{pmatrix} m_0 \\ \tilde{m}_0 \end{pmatrix} 
	= \left(I+\psi \mathcal{Y}\right)\begin{pmatrix} m_0 \\ \tilde{m}_0 \end{pmatrix}= \begin{pmatrix} m_0 \\ \tilde{m}_0 + \psi m_0 \end{pmatrix}.
\end{equation}

$(\mathfrak{I}_0):$  Given any $\psi \in \mathcal{V}_0^\Phi$, since by definition, $$\mathcal{V}_0^\Phi=\mathop{\mathrm{span}}\limits_{{0\le j \le 2d}}\{\theta_j\}.$$ 

Thus  there exists $\mathfrak{p} = (\mathfrak{p}_0, \dots, \mathfrak{p}_{2d}) \in \mathbb{R}^{2d+1} \text{ such that }$
\[
\psi(x) = -\sum_{j=0}^{2d} \mathfrak{p}_j \theta_j(x).
\]

Then, for any \(\tau > 0\), we have
\[
\frac{\psi(x)}{\tau} = -\sum_{j=0}^{2d} \frac{\mathfrak{p}_j}{\tau} \theta_j(x),
\]
and therefore 
\[
\mathcal{S}\left(t, \begin{pmatrix} m_0 \\ \tilde{m}_0 \end{pmatrix}, \frac{1}{\tau} \mathfrak{p}\right)
\]
is the solution of the following Cauchy problem:
\[
\frac{\partial}{\partial t} M(x,t) = \left(\mathcal{X} +\frac{\psi}{\tau} \mathcal{Y}\right) M(x,t), \quad M_0 = (m_0, \tilde{m}_0).
\]

Hence\[
\mathcal{S}\left(t, \begin{pmatrix} m_0 \\ \tilde{m}_0 \end{pmatrix}, \frac{1}{\tau} \mathfrak{p}\right) = e^{t \left(\mathcal{X} + \frac{\psi}{\tau} \mathcal{Y}\right)} \begin{pmatrix} m_0 \\ \tilde{m}_0 \end{pmatrix}.
\]

By Proposition \ref{prop:saturating_space_Phi}, we have \(\psi \in \mathcal{V}_0^\Phi \subset L^\infty(\mathbb{R}^n)\). Therefore, by \eqref{eqn:limit_order_1}, there exists \(\tau \in [0, T)\) such that
\[
\left\| \mathcal{S}\left(\tau, \begin{pmatrix}
	m_0 \\
	\tilde{m}_0
\end{pmatrix}, \frac{1}{\tau} \mathfrak{p} \right) - \begin{pmatrix}
	m_0 \\
	\tilde{m}_0 + \psi m_0
\end{pmatrix} \right\|_{H^1_\omega(\mathbb{R}^d) \times L^2_\omega(\mathbb{R}^d)} < \varepsilon,
\]
which establishes \((\mathfrak{I}_0)\).

$(\mathfrak{I}_n) \Rightarrow (\mathfrak{I}_{n+1})$: Assume that $(\mathfrak{I}_n)$ holds. We now prove $(\mathfrak{I}_{n+1})$ by induction.

Given any \(\psi \in \mathcal{V}_{n+1}^\Phi\), by the definition of \(\mathcal{V}_{n+1}^\Phi\), there exist \(\psi_0, \dots, \psi_N \in \mathcal{V}_n^\Phi\) such that
\[
\psi = \psi_0 - \sum_{i=1}^N \psi_i^2.
\]

To prove $(\mathfrak{I}_{n+1})$, which amounts to approximating the simulation of \(e^{\psi \mathcal{Y}}\), and noting that \(\mathcal{Y}^2 = 0\), it suffices to approximate
\[
e^{\psi_0 \mathcal{Y}} \cdot \prod_{i=1}^N e^{-\psi_i^2 \mathcal{Y}}.
\]

Since \(e^{\psi_0 \mathcal{Y}}\) is approximable by the induction hypothesis $(\mathfrak{I}_n)$, it remains to show how to approximate the simulation of each \(e^{-\psi_i^2 \mathcal{Y}}\).

For each \(\psi_i\), by Proposition \ref{prop:saturating_space_Phi}, we have \(\psi_i \in \mathcal{V}_n^\Phi \subset W^{1,\infty}(\mathbb{R}^d)\). Thus, by \eqref{eqn:limit_order_2}, there exists \(\mathfrak{b} < T/3\) such that
\begin{equation}\label{eqn:psi_square_1}
	\left\| e^{-\mathfrak{b}^{-1/2} \psi_i \mathcal{Y}} e^{\mathfrak{b} \mathcal{X}} e^{\mathfrak{b}^{-1/2} \psi_i \mathcal{Y}} 
	\begin{pmatrix}
		m_0 \\[6pt]
		\tilde{m}_0
	\end{pmatrix}
	- e^{-\psi_i^2 \mathcal{Y}} 
	\begin{pmatrix}
		m_0 \\[6pt]
		\tilde{m}_0
	\end{pmatrix} \right\|_{H^1_\omega(\mathbb{R}^d) \times L^2_\omega(\mathbb{R}^d)} < \varepsilon/4.
\end{equation}

Here, the operator 
\[
e^{-\mathfrak{b}^{-1/2} \psi_i \mathcal{Y}} \, e^{\mathfrak{b} \mathcal{X}} \, e^{\mathfrak{b}^{-1/2} \psi_i \mathcal{Y}}
\]
can be realized by concatenating three segments of control: the first and third correspond to \(e^{\pm \mathfrak{b}^{-1/2} \psi_i \mathcal{Y}}\), which are approximable by the induction hypothesis \((\mathfrak{I}_n)\) via \eqref{eqn:control_for_e_phi_X}, while the middle segment \(e^{\mathfrak{b} \mathcal{X}}\) corresponds to free evolution.

We then have that for given \(\eta, \varepsilon > 0\), there exist \(\mathfrak{a}, \mathfrak{c} < T/3\) and piecewise constant controls \(\mathfrak{p}^{\mathfrak{a}, \eta} : [0, \mathfrak{a}] \to \mathbb{R}^{2d+1}\) and \(\mathfrak{p}^{\mathfrak{c}, \varepsilon} : [0, \mathfrak{c}] \to \mathbb{R}^{2d+1}\) such that

\begin{equation}\label{eqn:psi_square_2}
	\left\|
	\mathcal{S}\left(\mathfrak{a}, 
	\begin{pmatrix}
		m_0 \\[6pt]
		\tilde{m}_0
	\end{pmatrix}, 
	\mathfrak{p}^{\mathfrak{a}, \eta}\right) 
	- e^{\mathfrak{b}^{-1/2} \psi_i \mathcal{Y}} 
	\begin{pmatrix}
		m_0 \\[6pt]
		\tilde{m}_0
	\end{pmatrix}
	\right\|_{H^1_\omega(\mathbb{R}^d) \times L^2_\omega(\mathbb{R}^d)} < \eta,
\end{equation}
and
\begin{equation}\label{eqn:psi_square_3}
	\left\|
	\mathcal{S}\left(\mathfrak{c}, 
	e^{\mathfrak{b} \mathcal{X}} e^{\mathfrak{b}^{-1/2} \psi_i \mathcal{Y}} 
	\begin{pmatrix}
		m_0 \\[6pt]
		\tilde{m}_0
	\end{pmatrix}, 
	\mathfrak{p}^{\mathfrak{c}, \varepsilon}\right)
	- e^{-\mathfrak{b}^{-1/2} \psi_i \mathcal{Y}} e^{\mathfrak{b} \mathcal{X}} e^{\mathfrak{b}^{-1/2} \psi_i \mathcal{Y}} 
	\begin{pmatrix}
		m_0 \\[6pt]
		\tilde{m}_0
	\end{pmatrix}
	\right\|_{H^1_\omega(\mathbb{R}^d) \times L^2_\omega(\mathbb{R}^d)} < \frac{\varepsilon}{4}.
\end{equation}

Then, if we denote by \(\mathbf{0}|_{[0,\mathfrak{b}]} = (0, \dots, 0)\) as the zero mapping 
\[
[0, \mathfrak{b}] \to \mathbb{R}^{2d+1}
\]
which corresponds to the free evolution \(e^{\mathfrak{b} \mathcal{X}}\), then using inequalities \eqref{eqn:psi_square_1} and \eqref{eqn:psi_square_3}, we deduce
\begin{align*}
	&\left\|
	\mathcal{S}\bigl(\mathfrak{a} + \mathfrak{b} + \mathfrak{c}, 
	\begin{pmatrix}
		m_0 \\[6pt]
		\tilde{m}_0
	\end{pmatrix}, 
	\mathfrak{p}^{\mathfrak{c}, \varepsilon} \diamond \mathbf{0}|_{[0,\mathfrak{b}]} \diamond \mathfrak{p}^{\mathfrak{a}, \eta}\bigr)
	- e^{-\psi_i^2 \mathcal{Y}} 
	\begin{pmatrix}
		m_0 \\[6pt]
		\tilde{m}_0
	\end{pmatrix}
	\right\|_{H^1_\omega(\mathbb{R}^d) \times L^2_\omega(\mathbb{R}^d)} \\
	&< \frac{\varepsilon}{2} 
	+ \left\|
	\mathcal{S}\bigl(\mathfrak{a} + \mathfrak{b} + \mathfrak{c}, 
	\begin{pmatrix}
		m_0 \\[6pt]
		\tilde{m}_0
	\end{pmatrix}, 
	\mathfrak{p}^{\mathfrak{c}, \varepsilon} \diamond \mathbf{0}|_{[0,\mathfrak{b}]} \diamond \mathfrak{p}^{\mathfrak{a}, \eta}\bigr)
	- \mathcal{S}\bigl(\mathfrak{c}, 
	e^{\mathfrak{b} \mathcal{X}} e^{\mathfrak{b}^{-1/2} \psi_i \mathcal{Y}} 
	\begin{pmatrix}
		m_0 \\[6pt]
		\tilde{m}_0
	\end{pmatrix}, 
	\mathfrak{p}^{\mathfrak{c}, \varepsilon}\bigr)
	\right\|_{H^1_\omega(\mathbb{R}^d) \times L^2_\omega(\mathbb{R}^d)}.
\end{align*}

Using inequality \eqref{eqn:Lipschitz_continuity} and inequality \eqref{eqn:psi_square_2}, we deduce
\begin{align*}
	&\left\|
	\mathcal{S}\left(\mathfrak{a} + \mathfrak{b} + \mathfrak{c}, 
	\begin{pmatrix}
		m_0 \\
		\tilde{m}_0
	\end{pmatrix}, 
	\mathfrak{p}^{\mathfrak{c}, \varepsilon} \diamond \mathbf{0}|_{[0, \mathfrak{b}]} \diamond \mathfrak{p}^{\mathfrak{a}, \eta} \right) 
	- e^{-\psi_i^2 \mathcal{Y}} 
	\begin{pmatrix}
		m_0 \\
		\tilde{m}_0
	\end{pmatrix}
	\right\|_{H^1_\omega(\mathbb{R}^d) \times L^2_\omega(\mathbb{R}^d)} \\
	&\le \frac{\varepsilon}{2} + \mathcal{C}_1(\mathfrak{p}^{\mathfrak{c}, \varepsilon}, \mathfrak{c}) 
	\left\|
	\mathcal{S}\left(\mathfrak{a} + \mathfrak{b}, 
	\begin{pmatrix}
		m_0 \\
		\tilde{m}_0
	\end{pmatrix}, 
	\mathbf{0}|_{[0, \mathfrak{b}]} \diamond \mathfrak{p}^{\mathfrak{a}, \eta} \right) 
	- e^{\mathfrak{b} \mathcal{X}} e^{\mathfrak{b}^{-1/2} \psi_i \mathcal{Y}} 
	\begin{pmatrix}
		m_0 \\
		\tilde{m}_0
	\end{pmatrix}
	\right\|_{H^1_\omega(\mathbb{R}^d) \times L^2_\omega(\mathbb{R}^d)} \\
	&\le \frac{\varepsilon}{2} + \mathcal{C}_1(\mathfrak{p}^{\mathfrak{c}, \varepsilon}, \mathfrak{c}) 
	\mathcal{C}_2(\mathfrak{b}) 
	\left\|
	\mathcal{S}\left(\mathfrak{a}, 
	\begin{pmatrix}
		m_0 \\
		\tilde{m}_0
	\end{pmatrix}, 
	\mathfrak{p}^{\mathfrak{a}, \eta} \right) 
	- e^{\mathfrak{b}^{-1/2} \psi_i \mathcal{Y}} 
	\begin{pmatrix}
		m_0 \\
		\tilde{m}_0
	\end{pmatrix}
	\right\|_{H^1_\omega(\mathbb{R}^d) \times L^2_\omega(\mathbb{R}^d)} \\
	&\le \frac{\varepsilon}{2} + \mathcal{C}_1(\mathfrak{p}^{\mathfrak{c}, \varepsilon}, \mathfrak{c}) 
	\mathcal{C}_2(\mathfrak{b}) \, \eta.
\end{align*}

Since \(\mathcal{C}_1, \mathcal{C}_2\) are independent of \(\eta\), there exists \(\eta > 0\) sufficiently small such that \(\varepsilon/2 + \mathcal{C}_1 \mathcal{C}_2 \eta < \varepsilon\). Then we have constructed a piecewise constant control, given by 
\[
\mathfrak{p}^{\mathfrak{c}, \varepsilon} \diamond \mathbf{0}|_{[0,\mathfrak{b}]} \diamond \mathfrak{p}^{\mathfrak{a}, \eta},
\]
that approximates the simulation of \(e^{-\psi_i^2 \mathcal{Y}}\).

We can then iteratively show that the simulation of 
\[
\prod_{i=1}^N e^{-\psi_i^2 \mathcal{Y}}
\]
is also approximable in arbitrarily small time. Then, by the induction hypothesis \((\mathfrak{I}_n)\), we know that there exists a piecewise constant control that steers the state
\[
\prod_{i=1}^N e^{-\psi_i^2 \mathcal{Y}}
\begin{pmatrix}
	m_0\\[3pt]
	\tilde{m}_0
\end{pmatrix}
=
\begin{pmatrix}
	m_0\\[3pt]
	\tilde{m}_0 - \sum_{i=1}^N \psi_i^2 m_0
\end{pmatrix}
\]
arbitrarily close to
\[
e^{\psi_0 \mathcal{Y}} \prod_{i=1}^N e^{-\psi_i^2 \mathcal{Y}}
\begin{pmatrix}
	m_0\\[3pt]
	\tilde{m}_0
\end{pmatrix}
=
\begin{pmatrix}
	m_0\\[3pt]
	\tilde{m}_0 + \psi_0 m_0 - \sum_{i=1}^N \psi_i^2 m_0
\end{pmatrix}
=
\begin{pmatrix}
	m_0\\[3pt]
	\tilde{m}_0 + \psi m_0
\end{pmatrix},
\]
which implies \((\mathfrak{I}_{n+1})\).

Thus, by induction, we conclude that $(\mathfrak{I}_n)$ holds for all $n \in \mathbb{N}$.

\textbf{Step 3. Show that $m_0\psi$ with $\psi\in\mathcal{V}_\infty^\Phi$ can approximate $m_0\psi_\eta$ in $L^2_\omega(\mathbb{R}^d)$.}

Since we have $\tilde{m}_0, v \in L^p_{\text{loc}}$ for a fixed $p > 2$, it follows that
\[
\psi_\eta = \frac{v - \tilde{m}_0}{m_0} \chi_{\mathcal{B} \setminus N_\eta(m_0)} \in L^p_\omega(\mathbb{R}^d).
\]

Thus, by \eqref{saturatinglem_Phi:enm3} in Proposition \ref{prop:saturating_space_Phi}, there exists a sequence $\{\psi_n\}_{n=1}^\infty \in \mathcal{V}_\infty^\Phi$ such that $\psi_n \to \psi_\eta$ in $L^p_\omega(\mathbb{R}^d)$. Since $m_0$ is a finite linear combination of eigenfunctions, it is a finite-degree polynomial, and hence $m_0 \in L^q_\omega(\mathbb{R}^d)$ for any $1 \le q < \infty$. Choosing $q$ such that $\frac{1}{p} + \frac{1}{q} = \frac{1}{2}$, we obtain
\[
\left\| \psi_\eta m_0 - \psi_n m_0 \right\|_{L^2_\omega(\mathbb{R}^d)} \le \left\| m_0 \right\|_{L^q_\omega(\mathbb{R}^d)} \left\| \psi_\eta - \psi_n \right\|_{L^p_\omega(\mathbb{R}^d)} \to 0.
\]

Thus for a given $\varepsilon>0$, there will exist $\psi_N\in\mathcal{V}_\infty^\Phi$ such that
\begin{equation}\label{eqn:error_mphi}
	\left\| \psi_\eta m_0-\psi_N m_0 \right\|_{L^2_\omega(\mathbb{R}^d)}<\varepsilon/4.
\end{equation}

Since $\mathcal{V}_\infty^\Phi = \bigcup_{n=0}^\infty \mathcal{V}_n^\Phi$, the function $\psi_N$ lies in some $\mathcal{V}_n^\Phi$. By applying $(\mathfrak{I}_n)$, for the given $\varepsilon, T$, there exist $\tau \in [0, T)$ and a piecewise constant function $\mathfrak{p} : [0, \tau] \to \mathbb{R}^{2d+1}$ such that the solution $m$ of \eqref{eqn:ControlPDE_linear} under the control law specified by \eqref{eqn:formula_theta}, \eqref{eqn:set_theta} with the initial state $(m_0, \tilde{m}_0)$ satisfies
\begin{equation}\label{eqn:error_sol_mphi}
	\left\| \left(m(\cdot, \tau), \partial_t m(\cdot, \tau)\right) - (m_0, \tilde{m}_0 + \psi_N m_0) \right\|_{H^1_\omega(\mathbb{R}^d) \times L^2_\omega(\mathbb{R}^d)} < \varepsilon / 4.
\end{equation}

Proposition \ref{prop:control_velocity} then follows from \eqref{eqn:error_v_mphi}, \eqref{eqn:error_mphi}, and \eqref{eqn:error_sol_mphi}.
\end{proof}

\subsection{Proof of  Theorem \ref{thm:control_for_nonlinear_wave}}
\begin{prop}\label{prop:control_for_linear_wave_short_time}
For any initial state \((m_0, \tilde{m}_0) \in H^1_\omega(\mathbb{R}^d) \times \bigl(L^2_\omega(\mathbb{R}^d) \cap L^p_{\mathrm{loc}}(\mathbb{R}^d)\bigr)\) for some \(p > 2\), satisfying either Assumption~\ref{asp:finitecombination} or Assumption~\ref{asp:finitecombination2}, and for any target state \((m_T, \tilde{m}_T) \in H^1_\omega(\mathbb{R}^d) \times L^2_\omega(\mathbb{R}^d)\), and any \(\varepsilon, T > 0\), there exist \(\tau \in [0, T)\) and a piecewise constant control \(\mathfrak{p} : [0, \tau] \to \mathbb{R}^{2d+1}\) such that the solution \(m\) of \eqref{eqn:ControlPDE_linear}, governed by the control law specified in \eqref{eqn:formula_theta} and \eqref{eqn:set_theta} with initial state \((m_0, \tilde{m}_0)\), satisfies
\[
\left\| \left( m(\cdot, \tau), \partial_t m(\cdot, \tau) \right) - (m_T, \tilde{m}_T) \right\|_{H^1_\omega(\mathbb{R}^d) \times L^2_\omega(\mathbb{R}^d)} < \varepsilon.
\]

\end{prop}
\begin{proof}

As a first step, we consider Assumption~\ref{asp:finitecombination} and assume that both \(m_T\) and \(\tilde{m}_T\) are nonzero finite linear combinations of eigenfunctions \(\psi_{\mathbf{n}}\). Our strategy proceeds as follows:

By Proposition~\ref{prop:control_displacement}, we can find an arbitrarily small time \(\mathfrak{b} > 0\), such that there exists an initial velocity \(v_0 \in L^2_\omega(\mathbb{R}^d)\) such that the pair \((m_0, v_0)\), under the free evolution \(e^{t\mathcal{X}}\), evolves to \((m_T, \partial_t m(\mathfrak{b}))\), which we denote by \((m_T, v_1)\). Then, by Proposition~\ref{prop:control_velocity}, there exist two controls—each defined over arbitrarily small time intervals of lengths \(\mathfrak{a}\) and \(\mathfrak{c}\)—such that the first steers the system from \((m_0, \tilde{m}_0)\) arbitrarily close to \((m_0, v_0)\), and the second steers the system from \((m_T, v_1)\) arbitrarily close to \((m_T, \tilde{m}_T)\). The controllability result then follows by concatenating these constructions and choosing \(\mathfrak{a} + \mathfrak{b} + \mathfrak{c} < T\).

By Remark~\ref{rmk:finite_combination}, we know that \(v_0, v_1 \in L^2_\omega(\mathbb{R}^d)\) are both finite linear combinations of eigenfunctions \(\psi_{\mathbf{n}}\), which are finite-degree polynomials. The same holds for \(\tilde{m}_T\). This implies that \(v_0, v_1, \tilde{m}_T \in L^p_{\text{loc}}(\mathbb{R}^d)\). Therefore, we can first apply Proposition~\ref{prop:control_velocity}: for any time interval of length \(T/3\), and for any \(\eta, \varepsilon > 0\), there exist \(\mathfrak{a}, \mathfrak{c} < T/3\) and piecewise constant controls \(\mathfrak{p}^{\mathfrak{a}, \eta} : [0, \mathfrak{a}] \to \mathbb{R}^{2d+1}\) and \(\mathfrak{p}^{\mathfrak{c}, \varepsilon} : [0, \mathfrak{c}] \to \mathbb{R}^{2d+1}\) such that
\begin{equation}\label{eqn:control_error_alpha}
	\left\|
	\mathcal{S}\left(\mathfrak{a}, \begin{pmatrix}
		m_0 \\
		\tilde{m}_0
	\end{pmatrix}, \mathfrak{p}^{\mathfrak{a}, \eta}\right) -
	\begin{pmatrix}
		m_0 \\
		v_0
	\end{pmatrix}
	\right\|_{H^1_\omega(\mathbb{R}^d) \times L^2_\omega(\mathbb{R}^d)} < \eta,
\end{equation}
\begin{equation}\label{eqn:control_error_gamma}
	\left\|
	\mathcal{S}\left(\mathfrak{c}, \begin{pmatrix}
		m_T \\
		v_1
	\end{pmatrix}, \mathfrak{p}^{\mathfrak{c}, \varepsilon}\right) -
	\begin{pmatrix}
		m_T \\
		\tilde{m}_T
	\end{pmatrix}
	\right\|_{H^1_\omega(\mathbb{R}^d) \times L^2_\omega(\mathbb{R}^d)} < \varepsilon/2.
\end{equation}

Then using inequality \eqref{eqn:control_error_gamma}, we deduce that
\begin{align*}
	&\left\|
	\mathcal{S}\bigl(\mathfrak{a}+\mathfrak{b}+\mathfrak{c}, \begin{pmatrix}
		m_0 \\
		\tilde{m}_0
	\end{pmatrix}, \mathfrak{p}^{\mathfrak{c}, \varepsilon} \diamond \mathbf{0}|_{[0,\mathfrak{b}]} \diamond \mathfrak{p}^{\mathfrak{a}, \eta}\bigr) - \begin{pmatrix}
		m_T \\
		\tilde{m}_T
	\end{pmatrix}
	\right\|_{H^1_\omega(\mathbb{R}^d) \times L^2_\omega(\mathbb{R}^d)} \\
	<& \frac{\varepsilon}{2} + \left\|
	\mathcal{S}\bigl(\mathfrak{a}+\mathfrak{b}+\mathfrak{c}, \begin{pmatrix}
		m_0 \\
		\tilde{m}_0
	\end{pmatrix}, \mathfrak{p}^{\mathfrak{c}, \varepsilon} \diamond \mathbf{0}|_{[0,\mathfrak{b}]} \diamond \mathfrak{p}^{\mathfrak{a}, \eta}\bigr) - \mathcal{S}\bigl(\mathfrak{c}, \begin{pmatrix}
		m_T \\
		v_1
	\end{pmatrix}, \mathfrak{p}^{\mathfrak{c}, \varepsilon}\bigr)
	\right\|_{H^1_\omega(\mathbb{R}^d) \times L^2_\omega(\mathbb{R}^d)}.
\end{align*}

Using inequality \eqref{eqn:Lipschitz_continuity}, and inequality \eqref{eqn:control_error_alpha}, we deduce
\begin{align*}
	& \left\|
	\mathcal{S}\bigl(\mathfrak{a}+\mathfrak{b}+\mathfrak{c}, \begin{pmatrix}
		m_0 \\
		\tilde{m}_0
	\end{pmatrix}, \mathfrak{p}^{\mathfrak{c}, \varepsilon} \diamond \mathbf{0}|_{[0,\mathfrak{b}]} \diamond \mathfrak{p}^{\mathfrak{a}, \eta}\bigr) - \begin{pmatrix}
		m_T \\
		\tilde{m}_T
	\end{pmatrix}
	\right\|_{H^1_\omega(\mathbb{R}^d) \times L^2_\omega(\mathbb{R}^d)} \\
	\le\; & \frac{\varepsilon}{2} + \mathcal C_1(\mathfrak{p}^{\mathfrak{c}, \varepsilon}, \mathfrak{c}) \left\|
	\mathcal{S}\bigl(\mathfrak{a}+\mathfrak{b}, \begin{pmatrix}
		m_0 \\
		\tilde{m}_0
	\end{pmatrix}, \mathbf{0}|_{[0,\mathfrak{b}]} \diamond \mathfrak{p}^{\mathfrak{a}, \eta}\bigr) - \begin{pmatrix}
		m_T \\
		v_1
	\end{pmatrix}
	\right\|_{H^1_\omega(\mathbb{R}^d) \times L^2_\omega(\mathbb{R}^d)} \\
	= \; & \frac{\varepsilon}{2} + \mathcal C_1(\mathfrak{p}^{\mathfrak{c}, \varepsilon}, \mathfrak{c}) \left\|
	e^{\mathfrak{b} \mathcal{X}} \mathcal{S}(\mathfrak{a}, \begin{pmatrix}
		m_0 \\
		\tilde{m}_0
	\end{pmatrix}, \mathfrak{p}^{\mathfrak{a}, \eta}) - e^{\mathfrak{b} \mathcal{X}} \begin{pmatrix}
		m_0 \\
		v_0
	\end{pmatrix}
	\right\|_{H^1_\omega(\mathbb{R}^d) \times L^2_\omega(\mathbb{R}^d)} \\
	\le\; & \frac{\varepsilon}{2} + \mathcal C_1(\mathfrak{p}^{\mathfrak{c}, \varepsilon}, \mathfrak{c}) \mathcal C_2(\mathfrak{b}) \eta.
\end{align*}

Since $\mathcal{C}_1, \mathcal{C}_2$ are independent of $\eta$, there exists $\eta$ small enough such that 
\[
\frac{\varepsilon}{2} + \mathcal{C}_1 \mathcal{C}_2 \eta < \varepsilon,
\]
which shows that we can steer the state $(m_0, \tilde{m}_0)$ arbitrarily close to $(m_T, \tilde{m}_T)$ in a time 
\[
\tau := \mathfrak{a} + \mathfrak{b} + \mathfrak{c} < T,
\]
under the assumptions that $m_0 \neq 0$ and that $m_T, \tilde{m}_T$ are both nonzero finite linear combinations of eigenfunctions $\psi_{\mathbf{n}}$.

Since the set of states $(m_T, \tilde{m}_T)$ with $m_T, \tilde{m}_T$ being nonzero finite linear combinations of $\psi_{\mathbf{n}}$ is dense in 
$H^1_\omega(\mathbb{R}^d) \times L^2_\omega(\mathbb{R}^d)$, this implies approximate controllability for general target states. Indeed, we can approximate any target $(m_T, \tilde{m}_T)$ by nonzero finite linear combinations of eigenfunctions, construct a control for the approximate target, and ensure the final state remains within a prescribed $\varepsilon$ of the true target. This proves the Proposition under Assumption~\ref{asp:finitecombination}.

For Assumption~\ref{asp:finitecombination2}, we first take a sufficiently small time $\mathfrak{a}_0 > 0$ and apply the free evolution. By formula~\eqref{eqn:linear_displacement_evolution}, we know that $m(\mathfrak{a}_0) \neq 0$, and $m(\mathfrak{a}_0)$ is a finite linear combination of eigenfunctions. Hence, the problem reduces to the setting of Assumption~\ref{asp:finitecombination}, and the result follows.
\end{proof}

The difference between Proposition \ref{prop:control_for_linear_wave} and Proposition \ref{prop:control_for_linear_wave_short_time} is that, for any given time \(T > 0\), the latter only guarantees controllability on some short subinterval within \([0, T)\). To prove Proposition \ref{prop:control_for_linear_wave}, our strategy is as follows:

Observe that \((m(t), \partial_t m(t)) = (1, 0)\) is a stationary solution of \eqref{eqn:ControlPDE_linear} under zero control:
\begin{equation}\label{eqn:stationary_01}
e^{t \mathcal{X}} \begin{pmatrix} 1 \\ 0 \end{pmatrix} = \begin{pmatrix} 1 \\ 0 \end{pmatrix} \quad \text{for all } t \in \mathbb{R}.
\end{equation}

We can first steer the initial state to \((1, 0)\) in a short time, then let the system remain at \((1, 0)\) for an arbitrary duration, and finally steer \((1, 0)\) to the final state in a short time. By choosing the duration of the stationary phase appropriately, we are able to construct a control supported on the entire time interval \([0, T]\).

\begin{proof}[Proof of Proposition \ref{prop:control_for_linear_wave}]

It follows from Proposition \ref{prop:control_for_linear_wave_short_time} that, for any \(\eta, \varepsilon > 0\), there exist \(\mathfrak{a}, \mathfrak{c} < T/2\) and piecewise constant controls 
\[
\mathfrak{p}^{\mathfrak{a}, \eta} : [0, \mathfrak{a}] \to \mathbb{R}^{2d+1}, \quad \mathfrak{p}^{\mathfrak{c}, \varepsilon} : [0, \mathfrak{c}] \to \mathbb{R}^{2d+1}
\]
such that
\begin{equation}\label{eqn:control_error_alpha_01}
	\left\|
	\mathcal{S}\left(\mathfrak{a}, \begin{pmatrix} m_0 \\ \tilde{m}_0 \end{pmatrix}, \mathfrak{p}^{\mathfrak{a}, \eta}\right)
	- \begin{pmatrix} 1 \\ 0 \end{pmatrix}
	\right\|_{H^1_\omega(\mathbb{R}^d) \times L^2_\omega(\mathbb{R}^d)} < \eta,
\end{equation}
\begin{equation}\label{eqn:control_error_gamma_01}
	\left\|
	\mathcal{S}\left(\mathfrak{c}, \begin{pmatrix} 1 \\ 0 \end{pmatrix}, \mathfrak{p}^{\mathfrak{c}, \varepsilon}\right)
	- \begin{pmatrix} m_T \\ \tilde{m}_T \end{pmatrix}
	\right\|_{H^1_\omega(\mathbb{R}^d) \times L^2_\omega(\mathbb{R}^d)} < \frac{\varepsilon}{2}.
\end{equation}

Set \(\mathfrak{b} := T - \mathfrak{a} - \mathfrak{c}\). Using inequality \eqref{eqn:control_error_gamma_01} and the Lipschitz continuity inequality \eqref{eqn:Lipschitz_continuity}, we deduce
\begin{align*}
	&\left\|
	\mathcal{S}\bigl(T, \begin{pmatrix} m_0 \\ \tilde{m}_0 \end{pmatrix}, \mathfrak{p}^{\mathfrak{c}, \varepsilon} \diamond \mathbf{0}|_{[0,\mathfrak{b}]} \diamond \mathfrak{p}^{\mathfrak{a}, \eta}\bigr) 
	- \begin{pmatrix} m_T \\ \tilde{m}_T \end{pmatrix}
	\right\|_{H^1_\omega(\mathbb{R}^d) \times L^2_\omega(\mathbb{R}^d)} \\
	<& \frac{\varepsilon}{2} + \left\|
	\mathcal{S}\bigl(\mathfrak{a} + \mathfrak{b} + \mathfrak{c}, \begin{pmatrix} m_0 \\ \tilde{m}_0 \end{pmatrix}, \mathfrak{p}^{\mathfrak{c}, \varepsilon} \diamond \mathbf{0}|_{[0,\mathfrak{b}]} \diamond \mathfrak{p}^{\mathfrak{a}, \eta}\bigr) 
	- \mathcal{S}\bigl(\mathfrak{c}, \begin{pmatrix} 1 \\ 0 \end{pmatrix}, \mathfrak{p}^{\mathfrak{c}, \varepsilon}\bigr)
	\right\|_{H^1_\omega(\mathbb{R}^d) \times L^2_\omega(\mathbb{R}^d)} \\
	\le& \frac{\varepsilon}{2} + \mathcal{C}_1(\mathfrak{p}^{\mathfrak{c}, \varepsilon}, \mathfrak{c}) \left\|
	\mathcal{S}\bigl(\mathfrak{a} + \mathfrak{b}, \begin{pmatrix} m_0 \\ \tilde{m}_0 \end{pmatrix}, \mathbf{0}|_{[0,\mathfrak{b}]} \diamond \mathfrak{p}^{\mathfrak{a}, \eta}\bigr) - \begin{pmatrix} 1 \\ 0 \end{pmatrix}
	\right\|_{H^1_\omega(\mathbb{R}^d) \times L^2_\omega(\mathbb{R}^d)}.
\end{align*}

By \eqref{eqn:stationary_01}, we have that
\[
\left\|
\mathcal{S}\bigl(\mathfrak{a} + \mathfrak{b}, \begin{pmatrix}
	m_0 \\
	\tilde{m}_0
\end{pmatrix}, \mathbf{0}|_{[0,\mathfrak{b}]} \diamond \mathfrak{p}^{\mathfrak{a}, \eta}\bigr) 
- \begin{pmatrix}
	1 \\
	0
\end{pmatrix}
\right\|_{H^1_\omega(\mathbb{R}^d) \times L^2_\omega(\mathbb{R}^d)}
= \left\|
e^{\mathfrak{b}\mathcal{X}} \mathcal{S}\bigl(\mathfrak{a}, \begin{pmatrix}
	m_0 \\
	\tilde{m}_0
\end{pmatrix}, \mathfrak{p}^{\mathfrak{a}, \eta}\bigr) 
- e^{\mathfrak{b}\mathcal{X}} \begin{pmatrix}
	1 \\
	0
\end{pmatrix}
\right\|_{H^1_\omega(\mathbb{R}^d) \times L^2_\omega(\mathbb{R}^d)}.
\]

Thus it follows from \eqref{eqn:control_error_alpha_01} that
\begin{align*}
	&\left\|
	\mathcal{S}\bigl(T,\begin{pmatrix}
		m_0 \\
		\tilde{m}_0
	\end{pmatrix}, \mathfrak{p}^{\mathfrak{c}, \varepsilon} \diamond \mathbf{0}|_{[0,\mathfrak{b}]} \diamond \mathfrak{p}^{\mathfrak{a}, \eta}\bigr)
	- \begin{pmatrix}
		m_T \\
		\tilde{m}_T
	\end{pmatrix}
	\right\|_{H^1_\omega(\mathbb{R}^d) \times L^2_\omega(\mathbb{R}^d)} \\
	\le\,& \frac{\varepsilon}{2} + \mathcal{C}_1(\mathfrak{p}^{\mathfrak{c}, \varepsilon}, \mathfrak{c}) 
	\left\| e^{\mathfrak{b}\mathcal{X}} \mathcal{S}(\mathfrak{a}, \begin{pmatrix}
		m_0 \\
		\tilde{m}_0
	\end{pmatrix}, \mathfrak{p}^{\mathfrak{a}, \eta}) 
	- e^{\mathfrak{b}\mathcal{X}} \begin{pmatrix}
		1 \\
		0
	\end{pmatrix} \right\|_{H^1_\omega(\mathbb{R}^d) \times L^2_\omega(\mathbb{R}^d)} \\
	\le\,& \frac{\varepsilon}{2} + \mathcal{C}_1(\mathfrak{p}^{\mathfrak{c}, \varepsilon}, \mathfrak{c}) \mathcal{C}_2(\mathfrak{b})
	\left\| \mathcal{S}(\mathfrak{a}, \begin{pmatrix}
		m_0 \\
		\tilde{m}_0
	\end{pmatrix}, \mathfrak{p}^{\mathfrak{a}, \eta}) - \begin{pmatrix}
		1 \\
		0
	\end{pmatrix} \right\|_{H^1_\omega(\mathbb{R}^d) \times L^2_\omega(\mathbb{R}^d)} \\
	\le\,& \frac{\varepsilon}{2} + \mathcal{C}_1(\mathfrak{p}^{\mathfrak{c}, \varepsilon}, \mathfrak{c}) \mathcal{C}_2(\mathfrak{b}) \eta.
\end{align*}

Since $\mathcal{C}_1$ and $\mathcal{C}_2$ are independent of $\eta$, there exists $\eta$ sufficiently small such that 
\[
\frac{\varepsilon}{2} + \mathcal{C}_1 \mathcal{C}_2 \eta < \varepsilon.
\]

Thus, we have constructed a control that steers the system arbitrarily close to the desired state on $[0,T]$.
\end{proof}

\begin{proof}[Proof of Theorem \ref{thm:control_for_nonlinear_wave}]
For the given initial state $(m_0,\tilde{m}_0)$ and final state $(m_T,\tilde{m}_T)$, for any $\varepsilon,T>0$, by Proposition \ref{prop:control_for_linear_wave}, there exists $\mathfrak{p}^*:[0,T]\rightarrow \mathbb{R}^{2d+1}$ piecewise constant such that the solution $m^*$ of the following PDE
\begin{equation}
	\begin{cases}
		\displaystyle
		\frac{\partial^2}{\partial t^2}m^*(x,t)-\Delta m^*(x,t)+2x\cdot\nabla m^*(x,t)+\theta^*(x,t)m^*(x,t)=0,\\
		m^*(x,0)=m_0(x),\\
		\partial_t m^*(x,0)=\tilde{m}_0(x),\\
	\end{cases}
\end{equation}
with
\begin{equation}
	\theta^*(x,t)=\sum_{j=0}^{2d}\mathfrak{p}_j^*
	(t)\theta_j(x),
\end{equation}
satisfies
$$\left\|  \left( m^*(\cdot,T),\partial_t m^*(\cdot, T) \right)-(m_T,\tilde{m}_T) \right\|_{H^1_\omega(\mathbb{R}^d)\times L^2_\omega (\mathbb{R}^d)}<\varepsilon .$$

Since $\sigma$ is a Lipschitz function and $\sigma(0)=0$, thus $$\bar{\sigma}(x):=\begin{cases}
	\displaystyle \frac{\sigma(x)}{x}, & \text{if } x \ne 0, \\
	0, & \text{if } x = 0.
\end{cases}$$
is a.e.~well-defined, and $\bar{\sigma}\in L^\infty(\mathbb{R})$. Hence $\bar{\sigma}(m^*)=\bar{\sigma}\circ m^*\in L^\infty(\mathbb{R}^d\times[0,T])$ and we have that
$$\bar{\sigma}(m^*) m^*=\sigma(m^*).$$

Then    \begin{equation}
	\frac{\partial^2}{\partial t^2}m^*(x,t)-(\Delta -2x\cdot\nabla) m^*(x,t)+ \left( \theta^*(x,t)+\bar{\sigma}(m^*) \right)m^*(x,t)=\sigma(m^*),
\end{equation}

If we set 
$$\theta(x,t)=\theta^*(x,t)+\bar{\sigma}(m^*),$$
then $\theta\in L^\infty(\mathbb{R}^d\times [0,T])$ and $m(x,t):=m^*(x,t)$ would be the solution of the following equation
\begin{equation}
	\frac{\partial^2}{\partial t^2}m(x,t)-\Delta m(x,t)+2x\cdot\nabla m(x,t)+\theta(x,t)m(x,t)=\sigma(m),
\end{equation}
with initial state $(m(\cdot, 0),\tilde{m}(\cdot, 0))=(m_0,\tilde{m}_0)$, and the uniqueness is guaranteed by Proposition \ref{Mild}. Furthermore, we have that
$$\left\|  \left( m(\cdot,T),\partial_t m(\cdot, T) \right)-(m_T,\tilde{m}_T) \right\|_{H^1_\omega(\mathbb{R}^d)\times L^2_\omega (\mathbb{R}^d)}<\varepsilon,$$
which shows the approximate controllability of \eqref{eqn:ControlPDE_nonlinear}.
\end{proof}

\section{Conclusion}
Neural networks have played a crucial role in deep learning. In this work, we propose to view neural networks as discrete \emph{Neural Partial Differential Equations (Neural PDEs)}. Under this perspective, the supervised learning task reduces to optimizing the coefficients of the associated parabolic and hyperbolic PDEs. To the best of our knowledge, this approach has not yet been explored within the control and optimization theory literature, thus opening a novel research direction.

We emphasize that, depending on the choice of initial condition \( m_0 \) and target state \( m_T \), the minimum of the optimization problem \eqref{Para2}
may not be zero. In this work, we prove the existence of optimal parameters for the parabolic equation case. However, the uniqueness of minimizers cannot be guaranteed, as multiple minimizers may exist.

Furthermore, we establish for the first time a \emph{dual system} associated with the optimization problem
\eqref{Para1}-\eqref{Para2}. Designing numerical schemes for solving this system poses significant challenges, as existing methods for parabolic equations typically assume known coefficients, whereas in our problem the coefficients themselves are the optimization variables.

For the hyperbolic case, we prove an \emph{approximate controllability result} as expressed in
\eqref{Hyper1}-\eqref{Hyper2}.

In upcoming work, we will address the question of \emph{exact controllability}—that is, whether the minimum of \eqref{Para2} can be zero—and provide explicit solutions for the system
\eqref{Para1}-\eqref{Para2}
under exact controllability assumptions. We will also develop and analyze numerical schemes suitable for these problems, as well as investigate the exact control problem
\eqref{Hyper1}-\eqref{Hyper3}.
\bibliographystyle{amsplain}
\bibliography{ref}{}

\end{document}